\documentclass{amsart}

\usepackage{latexsym,amsfonts,amssymb,exscale,enumerate}
\usepackage{amsmath,amsthm,amscd}
\usepackage{hyperref}

\usepackage{pstricks}

\psset{linewidth=0.3pt,dimen=middle}
\psset{xunit=.70cm,yunit=0.70cm}
\psset{arrowsize=1pt 5,arrowlength=0.6,arrowinset=0.7}

\input xy
\usepackage[all,dvips]{xy}
\xyoption{line}
\xyoption{arrow}
\xyoption{color}
\SelectTips{cm}{}
\usepackage{graphicx}
\usepackage{color}
\newcommand{\onen}{{\mathbf 1}_{n}}
\newcommand{\onenn}[1]{{\mathbf 1}_{#1}}
\newcommand\rE{{\sf{E}}}
\newcommand\rF{{\sf{F}}}
\theoremstyle{definition}
\newtheorem{thm}{Theorem}[section]
\newtheorem{cor}[thm]{Corollary}

\newtheorem{lem}[thm]{Lemma}
\newtheorem{rem}[thm]{Remark}
\newtheorem{prop}[thm]{Proposition}
\newtheorem{defn}[thm]{Definition}
\newcommand{\Uq}{{\bf U}_q(\mathfrak{sl}_2)}
\newcommand{\Ucat}{\cal{U}}
\newcommand{\UcatD}{\dot{\cal{U}}}
\newcommand{\UA}{{_{\cal{A}}\dot{{\bf U}}}}
\newcommand{\maps}{\colon\thinspace}
\newcommand{\bins}[2]{
\left(
 \begin{array}{c}
 \scs #1 \\
 \scs #2 \\
 \end{array}
 \right)
}
\allowdisplaybreaks
\newcommand{\Hom}{{\rm Hom}}
\def\mf{\mathfrak}
\numberwithin{equation}{section}

\let\hat=\widehat
\let\tilde=\widetilde

\let\theta=\vartheta
\let\epsilon=\varepsilon

\usepackage{bbm}
\def\C{{\mathbbm C}}

\def\Z{{\mathbbm Z}}
\def\Q{{\mathbbm Q}}

\def\cal#1{\mathcal{#1}}%
\def\1{\mathbbm{1}}%
\def\tr{\mathrm{tr}}%
\def\nn{\notag}

\def\la{\langle}
\def\ra{\rangle}

\newcommand{\ccbub}[1]{
\xybox{%
 (-6,0)*{};
  (6,0)*{};
  (-4,0)*{}="t1";
  (4,0)*{}="t2";
  "t2";"t1" **\crv{(4,6) & (-4,6)};
   ?(1)*\dir{>};
  "t2";"t1" **\crv{(4,-6) & (-4,-6)};
   ?(.3)*\dir{}+(0,0)*{\bullet}+(0,-3)*{\scs {#1}};
}}

\newcommand{\cbub}[1]{
\xybox{%
 (-6,0)*{};
  (6,0)*{};
  (-4,0)*{}="t1";
  (4,0)*{}="t2";
  "t2";"t1" **\crv{(4,6) & (-4,6)};
    ?(.95)*\dir{<};
  "t2";"t1" **\crv{(4,-6) & (-4,-6)};
   ?(.3)*\dir{}+(0,0)*{\bullet}+(0,-3)*{\scs {#1}};
}}
\newcommand{\bbpef}{\xybox{%
  (-6,0)*{};
  (6,0)*{};
  (-4,0)*{}="t1";
  (4,0)*{}="t2";
  "t1";"t2" **\crv{(-4,-6) & (4,-6)}; ?(.15)*\dir{>} ?(.9)*\dir{>};
}}
\newcommand{\bbpfe}{\xybox{%
  (-6,0)*{};
  (6,0)*{};
  (-4,0)*{}="t1";
  (4,0)*{}="t2";
  "t2";"t1" **\crv{(4,-6) & (-4,-6)}; ?(.15)*\dir{>} ?(.9)*\dir{>};
}}

\newcommand{\bbcfe}[1]{\xybox{%
  (-6,0)*{};
  (6,0)*{};
  (-4,0)*{}="t1";
  (4,0)*{}="t2";
  "t1";"t2" **\crv{(-4,6) & (4,6)}; ?(.15)*\dir{>} ?(.9)*\dir{>}
  ?(.5)*\dir{}+(0,2)*{\scriptstyle{#1}};
}}
\newcommand{\bbcef}[1]{\xybox{%
  (-6,0)*{};
  (6,0)*{};
  (-4,0)*{}="t1";
  (4,0)*{}="t2";
  "t2";"t1" **\crv{(4,6) & (-4,6)}; ?(.15)*\dir{>}
  ?(.9)*\dir{>} ?(.5)*\dir{}+(0,2)*{\scriptstyle{#1}};
}}

\newcommand{\lowrru}[1]{\xybox{%
  (-8,0)*{};
  (8,0)*{};
  (-6,-18)*{};(6,-9)*{} **\crv{(-6,-13) & (6,-15)} ?(1)*\dir{>};
  (6,-9)*{};(6,0)*{}  **\dir{-} ?(.3)*\dir{ }+(2,0)*{\scs {\bf j}};
}}

\newcommand{\lowllu}[1]{\xybox{%
  (-8,0)*{};
  (8,0)*{};
  (6,-18)*{};(-6,-9)*{} **\crv{(6,-13) & (-6,-15)} ?(1)*\dir{>};
  (-6,-9)*{};(-6,0)*{}  **\dir{-} ?(.3)*\dir{ }+(-2,0)*{\scs {\bf j}};
}}

\newcommand{\bbdl}[1]{\xybox{%
  (2,0);(0,-8) **\crv{(2,-2)&(0,-6)}; ?(.5)*\dir{>}
}}
\newcommand{\bbdlu}[1]{\xybox{%
  (2,0);(0,-8) **\crv{(2,-2)&(0,-6)}; ?(.5)*\dir{<}
}}
\newcommand{\bbdr}[1]{\xybox{%
  (-2,0);(0,-8) **\crv{(-2,-2)&(0,-6)}; ?(.5)*\dir{>}
}}
\newcommand{\bbdru}[1]{\xybox{%
  (-2,0);(0,-8) **\crv{(-2,-2)&(0,-6)}; ?(.5)*\dir{<}
}}

\DeclareGraphicsExtensions{.pdf,.eps}
\usepackage{psfrag}
\usepackage{bbm}

\def\cal#1{\mathcal{#1}}

\def \Z {\mathbbm{Z}}

\def \Q {\mathbbm{Q}}
\def \E {\mathcal{E}}
\def \F {\mathcal{F}}
\def \U {\mathcal{U}}

\def \B {\mathcal{B}}
\def \C {\mathcal{C}}
\def \Tr{\operatorname{Tr}}

\def \Span{\operatorname{Span}}
\def \Ob{\operatorname{Ob}}
\def \Set{\mathbf{Set}}
\def \Cat{\mathbf{Cat}}

\def \HH{\operatorname{HH}}
\def \min {{\rm min}}
\def \Id {{\rm Id}}
\def \i {{\rm i}}

\newcommand{\xto}[1]{{\overset{#1}{\longrightarrow}}}

\newcommand{\bigb}[1]{
\begin{pspicture}(0,0)
 \rput(0,0){\psframebox[framearc=.5,fillstyle=solid]{\small $#1$}}
\end{pspicture}}

\newcommand\nc{\newcommand}
\nc\rnc{\renewcommand}
\nc\Kar{\operatorname{Kar}}
\nc\End{\operatorname{End}}
\nc\modQ {{\mathbb Q}}
\nc\modZ {{\mathbb Z}}
\nc\simeqto{\overset{\simeq}{\longrightarrow }}
\nc\modC {{\mathcal C}}
\nc\modD {{\mathcal D}}
\nc\K{\mathcal {K}}
\nc\CC{\mathbf{C}}
\nc\tP{\tilde{P}}
\nc\tl{\tilde{\lambda}}
\nc\tm{\tilde{\mu}}

\newcommand{\scs}{\scriptstyle}
\newcommand{\lbub}{\xybox{
  (-3,0)*{};(3,0)*{} **\crv{(-3,4) & (3,4)} ?(.0)*\dir{>};
  (3,0)*{};(-3,0)*{} **\crv{(3,-4) & (-3,-4)} ?(.1)*{*};
  (-5,-5)*{}; (5,5)*{};
}}
\newcommand{\rbub}{\xybox{
  (-3,0)*{};(3,0)*{} **\crv{(-3,4) & (3,4)} ?(.0)*\dir{<};
  (3,0)*{};(-3,0)*{} **\crv{(3,-4) & (-3,-4)} ?(.1)*{*};
  (-5,-5)*{}; (5,5)*{};
}}
\newcommand{\lbbub}{\xybox{
  (-3,0)*{};(3,0)*{} **\crv{(-3,4) & (3,4)} ?(.0)*\dir{>};
  (3,0)*{};(-3,0)*{} **\crv{(3,-4) & (-3,-4)} ?(.1)*{\bullet};
  (-5,-5)*{}; (5,5)*{};
}}
\newcommand{\rbbub}{\xybox{
  (-3,0)*{};(3,0)*{} **\crv{(-3,4) & (3,4)} ?(.0)*\dir{<};
  (3,0)*{};(-3,0)*{} **\crv{(3,-4) & (-3,-4)} ?(.1)*{\bullet};
  (-5,-5)*{}; (5,5)*{};
}}

\nc\Sym{\operatorname{Sym}}
\nc\one{\mathbf{1}}

\nc\ct{\underset{\cong}{\longrightarrow}}
\nc\calU{\mathcal{U}}
\nc\cU{\calU}
\nc\col{\colon\thinspace}
\nc\calA{\mathcal{A}}
\nc\Ab{\mathbf{Ab}}
\nc\Ko{K_0}
\nc\TrhorCC{\Tr^{\mathrm{hor}}(\CC)}
\nc\AdCat{\mathbf{AdCat}}
\nc\TrCC{\Tr(\CC)}
\nc\Udot{\dot{\mathcal{U}}}
\nc\diag{\mathrm{d}}
\nc\modU {\mathcal{U}}
\nc\bfU{\mathbf{U}}
\nc\dU{\dot{\mathbf U}}
\nc\dUZ{{_\modZ\dot{\mathbf U}}}
\nc\UZ{{_\modZ \mathbf U} }
\nc\fsl{\mathfrak{sl}}
\nc\Uaa{{\bf U} (\mathfrak{sl}_2\otimes \Q[t,t^{-1}])}
\nc\UZslt{{_\modZ\mathbf{U}} \LL}
\nc\UdZslt{{_\modZ\dot{\mathbf{U}}}\LL}
\nc\LL{(\fsl_2[t])}
\nc\UL{\mathbf U\LL}
\nc\UZL{\UZ\LL}
\nc\dUZL{\dUZ\LL}
\nc\dUL{\dU\LL}

\nc\on{\onen}

\newcommand{\elem}{e}
\newcommand{\idemp}{i}

\nc\tG{\tilde{G}} \nc\tE{\tilde{E}}

\nc\yto[1]{\underset{#1}{\to}}
\nc\Ear{\yto{E}}
\nc\dotU{\dot{\mathcal{U}}}

\begin{document}
%\title{Trace decategorification of categorified quantum groups}

\title{Trace decategorification of categorified quantum $\mathfrak{sl}_2$}
\date{April 7, 2014}
\author{Anna Beliakova}
\address{Universit\"at Z\"urich, Winterthurerstr. 190
CH-8057 Z\"urich, Switzerland}
\email{anna@math.uzh.ch}

\author{Kazuo Habiro}
\address{Research Institute for Mathematical Sciences, Kyoto University, Kyoto, 606-8502, Japan}
\email{habiro@kurims.kyoto-u.ac.jp}

\author{Aaron D.~Lauda}
\address{University of Southern California, Los Angeles, CA 90089, USA}
\email{lauda@usc.edu}

\author{Marko \v Zivkovi\' c}
\address{Universit\"at Z\"urich, Winterthurerstr. 190
CH-8057 Z\"urich, Switzerland}
\email{marko.zivkovic@math.uzh.ch}

\begin{abstract}
  The trace or the $0$th Hochschild--Mitchell homology of a linear
  category $\mathcal{C}$ may be regarded as a kind of decategorification of $\mathcal{C}$.
  We compute the traces of the two versions $\UcatD$ and $\U^*$ of
  categorified quantum $\mathfrak{sl}_2$ introduced by the third
  author.  The trace of $\Ucat$ is isomorphic to the split
  Grothendieck group $K_0(\UcatD)$, and the higher
  Hochschild--Mitchell homology of $\UcatD$ is zero.  The trace of
  $\U^*$ is isomorphic to the idempotented integral form of the
  current algebra $\bfU(\mathfrak{sl}_2[t])$.
\end{abstract}

\maketitle

\section{Introduction}

Categorification is a process of transforming a set-like structure
into a category-like structure, which is inverse to
decategorification, usually defined as a functor from category-like
structures to set-like structures.  For example, the functor from the
category of small categories to the category of sets,
mapping a category $\modC $ to the set $\Ob(\modC )/\cong$ of
isomorphism classes of objects in $\modC $, is considered to be a
decategorification.  In this example, the categorification problem is
to find, for a given set $S$, a nice category $\modC $ and a bijection
between $\Ob(\modC )/\cong$ and $S$.  The notion of niceness depends
on the circumstances, and categorification is not a well-defined
notion in general.  Usually, a decategorification functor maps
categories with structures, such as monoidal categories and
$2$-categories, to sets with structures, such as monoids and
categories.  The corresponding categorification problem for a monoid
$M$ is to find a nice monoidal category $\modC $ with the monoid
$\Ob(\modC )/\cong$ isomorphic to $M$.

Another example of decategorification functor is the Grothendieck
group $G_0$, which is a functor from small abelian categories to
abelian groups.  This functor is enhanced to a functor from small
abelian monoidal categories to associative, unital rings.
The categorification problem for a ring $R$ is to find a nice abelian
monoidal category $\modC $ with a ring isomorphism $G_0(\modC )\cong
R$.  For an additive category $\modC $, which is not necessarily
abelian, the usual Grothendieck group is not defined, but there is the
split Grothendieck group, denoted $\Ko(\modC )$,
which is generated by the isomorphism classes of objects in $\modC $,
with relations $[x\oplus y]=[x]+[y]$ for $x,y\in \Ob(\modC )$ (see
Section \ref{sec:trac-line-bicat-1} for details).

One can apply the split Grothendieck group $K_0$ to an additive
$2$-category $\CC$ to obtain a linear category $K_0(\CC)$.  (By a {\em
linear category} we mean a category enriched
over the category $\Ab$ of abelian groups.)  Here, the objects in
$K_0(\CC)$ are the objects in $\CC$, and for objects $x$ and $y$ the
hom-module $K_0(\CC)(x,y)$ is defined as the split Grothendieck
group $K_0(\CC(x,y))$ of the hom-category $\CC(x,y)$.

In \cite{Lau1}, the third author defined an additive $2$-category
$\dot\cU$, which categorifies the Beilinson--Lusztig--MacPherson
idempotented integral form $_{\mathcal{A}}\dot{\mathbf{U}}$ of
the quantum group $\bfU=\Uq$ \cite{BLM}.  That $\dot\cU$ categorifies
$_{\mathcal{A}}\dot{\mathbf{U}}$ means here that the split
Grothendieck group $K_0(\dot\cU)$ is isomorphic to
$_{\mathcal{A}}\dot{\mathbf{U}}$.  For an elementary introduction to the
 $2$-category $\dot\cU$ see \cite{Lau3}.

In this paper, we consider another decategorification functor on
linear categories, called the {\em trace}.
The trace of a linear category $\modC $ is defined to be the coend of
the Hom functor $\Hom_\modC\col\modC^{\mathrm{op}}\times\modC\to\Ab$,
i.e., the abelian group
\begin{gather*}
  \begin{split}
  \Tr(\modC)
  = \int^{x\in \modC }\Hom_\modC (x,x)
  =\left(\bigoplus_{x\in \Ob(\modC)}\Hom_\modC (x,x)\right)/\Span\{fg-gf\}.
  \end{split}
\end{gather*}
where the subgroup $\Span\{fg-gf\}$ is spanned by $fg-gf$ for
all pairs $f\col x\rightarrow y$, $g\col y\rightarrow x$ with $x,y\in
\Ob(\modC )$.
The trace $\Tr(\modC)$ is also known as the $0$th Hochschild--Mitchell
homology $\HH_0(\modC)$ of $\modC $.

Unlike the split Grothendieck group, the trace can be defined for any
linear category, not necessarily additive.  The trace of a linear
category $\C$ is naturally isomorphic to the trace of the additive
closure of $\C$.  Moreover, the trace $\Tr(\Kar(\C))$ of the Karoubi
envelope $\Kar(\C)$ is isomorphic to $\Tr(\C)$, see Section
\ref{sec:trac-line-bicat-1}.

Similarly to $K_0$, the trace $\Tr$ can be applied to linear
$2$-categories.  The trace $\Tr(\CC)$ of a linear $2$-category is a
linear category such that $\Ob(\Tr(\CC))=\Ob(\CC)$, and
\begin{gather*}
  \Tr(\CC)(x,y)=\Tr(\CC(x,y))
\end{gather*}
for $x,y\in\Ob(\C)$.

For an additive category $\modC$, the trace $\Tr(\modC)$ and the split
Grothendieck group $K_0(\modC)$ are related by the linear map
\begin{gather*}
  h_{\modC}\col K_0(\modC) \to \Tr(\modC)
\end{gather*}
which maps the equivalence class $[X]$ of $X\in\modC$ to the
equivalence class $[1_X]$ of the identity morphism $1_X\col X\to X$.

For an additive $2$-category $\CC$, we have a linear functor
\begin{gather*}
  h_{\CC}\col K_0(\CC) \to \Tr(\CC),
\end{gather*}
which maps each object to itself and maps morphisms by
\begin{gather*}
  h_{\CC(x,y)}\col K_0(\CC(x,y)) \to \Tr(\CC(x,y)).
\end{gather*}

In this paper we compute the traces of the additive $2$-categories
$\UcatD$ and $\U^*$, where $\U^*$ is another version of the
categorified quantum $\fsl_2$, introduced in \cite{Lau1}.  Here we
recall that the objects of $\UcatD$ (and, in fact, of $\U^*$) are the
integers, which are regarded as the elements of the weight lattice of
$\fsl_2$.  For more details about $\UcatD$, see Section
\ref{sec:2-categories-ucatd}.

The first main result of this paper is the following.

\begin{thm} \label{thm1}
  The linear functor
  \begin{gather*}
h_{\UcatD}\col K_0(\UcatD) \longrightarrow \Tr(\UcatD)=\HH_0(\UcatD)
  \end{gather*}
  is an isomorphism.  Hence $\Tr(\UcatD)$ is isomorphic to
  $_\calA\dot{\mathbf{U}}$.  Moreover, for $i>0$,
  and $n,m\in\mathbb{Z}$, the Hochschild--Mitchell homology
  group $\HH_i(\Udot(n,m))$ of $\Udot(n,m)$ is zero.
\end{thm}

Now we consider the trace of the  additive $2$-category $\U^*$.

The Lie algebra $\fsl_2[t]=\fsl_2\otimes\Q[t]$ is the non-negative part of the loop Lie algebra $\fsl_2[t,t^{-1}]$.  The universal enveloping algebra $\UL$ of $\fsl_2[t]$ is called the {\em current algebra} of $\fsl_2$.  Thus, $\UL$ is the $\Q$-algebra generated by $E_i,F_i$ and $H_i$ for $i\geq 0$, where $X_i=X\otimes t^i$, subject to the following relations:
\begin{gather*}
  [H_i,H_j]=[E_i,E_j]=[F_i,F_j]=0,\\
  [H_i,E_j]=2E_{i+j},\quad [H_i,F_j]=-2F_{i+j},\quad [E_i,F_j]=H_{i+j}.\label{rel2}
\end{gather*}
An integral basis of the loop algebra $\mathbf U ({\mathfrak sl}_2[t,t^{-1}])$
was constructed in \cite{Gar}.
The following is our second main result.
\begin{thm} \label{qqq}
  As a linear category, $\Tr(\U^*)$ is isomorphic to the idempotented
  integral form $\dUZL$ of the current algebra $\UL$.
\end{thm}
For the definition of $\dUZL$, see Section \ref{sec:linear-category-modz}.

We conclude by summarizing some advantages of $\Tr$ compared with
$K_0$:
\begin{itemize}
\item
The trace is defined for linear categories, not only for additive categories.
\item
The trace is sometimes richer than $K_0$.
\item
The trace of the category is isomorphic to the trace of its Karoubi
envelope.  This property does not hold for the split Grothendieck
group in general.
\end{itemize}

The paper is organized as follows. After recalling general facts about
traces for (linear) categories and bicategories in the first two
sections, we introduce (strongly) upper-triangular linear categories
and compute all Hochschild--Mitchell homology groups for them.  Then
we recall the definitions of the additive $2$-categories $\Ucat$,
$\Udot$ and $\U^*$ and show that $\Udot$ is strongly upper-triangular,
yielding the proof of Theorem \ref{thm1}.  In Section
\ref{sec:presentation-un-m}, we give an explicit presentation of the
categories $\U^*(m,n)$ by generators and relations.  Finally, we
provide the proof of Theorem~\ref{qqq}.
\vspace*{3mm}

\noindent
{\it Acknowledgement}.
The first and  last authors (A.B. $\&$ M.\v{Z}.) were supported by Swiss
National Science Foundation under Grant PDFMP2-141752/1.
K.H. was partially supported by JSPS Grant-in-Aid for Scientific Research (C) 24540077.
A.D.L  was partially supported by NSF grant DMS-1255334 and by the John Templeton Foundation. A.D.L is extremely grateful to Mikhail Khovanov for sharing his insights and vision about higher representation theory.  Some of the calculations appearing in this article were computed in collaboration with him.

\section{Traces of categories and bicategories}
In this section we define the traces of  categories and bicategories.

\subsection{Traces of categories}
Let $\modC $ be a small category.
We denote by $\modC(x,y)$ the set of morphisms from $x$ to $y$ in $\modC$.
For an object $x\in \Ob(\modC )$, set
$\End_\modC (x)=\modC (x,x)$, the set of endomorphism of $x$, which has a
monoid structure.  Set $\End(\modC )=\coprod_{x\in \Ob(\modC )}\End_\modC (x)$.

The {\em trace} of $\modC $ is defined to be the coend
of the Hom functor
\begin{gather*}
  \begin{split}
    \Tr(\modC )
    = \int^{x\in \modC }\modC (x,x)\in \Set.
  \end{split}
\end{gather*}
By abuse of notation, we will also write
$\Tr\modC$ in later sections. %$\Tr(\modC)$.
Unravelling the definition of the coend
(see e.g. \cite[p. 226]{MacLane})
 $\Tr(\modC )$ may be defined also by
\begin{gather*}
  \Tr(\modC )=\End(\modC )/\sim,
\end{gather*}
where the equivalence relation $\sim$ is generated by $fg\sim gf$ for all pairs $f\col x\rightarrow y$, $g\col y\rightarrow x$ with
$x,y\in \Ob(\modC )$.  For $f\in \End(\modC )$, let
$[f]\in \Tr(\modC )$ denote the corresponding equivalence class.

The trace gives rise to a functor
\begin{gather*}
  \Tr\col \Cat\rightarrow \Set,
\end{gather*}
from the category of small categories to the category of sets,
which maps a functor $F\col \modC \rightarrow \modD $ to the map $\Tr(F)\col \Tr(\modC )\rightarrow \Tr(\modD )$
defined by
\begin{gather*}
  \Tr(F)([f])=[F(f)]\in \Tr(\modD )
\end{gather*}
for $f\col x\rightarrow x$ in $\modC $.

Here we summarize some useful facts:
\begin{itemize}
\item For $f\in \End_\modC (x)$ and an isomorphism $a\col %y\simeqto x
\xymatrix@1{y \ar[r]^-{\simeq} & x}$, we have
  $[f]=[a^{-1}fa]$ in $\Tr(\modC )$.
\item
If $\sigma \col %F\overset{\simeq}{\Rightarrow }G
\xymatrix@1{F \ar@{=>}[r]^-{\simeq} & G}$ is a natural isomorphism between two
functors $F,G\col \modC \rightarrow \modD $, then $\Tr(F)=\Tr(G)$.  Thus, the functor
$\Tr\col \Cat\rightarrow \Set$ can be refined to a $2$-functor from the
$2$-category of categories, functors and natural isomorphisms to the
$2$-category of sets, functions and identities of functions.
\item
Equivalence of categories
$\modC \simeq \modD $ induces an isomorphism
$\Tr(\modC )\cong\Tr(\modD )$.
\end{itemize}

\subsection{Products}
\label{sec:products}

Recall that the categories $\Cat$ and $\Set$ have finite products given by
direct products for categories and sets, respectively.

\begin{lem}
  \label{r9}
  The functor $\Tr\col \Cat\rightarrow \Set$ preserves finite products.
\end{lem}

\begin{proof}
For $\modC ,\modD \in \Cat$, we have a bijection
\begin{gather}
  \label{e1}
\varphi\col \Tr(\modC )\times \Tr(\modD )\rightarrow \Tr(\modC \times \modD )
\end{gather}
defined by $\varphi([f],[g])=[(f,g)]$ for $f\in \End(\modC )$,
$g\in \End(\modD )$.
Note also that the trace $\Tr(\mathcal{T})$ of the terminal object $\mathcal{T}$ in
$\Cat$, which is a category with one object $t$ and one morphism $1_t$, is a
one-element set $\{[1_t]\}$, which is terminal in $\Set$.
\end{proof}

By Lemma \ref{r9}, the functor $\Tr$ defines a (strong) monoidal functor between
 cartesian monoidal categories
\begin{gather*}
  \Tr\col (\Cat,\times )\rightarrow (\Set,\times ).
\end{gather*}

\subsection{Traces of bicategories}

Recall that a (small) {\em bicategory} $\CC$ consists of
\begin{itemize}
\item a set
$\Ob(\CC)$ of {\em objects} in $\CC$,
\item a small category $\CC(x,y)$ for $x,y\in \Ob(\CC)$,
\item a functor $\circ\col \CC(y,z)\times \CC(x,y)\rightarrow \CC(x,z)$ for
  $x,y,z\in \Ob(\CC)$,
\item an object $I_x\in \CC(x,x)$  for $x\in \Ob(\CC)$,
\item natural isomorphisms
\begin{gather*}
  \alpha _{h,g,f}\col \xymatrix{(h\circ g)\circ f \ar[r]^{\simeq} & h\circ (g\circ f), } \\
  \lambda _f\col \xymatrix{I_y\circ f \ar[r]^-{\simeq} &f} ,\\
  \rho _f\col \xymatrix{f\circ I_x \ar[r]^-{\simeq} &f},
\end{gather*}
for objects $f\in \Ob(\CC(x,y))$, $g\in \Ob(\CC(y,z))$, $h\in \Ob(\CC(z,w))$,
\end{itemize}
which satisfy certain relations (pentagon etc.).  For more details see \cite[Chapter 7]{Bor} or \cite{Benabou}.
A $2$-category is a bicategory such that the $\alpha _{h,g,f}$, $\lambda _f$ and $\rho _f$ are identities.

For a bicategory $\CC$, we define a category $\TrCC$ with
$\Ob(\TrCC)=\Ob(\CC)$ as follows.
For $x,y\in \Ob(\CC)$, set $\TrCC(x,y)=\Tr(\CC(x,y))$.
For $x,y,z\in \Ob(\CC)$, define the composition map
\begin{gather*}
  \circ\col \TrCC(y,z)\times \TrCC(x,y)\rightarrow \TrCC(x,z)
\end{gather*}
as the composite
\begin{gather}
  \label{e2} \xymatrix{
  \Tr(\CC(y,z))\times \Tr(\CC(x,y)) \ar[r]^-{\varphi}_-{\simeq} &
  \Tr(\CC(y,z)\times \CC(x,y)) \ar[r]^-{\Tr(\circ)} &
  \Tr(\CC(x,z))
  },
\end{gather}
where $\varphi$ is the isomorphism as defined in equation ~\eqref{e1}.  For $\sigma \in \End_{\CC(x,y)}$ and
$\tau \in \End_{\CC(y,z)}$, we have
$[\tau ]\circ[\sigma ]=[\tau \circ\sigma ]$.
The identity morphism for $x\in \Ob(\TrCC)=\Ob(\CC)$ is given by
$[1_{I_x}]$.  The associativity and unitality of composition in $\TrCC$ follows from
the natural isomorphisms $\alpha $, $\lambda $ and $\rho $.  For example, for
$2$-morphisms $\sigma \col f\Rightarrow f$, $\tau \col g\Rightarrow g$ and $\rho \col h\Rightarrow h$ in $\CC$, we have
\begin{gather*}
  ([\rho ]\circ[\tau ])\circ[\sigma ]=
  [(\rho \circ\tau )\circ\sigma ]=
  [\alpha _{h,g,f}((\rho \circ\tau )\circ\sigma )\alpha ^{-1}_{h,g,f}]=
  [\rho \circ(\tau \circ\sigma )]=
  [\rho ]\circ([\tau ]\circ[\sigma ]).
\end{gather*}

\begin{rem}
  \label{r10}
  Recall that a monoidal category may be regarded as a bicategory with
  one object.  Hence the above definition of the trace of bicategory
  specializes to the trace of monoidal category, which takes values in
  categories with one object, i.e., monoids. An internal notion of trace for 1-morphisms inside of a bicategory is given in \cite{GK}.
\end{rem}

\subsection{Horizontal trace of a bicategory}

This subsection is a digression on ``horizontal traces'' of
bicategories, which is not necessary for the rest of the paper.
The motivation for considering ``horizontal traces'' is given at the
end of this subsection.

For a bicategory $\CC$, there is another structure which may be called
the ``trace'' of $\CC$.  In graphical explanations, a $1$-morphism in a
bicategory is sometimes drawn as a horizontal arrow, and a
$2$-morphism between two $1$-morphisms are ``vertical face'' bounded
by two horizontal arrows.  The trace $\TrCC$ of $\CC$ is obtained from
$\CC$ by taking traces in the ``vertical direction'', i.e., in
$2$-morphisms.  Thus one might call $\TrCC$ the {\em vertical trace}
of $\CC$.

Here we introduce another notion of ``trace'' of a bicategory $\CC$,
which we call the ``horizontal trace'' $\TrhorCC$, which is obtained
from $\CC$ by taking trace in the ``horizontal direction''.  More
precisely, $\TrhorCC$ is the category defined as follows.  Here, for
simplicity, our bicategory $\CC$ is a $2$-category.
The objects
in $\TrhorCC$ are $1$-endomorphisms
$f\col x\rightarrow x$, $x\in \Ob(\CC)$.
Morphisms from $f\col x\rightarrow x$ to $g\col y\rightarrow y$ are equivalence classes $[p,\sigma ]$
of pairs $(p,\sigma )$ of a morphism $p\col x\rightarrow y$ in $\CC$ and a $2$-morphism
$\sigma \col p\circ f\Rightarrow g\circ p\col x\rightarrow y$, depicted by
$$
\xy
 (7,10);(-7,10); **\dir{-} ?(1)*\dir{>};
 (7,-10);(-7,-10); **\dir{-} ?(1)*\dir{>};
 (-10,-7);(-10,7); **\dir{-} ?(1)*\dir{>};
 (10,-7);(10,7); **\dir{-} ?(1)*\dir{>};
 (-3,-3);(3,3); **\dir{=} ?(1.05)*\dir{>};
 (-10,10)*{x};
 (10,10)*{x};
 (-10,-10)*{y};
 (10,-10)*{y};
 (0,12)*{f};
 (0,-12)*{g};
 (-12,0)*{p};
 (12,0)*{p};
 (2,-2)*{\sigma};
 (0,15)*{};
 (0,-15)*{};
\endxy
$$
where the equivalence relation on such
pairs is generated by
\begin{gather*}
 (p,(g\circ\tau )\sigma) \sim (p',\sigma (\tau \circ f))
\end{gather*}
or
$$
\xy
 (7,10);(-9,10); **\dir{-} ?(1)*\dir{>};
 (7,-10);(-9,-10); **\dir{-} ?(1)*\dir{>};
 (-14,8)*{};(-14,-8)*{} **\crv{(-20,0)}?(1)*\dir{>};
 (-10,8)*{};(-10,-8)*{} **\crv{(-4,0)}?(1)*\dir{>};
 (10,-7);(10,7); **\dir{-} ?(1)*\dir{>};
 (-1,-2);(5,4); **\dir{=} ?(1.05)*\dir{>};
 (-14,0);(-10,0); **\dir{=} ?(1.12)*\dir{>};
 (-12,10)*{x};
 (10,10)*{x};
 (-12,-10)*{y};
 (10,-10)*{y};
 (-1,12)*{f};
 (-1,-12)*{g};
 (-19,0)*{p};
 (-5,0)*{p'};
 (12,0)*{p};
 (4,-1)*{\sigma};
 (-12,2)*{\tau};
 (0,15)*{};
 (0,-15)*{};
\endxy \quad \sim \quad \xy
 (9,10);(-7,10); **\dir{-} ?(1)*\dir{>};
 (9,-10);(-7,-10); **\dir{-} ?(1)*\dir{>};
 (14,8)*{};(14,-8)*{} **\crv{(20,0)}?(1)*\dir{>};
 (10,8)*{};(10,-8)*{} **\crv{(4,0)}?(1)*\dir{>};
 (-10,-7);(-10,7); **\dir{-} ?(1)*\dir{>};
 (-5,-4);(1,2); **\dir{=} ?(1.05)*\dir{>};
 (10,0);(14,0); **\dir{=} ?(1.12)*\dir{>};
 (12,10)*{x};
 (-10,10)*{x};
 (12,-10)*{y};
 (-10,-10)*{y};
 (1,12)*{f};
 (1,-12)*{g};
 (19,0)*{p'};
 (5,0)*{p};
 (-12,0)*{p'};
 (0,-3)*{\sigma};
 (12,2)*{\tau};
\endxy
$$
for $p,p'\col x\rightarrow y$, $\sigma \col p\circ f\Rightarrow g\circ p'\col x\rightarrow y$,
$\tau \col p'\Rightarrow p\col x\rightarrow y$.
The composite of two morphisms $[p,\sigma ]\col (f\col x\rightarrow x)\rightarrow (g\col y\rightarrow y)$ and
$[q,\tau ]\col (g\col y\rightarrow y)\rightarrow (h\col z\rightarrow z)$ is defined by
\begin{gather*}
  [q,\tau ][p,\sigma ]=[qp, (\tau \circ p)(q\circ \sigma )].
\end{gather*}
The identity morphism is given by
\begin{gather*}
  1_f=[1_x, 1_f]
\end{gather*}
for $f\col x\rightarrow x$.  It is not difficult to check that $\TrhorCC$ is a
well-defined category.

The notion of the horizontal trace of a bicategory may be regarded as
a categorification of the notion of the trace of a category.

There is a canonical functor $\i\col \TrCC\rightarrow \TrhorCC$ defined as follows.
For an object $x\in \Ob(\CC)$ in $\TrCC$, set
\begin{gather*}
  \i(x)=(1_x\col x\rightarrow x).
\end{gather*}
For a morphism $[\sigma \col f\Rightarrow f\col x\rightarrow y]\col x\rightarrow y$ in $\TrCC$, set
\begin{gather*}
  \i([\sigma ])=[f, \sigma\col f\circ 1_x\Rightarrow 1_y\circ f].
\end{gather*}
It is easy to see that the functor $\i$ is full and faithful.
Thus, $\TrhorCC$ has more information about $\CC$ than $\TrCC$.

The horizontal trace may be loosely regarded as a generalization to
bicategories of the {\em tube algebra} of monoidal category (see for
example \cite{Ocn,EY}).  If our bicategory $\CC$ admits biadjoints, then the hom-sets
$\TrhorCC(f\col x\to x, g\col y\to y)$ can be naturally regarded as a
``skein module of $\CC$-diagrams'' on the annulus $A=S^1\times[0,1]$
with the $1$-morphism $x$ put on a point $(\mathrm{pt},0)\in\partial
A$ and $y$ on $(\mathrm{pt},1)\in\partial A$.  It is natural to
consider skein modules of $\CC$-diagrams on more general surfaces with
specified $1$-morphisms on boundary points.  The horizontal trace
$\TrhorCC$ plays an analogous role for studying such general skein
modules, as the tube algebra does to the skein modules associated to a
monoidal category with duals.  It would be interesting to consider the
horizontal trace and the skein modules for $\CC={\mathcal{U}}^*$.

\section{Traces of linear (bi)categories}
\label{sec:trac-line-bicat-1}

\subsection{Linear and additive categories}
A {\em linear} category (also called $\Ab$-category or {\em
preadditive} category) is a category enriched in
 the category $\Ab$
of abelian groups. This means it is a category whose hom-sets are
equipped with structures of abelian groups and the composition maps
are bilinear  (compare \cite[p. 276]{
MacLane}).

A {\em linear functor}  between
two linear categories $\modC $ and $\modD $ is a functor $F$ from
$\modC $ to $\modD $ such that for $x,y\in \Ob(\modC )$, the map
$F\col \modC (x,y)\longrightarrow \modD (F(x),F(y))$ is an
abelian group homomorphism.

An {\em additive} category is a linear category equipped with a zero
object and direct sums.  For a linear category $\modC $, there is a
universal additive category generated by $\modC $, called the {\em
additive closure} $\modC ^{\oplus}$, in which the objects are formal
finite direct sums of objects in $\modC $ and the morphisms are
matrices of morphisms in $\modC$.  There is a canonical fully faithful
functor $i\col \modC \rightarrow \modC ^\oplus$.  Every linear functor
$F\col \modC \rightarrow \modD $ from $\modC $ to an additive category
$\modD $ factors through $i$ uniquely up to natural isomorphism.

\subsection{Traces of linear categories}
\label{sec:trac-line-categ}

Similarly to the case of categories which are not linear, for a linear
category $\C$, the {\em trace} $\Tr(\modC )$ of $\modC $ is defined to
be the coend of the Hom functor
\begin{gather*}
  \Tr(\C )=\int^{x\in \modC }\modC (x,x)
=\bigoplus_{x\in \Ob(\modC )}\modC (x,x)/\Span\{fg-gf\},
\end{gather*}
where $f$ and $g$ runs through all pairs of morphisms $f\col
x\longrightarrow y$, $g\col y\longrightarrow x$ with $x,y\in
\Ob(\modC )$.

The trace $\Tr$ gives a functor from the category of small linear
categories to the category of abelian groups.

\begin{lem}
  \label{r14}
  If $\modC $ is an additive category, then for $f\col x\rightarrow x$ and $g\col y\rightarrow y$,
  we have
  \begin{gather*}
    [f\oplus g]=[f]+[g].
  \end{gather*}
\end{lem}

\begin{proof}
  Since $f\oplus g=(f\oplus 0)+(0\oplus g)\col x\oplus y\rightarrow x\oplus y$, we have
  \begin{gather*}
    [f\oplus g]=[f\oplus 0]+[0\oplus g].
  \end{gather*}
  Now we have $[f\oplus0]=[ifp]=[pif]=[f]$ where $p\col x\oplus y\rightarrow x$
  and $i\col x\rightarrow x\oplus y$ are the projection and the
  inclusion. Similarly, we have $[0\oplus g]=[g]$. Hence the result.
\end{proof}

\subsection{Traces of linear bicategories}
\label{sec:trac-line-bicat}

A {\em linear bicategory} is a bicategory $\CC$ such that
\begin{enumerate}
\item for $x,y\in \Ob(\CC)$, the category $\CC(x,y)$ is equipped with a
  structure of a linear category,
\item for $x,y,z\in \Ob(\CC)$, the functor
  $\circ\col \CC(y,z)\times \CC(x,y)\rightarrow \CC(x,z)$ is ``bilinear'' in the sense
  that the functors $-\circ f\col \CC(y,z)\rightarrow \CC(x,z)$ for
  $f\in \Ob(\CC(x,y))$ and $g\circ -\col \CC(x,y)\rightarrow \CC(x,z)$ for
  $g\in \Ob(\CC(y,z))$ are linear functors.
\end{enumerate}

The trace $\Tr(\CC)$ of a linear bicategory $\CC$ is defined similarly
to the trace of bicategory, and is equipped with a linear category
structure.

\subsection{Traces of additive closures}

Here we consider the trace of the additive closure $\modC ^{\oplus}$ of a linear
category $\modC $.
The homomorphism
\begin{gather*}
  \Tr(i)\col \Tr(\modC )\rightarrow \Tr(\modC ^\oplus)
\end{gather*}
induced by the canonical functor $i\col \modC \rightarrow \modC ^{\oplus}$ is an
isomorphism.
In fact, the inverse
$\tr\col \Tr(\modC ^\oplus)\rightarrow \Tr(\modC )$ is defined by
\begin{gather*}
  \tr([(f_{k,l})_{k,l}])=\sum_k [f_{k,k}]
\end{gather*}
for an endomorphism in $\modC ^\oplus$
\begin{gather*}
  (f_{k,l})_{k,l\in \{1,\ldots ,n\}}\col x_1\oplus\dots \oplus x_n\rightarrow x_1\oplus\dots \oplus x_n
\end{gather*}
with $f_{k,l}\col x_l\rightarrow x_k$ in $\modC $.

\subsection{Traces and the Karoubi envelope}

Let $\modC $ be a linear category.  An idempotent $e\col x\rightarrow x$ in $\modC $ is
said to {\em split} if there is an object $y$ and morphisms $g\col x\rightarrow y$,
$h\col y\rightarrow x$ such that $hg=e$ and $gh=1_y$.

The {\em Karoubi envelope} $\Kar(\mathcal C)$ (also called {\em
idempotent completion}) of $\modC $ is the category whose objects are
pairs $(x,e)$ of objects $x\in \Ob(\C)$ and an idempotent endomorphism
$e : x\rightarrow x$, $e^2=e$, in $\mathcal C$ and whose morphisms
$$f : (x,e)\rightarrow(y,e')$$
are morphisms $f\col x\rightarrow y$ in $\modC $ such that $f=e'fe$.
Composition is induced by the composition in $\mathcal C$ and the identity morphism is
$e : (x,e)\rightarrow(x,e)$.
$\Kar(\modC )$ is equipped with a linear category structure.
It is well known that the Karoubi envelope of an additive category
is additive.

There is a natural fully faithful linear functor $$\iota:\mathcal C\rightarrow
\Kar(\mathcal C)$$ such that $\iota(x)=(x,1_x)$ for $x\in \Ob(\modC )$ and
$\iota(f\col x\rightarrow y)=f$.  The Karoubi envelope $\Kar(\modC )$ is
universal in the sense that if $F\col \modC \rightarrow \modD $ is
a linear functor from $\modC$ to a linear category $\modD $ with split idempotents,
then $F$ extends to a functor from $\Kar(\modC)$ to $\modD$ uniquely
up to natural isomorphism \cite[Proposition 6.5.9]{Bor}.

\begin{prop}
  \label{r1}
  The functor $\iota\col \modC\longrightarrow \Kar(\modC)$ induces an isomorphism
  \begin{gather*}
    \Tr(\iota)\col \Tr(\modC )\overset{\cong}\longrightarrow \Tr(\Kar(\modC )).
  \end{gather*}
\end{prop}

\begin{proof}
  We can construct a map $u\col \Tr(\Kar(\modC ))\longrightarrow \Tr(\modC )$ such that, for
  $f\col (x,e)\longrightarrow (x,e)$ in $\Kar(\modC )$, $[f]\in \Tr(\Kar(\modC ))$ is mapped to
  $[f]\in \Tr(\modC )$. Then one can check that $u$ is an inverse to $\Tr(\iota)$.
\end{proof}

Given a linear bicategory $\CC$ we define the Karoubi envelope
$\Kar(\CC)$ as the linear category with $\Ob(\Kar(\CC)) = \Ob(\CC)$,
and for $x,y \in \Ob(\CC)$ the hom-categories are given by
$\Kar(\CC)(x,y) := \Kar( \CC(x,y) )$.  The composition functor
$\Kar(\CC)(y,z) \times \Kar(\CC)(x,y) \to \Kar(\CC)(x,z)$ is induced
by the universal property of the Karoubi envelope from the composition
functor in $\CC$.  The fully-faithful additive functors $\CC(x,y) \to
\Kar(\CC(x,y))$ combine to form an additive $2$-functor $\CC \to
\Kar(\CC)$ that is universal with respect to splitting idempotents in
the  hom-categories $\CC(x,y)$.

\subsection{Split Grothendieck groups and traces}
\label{sec:split-groth-groups}

For an additive category $\modC $, the {\em split Grothendieck group}
$K_0(\modC )$ of $\modC $ is the abelian group generated by the
isomorphism classes of objects of $\modC $ with relations $[x\oplus
y]_{\cong}=[x]_{\cong}+[y]_{\cong}$ for $x,y\in \Ob(\modC)$.  Here
$[x]_{\cong}$ denotes the isomorphism class of $x$.  The split
Grothendieck group $K_0$ is functorial.

Define a homomorphism
\begin{gather*}
  h_\modC \col K_0(\modC )\longrightarrow \Tr(\modC )
\end{gather*}
by
\begin{gather*}
  h_\modC ([x]_{\cong})=[1_x]
\end{gather*}
for $x\in \Ob(\modC )$.  Indeed, one can easily check that
\begin{gather*}
  h_\modC ([x\oplus y]_{\cong})=[1_x]+[1_y] \, .
\end{gather*}

The maps $h_\modC $ form a natural transformation
\begin{gather*}
  h\col K_0 \Rightarrow  \Tr \col \AdCat\rightarrow \Ab,
\end{gather*}
where $\AdCat$ denote the category of small additive category.

Given a linear bicategory $\CC$, we define the split Grothendieck
category $K_0(\CC)$ of $\CC$ as the linear category with
$\Ob(K_0(\CC)) = \Ob(\CC)$ and $K_0(\CC)(x,y) := K_0(\CC(x,y))$ for
any two objects $x,y \in \Ob(\CC)$.  For $[f]_{\cong} \in
\Ob(K_0(\CC)(x,y))$ and $[g]_{\cong} \in \Ob(K_0(\CC)(y,z)) $ the
composition in $K_0(\CC)$ is defined by $[g]_{\cong} \circ [f]_{\cong}
:= [g \circ f]_{\cong}$.

The homomorphisms $h_{\CC(x,y)}$ taken over all objects $x,y \in \Ob(\CC)$ give rise to a linear functor \begin{equation} \label{eq_chernchar}
h_{\CC} \col \K_0(\CC) \to \Tr(\CC)
\end{equation}
which is the identity map on objects and sends $K_0(\CC)(x,y) \to \Tr(\CC)(x,y)$ via the homomorphism $h_{\CC(x,y)}$.  It is easy to see that this assignment preserves composition since
\begin{align}
  h_{\CC}([g]_{\cong} \circ [f]_{\cong}) = h_{\CC}([g \circ f]_{\cong}) = [1_{g \circ f}]
   = [1_g \circ 1_f] = [1_g] \circ [1_f] = h_{\CC}([g]_{\cong}) \circ h_{\CC}([f]_{\cong}).
\end{align}

\subsection{Tools for computing the trace}
In this subsection we collect a few results which will be needed later.

\begin{prop}
\label{trdec}
Let $\C$ be a linear category and $B\subset\Ob(\C)$ such that every object $x$ of $\C$ is isomorphic to direct sum of elements of $B$. Let $\C|_B$ be the full
subcategory of $\C$ with $\Ob(\C|_B)=B$.
Then $\Tr(\C)$ is isomorphic to $\Tr(\C|_B)$.
\end{prop}

\begin{proof}
Let $S:=\{\bigoplus_ib_i\;|\;b_i\in B\}$, so $\C|_S$ is equivalent to ${\C|_B}^\oplus$.
Since every object of $\C$ is equivalent to an element of $S$, we have
$\C\simeq\C|_S\simeq{\C|_B}^\oplus$.  So,
$\Tr(\C)\simeq\Tr({\C_B}^\oplus)\simeq\Tr(\C|_B)$.
\end{proof}

\begin{prop}
\label{general}
Let $\C$ be a small linear category. Let $H:=\bigoplus_{x\in\Ob(\C)}\C(x,x)$,
and  let $K\subset H$ be a subgroup.
Assume that there is a homomorphism
 $\pi\col H\rightarrow K$  with the following properties:
\begin{enumerate}
\item $\pi$ is a projection,
\item $[\pi(f)]=[f]\in\Tr(\C)$ for every $f\in H$, and
\item $\pi(gh)=\pi(hg)$ for every $g\in\C(x,y)$ and $h\in\C(y,x)$ ($x,y\in\Ob(\C)$);
\end{enumerate}
then $\Tr(\C)$ is isomorphic to $K$.
\end{prop}

\begin{proof}
  The inclusion $i\col K\hookrightarrow H$ induces a
  homomorphism $i^*\col K\rightarrow\Tr(\C)$, $k\mapsto[k]$.
  The map $\pi$ induces $\bar{\pi}\col \Tr(\C)\to K$.
  It is easy to check that $i^*$ and $\bar{\pi}$ are inverse to each
  other.  Hence $\Tr(\C)\cong K$.
\end{proof}

\section{The Hochschild--Mitchell homology of upper-triangular categories}

In this section, we introduce  (strongly) upper-triangular linear categories
and compute their Hochschild--Mitchell homology groups.

\subsection{Hochschild--Mitchell homology of linear categories}

Recall that for a ring $R$, the Hochschild homology $\HH_*(R)$ of $R$
can be defined as the homology of the Hochschild complex (see \cite{Mitchell}):
$$
C_*=C_*(R):\quad \quad \dots\longrightarrow C_n \xto{d_n} C_{n-1} \xto{d_{n-1}} \dots \xto{d_2} C_1 \xto{d_1} C_0 \longrightarrow 0,
$$
where $C_n(R)=R^{\otimes n+1}$ and
$$
d_n(a_0\otimes\dots\otimes
a_n):=\sum_{i=0}^{n-1}(-1)^ia_0\otimes\dots\otimes
a_ia_{i+1}\otimes\dots\otimes a_n+
(-1)^na_na_0\otimes a_1\otimes\dots\otimes a_{n-1}
$$
for $a_0,\ldots ,a_{n}\in R$.

The Hochschild--Mitchell homology of a linear category can be
similarly defined as follows.
Let $\modC $ be a small linear category.
Define the {\em Hochschild--Mitchell complex} of $\modC $
$$
C_*=C_*(\modC ):\quad \quad \dots\longrightarrow C_n \xto{d_n} C_{n-1} \xto{d_{n-1}} \dots \xto{d_2} C_1 \xto{d_1} C_0 \longrightarrow 0,
$$
where
$$
C_n=C_n(\modC ):=\bigoplus_{x_0,\dots,x_n\in\Ob(\C)}\C(x_n,x_0)\otimes\C(x_{n-1},x_n)\otimes\dots\otimes\C(x_0,x_1),
$$
$$
d_n( f _n\otimes f _{n-1}\otimes\dots\otimes f _0):=
\left(\sum_{i=0}^{n-1}(-1)^i f _n\otimes\dots\otimes f _{n-i} f _{n-i-1}\otimes\dots\otimes f _0\right)
+(-1)^n f _0 f _n\otimes f _{n-1}\otimes\dots\otimes f _1.
$$

The Hochschild--Mitchell homology $\HH_*(\modC )$ of $\modC $ is defined to
be the homology of the chain complex $C_*(\modC )$.
It is clear that $\HH_0(\modC )=C_0/d_1(C_1)$ is isomorphic to $\Tr(\modC )$.
Here we list some well-known properties, see \cite{Mitchell}.

\begin{lem} \hfill
  \label{r18}
  \begin{enumerate}
  \item An equivalence $\modC \simeq\modD $ of linear categories induces
    an isomorphism of the Hochschild--Mitchell homology $\HH_*(\modC )\cong
    \HH_*(\modD )$. (More generally, a Morita equivalence of linear categories
    induces an isomorphism on the Hochschild--Mitchell homology. We do not
    need this fact.)
  \item
    The inclusion functor $\modC \rightarrow \modC ^\oplus$ of a linear category $\modC $ into
    its additive closure $\modC ^\oplus$ induces an isomorphism
    $\HH_*(\C)\cong\HH_*(\C^{\oplus})$.
  \end{enumerate}
\end{lem}

\subsection{Upper-triangularity of linear categories}
\label{sec:upper-triang-line}

Let us define a property of linear categories which allows us to compute all
their Hochschild--Mitchell homology groups.

\begin{defn}
  \label{r22}
  A small linear category $\modC $ is said to be {\em upper-triangular} if
  there is a partial order $\le$ on the set $\Ob(\modC )$ such that,
  for all $x,y\in\Ob(\modC)$, $\modC(x,y)\neq0$ implies $x\le y$.
\end{defn}

Note that $\modC$ is upper-triangular if and only if there is no sequence
\begin{gather*}
  x_0\xto{f_0}x_1\xto{f_1}\cdots \xto{f_{n-1}}x_n\xto{f_n}x_0 \quad (n\ge 1)
\end{gather*}
of nonzero morphisms in $\modC $ unless $x_0=x_1=\dots =x_n$.

\begin{lem}
  \label{r19}  For an upper-triangular linear category $\modC $, we have
  \begin{gather}
    \label{e5}
    \HH_*(\modC )
    \cong\bigoplus_{x\in \Ob(\modC )}\HH_*(\modC (x,x)).
  \end{gather}
\end{lem}

\begin{proof}
From the definition of upper-triangularity, we have, for
$x_0,\dots,x_n\in\Ob(\C)$, $n\ge 0$,
\begin{gather*}
  \C(x_n,x_0)\otimes\C(x_{n-1},x_n)\otimes\dots\otimes\C(x_0,x_1)=0
\end{gather*}
unless $x_0=x_1=\dots =x_n$.
Hence
\begin{gather*}
  C_n(\modC )=\bigoplus_{x\in \Ob(\modC )}\modC (x,x)^{\otimes n+1},
\end{gather*}
and the Hochschild--Mitchell complex of $\modC $ decomposes
as
\begin{gather*}
  C_*(\modC )\cong \bigoplus_{x\in \Ob(\modC )}C_*(\modC (x,x)).
\end{gather*}
Hence we have \eqref{e5}.
\end{proof}

\begin{rem}
  \label{r26}
  In the situation of Lemma \ref{r18},
  similar isomorphisms hold for the cyclic homology \cite{Loday}
  \begin{gather*}
  \operatorname{HC}_*(\modC )\cong\bigoplus_{x\in \Ob(\modC )}
\operatorname{HC}_*(\modC (x,x)).
  \end{gather*}
\end{rem}

\begin{defn}
  \label{r23}
A {\em strongly upper-triangular} linear category is an
upper-triangular linear category $\modC $ such that for all $x\in \Ob(\modC )$,
we have $\modC (x,x)\cong \modZ $.
\end{defn}

\begin{prop}
  \label{r20}
  For a strongly upper-triangular linear category $\modC $, we have
  \begin{gather*}
    \HH_i(\modC) \cong
    \begin{cases}
     \modZ \Ob(\modC )& \text{if $i=0$},\\
     0& \text{if $i>0$}.
    \end{cases}
  \end{gather*}
  Here $\modZ \Ob(\modC )$ denotes the free $\modZ $-module spanned by $\Ob(\modC )$.
\end{prop}

\begin{proof}
  The result follows from Lemma \ref{r19} and
  \begin{gather*}
    \HH_i(\Z) \cong
    \begin{cases}
      \Z& \text{if $i=0$},\\
      0& \text{if $i>0$}.
    \end{cases}
  \end{gather*}
\end{proof}

\subsection{Additive categories with strongly upper-triangular bases}

Let $\modC $ be an additive category, and let $B\subset \Ob(\modC )$ be a subset.
Denote by $\modC |_B$ the full subcategory of $\modC $ with
$\Ob(\modC |_B)=B$.
The set $B$ is called a {\em strongly upper-triangular basis} of
$\modC $ if the following two conditions hold.
\begin{enumerate}
\item The inclusion functor $\modC |_B\rightarrow \modC $ induces equivalence of additive
  categories $(\modC |_B)^\oplus\simeq \modC $.
\item $\modC |_B$ is strongly upper-triangular.
\end{enumerate}

\begin{prop}
  \label{r24}
  For an additive category $\modC $ with a strongly upper-triangular
  basis $B$, we have
\begin{gather*}
  \HH_i(\modC)\cong
  \begin{cases}
    \Z B&\text{if $i=0$},\\
    0&\text{if $i>0$}.
  \end{cases}
\end{gather*}

\end{prop}

\begin{proof}
Because of Lemma \ref{r18} and the definition of the strongly
upper-triangular basis, the Hochschild--Mitchell homology of the category
$\modC$ is the same as that of the category $\modC |_B$.  Proposition
\ref{r20} concludes the proof.
\end{proof}

\section{The $2$-categories $\Ucat$, $\U^*$ and their Karoubi envelopes}
\label{sec:2-categories-ucatd}

In this section we define $2$-categories $\Ucat$ and $\U^*$ as well as
their Karoubi envelopes $\dotU$ and $\dotU^*$.
We recall some definitions and results of the ``thick calculus''
developed in \cite{KLMS}.

\subsection{The algebras $\bfU$, $_{\mathcal{A}}{\mathbf{U}}$ and $\UA$}
 The quantum group $\bfU=\bfU_q(\fsl_2)$ is the unital, associative algebra over the
 rational function field $\Q(q)$ with generators $E$, $F$, $K$ and
 $K^{-1}$ subject to the relations
 \begin{gather*}
 KK^{-1}=K^{-1}K=1,\quad  KE=q^2 EK,\quad KF =q^{-2}FK,\quad
 EF-FE=\frac{K-K^{-1}}{q-q^{-1}}.
 \end{gather*}
For $\calA=\Z[q,q^{-1}]$,
the integral form ${}_\calA\bfU={}_\calA\bfU(\fsl_2)$ is the $\calA$-subalgebra of $\bfU$,
  generated by $K,K^{-1}$ and the divided
 powers $E^{(k)}=E^k/[k]!$ and $F^{(k)}=F^k/[k]!$ for $k\ge1$, where
 $[k]!=\prod_{i=1}^k\frac{q^i-q^{-i}}{q-q^{-1}}$.

 The idempotented version $\UA=\UA(\fsl_2)$ of $_{\mathcal{A}}{\mathbf{U}}$,
 introduced in \cite{BLM}, can be obtained from ${}_\calA\bfU$ by
 replacing the unit with a collection of mutually orthogonal
 idempotents $1_n$ for $n\in \Z$, such that
 $$K1_n=1_nK=q^n 1_n, \quad E^{(k)}1_n=1_{n+2k}E^{(k)},\quad
 F^{(k)}1_n=1_{n-2k}F^{(k)}\, .$$
One may regard $\UA$ both as a
 non-unital algebra and as a linear category.  In the latter case, we
 regard $\UA$ as a linear category with objects in $\Z$ such that the
 abelian group of morphisms from $m\in\Z$ to $n\in\Z$ is the
 $\calA$-submodule of the non-unital ring $\UA$ consisting of the
elements $1_nx1_m$ for $x\in {}_\calA\bfU$.

The defining relations for the algebra $\UA$ are
\begin{equation}
 \label{lubp}
 E^{(a)}E^{(b)}1_n=\left[ \begin{array}{c}a+b\\b\end{array}\right]E^{(a+b)}1_n,
\end{equation}
\begin{equation}
 F^{(a)}F^{(b)}1_n=\left[ \begin{array}{c}a+b\\b\end{array}\right]F^{(a+b)}1_n,
\end{equation}
\begin{equation}
 E^{(a)}F^{(b)}1_n=\sum_{j=0}^{\min(a,b)}\left[ \begin{array}{c}a-b+n\\j\end{array}\right]F^{(b-j)}E^{(a-j)}1_n,\quad\text{for }n\geq b-a,
\end{equation}
\begin{equation}
 \label{lubz}
 F^{(b)}E^{(a)}1_n=\sum_{j=0}^{\min(a,b)}\left[ \begin{array}{c}b-a-n\\j\end{array}\right]E^{(a-j)}F^{(b-j)}1_n,\quad\text{for }n\leq b-a\, .
\end{equation}

\subsection{The $2$-category $\Ucat$}
The $2$-category $\Ucat=\Ucat(\mf{sl}_2)$ is the additive $2$-category
consisting of
\begin{itemize}
\item objects $n$ for $n \in \Z$,
\item for a signed sequence $\underline{\epsilon}=(\epsilon_1,\epsilon_2, \dots, \epsilon_m)$, with $\epsilon_1,\dots,\epsilon_m \in \{+,-\}$, define
\[\cal{E}_{\underline{\epsilon}}:= \cal{E}_{\epsilon_1} \cal{E}_{\epsilon_2} \dots \cal{E}_{\epsilon_m}
\]
where $\cal{E}_{+}:=\cal{E}$ and $\cal{E}_{-}:= \cal{F}$.  A
$1$-morphisms from $n$ to $n'$ is a formal finite direct sum of strings
\[
 \cal{E}_{\underline{\epsilon}}\onen \la t\ra =
  \onenn{n'}\cal{E}_{\underline{\epsilon}}\onen \la t\ra
\]
for any $t \in \Z$ and signed sequence $\underline{\epsilon}$ such that $n'=n+ 2\sum_{j=1}^m\epsilon_j1$.
\item $2$-morphisms are $\Z$-modules spanned by (vertical and horizontal) composites of identity $2$-morphisms and the following tangle-like diagrams
\begin{align}
  \xy 0;/r.17pc/:
 (0,7);(0,-7); **\dir{-} ?(.75)*\dir{>};
 (0,0)*{\bullet};
 (7,3)*{ \scs n};
 (-9,3)*{\scs n+2};
 (-10,0)*{};(10,0)*{};
 \endxy &\maps \cal{E}\onen\la t \ra \to \cal{E}\onen\la t+2 \ra  & \quad
 &
    \xy 0;/r.17pc/:
 (0,7);(0,-7); **\dir{-} ?(.75)*\dir{<};
 (0,0)*{\bullet};
 (7,3)*{ \scs n};
 (-9,3)*{\scs  n-2};
 (-10,0)*{};(10,0)*{};
 \endxy\maps \cal{F}\onen\la t \ra \to \cal{F}\onen\la t+2 \ra  \nn \\
   & & & \nn \\
   \xy 0;/r.17pc/:
  (0,0)*{\xybox{
    (-4,-4)*{};(4,4)*{} **\crv{(-4,-1) & (4,1)}?(1)*\dir{>} ;
    (4,-4)*{};(-4,4)*{} **\crv{(4,-1) & (-4,1)}?(1)*\dir{>};
     (9,1)*{\scs  n};
     (-10,0)*{};(10,0)*{};
     }};
  \endxy \;\;&\maps \cal{E}\cal{E}\onen\la t \ra  \to \cal{E}\cal{E}\onen\la t-2 \ra  &
  &
   \xy 0;/r.17pc/:
  (0,0)*{\xybox{
    (-4,4)*{};(4,-4)*{} **\crv{(-4,1) & (4,-1)}?(1)*\dir{>} ;
    (4,4)*{};(-4,-4)*{} **\crv{(4,1) & (-4,-1)}?(1)*\dir{>};
     (9,1)*{\scs  n};
     (-10,0)*{};(10,0)*{};
     }};
  \endxy\;\; \maps \cal{F}\cal{F}\onen\la t \ra  \to \cal{F}\cal{F}\onen\la t-2 \ra  \nn \\
  & & & \nn \\
     \xy 0;/r.17pc/:
    (0,-3)*{\bbpef{}};
    (8,-5)*{\scs  n};
    (-10,0)*{};(10,0)*{};
    \endxy &\maps \onen\la t \ra  \to \cal{F}\cal{E}\onen\la t+1+n \ra   &
    &
   \xy 0;/r.17pc/:
    (0,-3)*{\bbpfe{}};
    (8,-5)*{\scs n};
    (-10,0)*{};(10,0)*{};
    \endxy \maps \onen\la t \ra  \to\cal{E}\cal{F}\onen\la t+1-n\ra  \nn \\
      & & & \nn \\
  \xy 0;/r.17pc/:
    (0,0)*{\bbcef{}};
    (8,4)*{\scs  n};
    (-10,0)*{};(10,0)*{};
    \endxy & \maps \cal{F}\cal{E}\onen\la t \ra \to\onen\la t+1+n \ra  &
    &
 \xy 0;/r.17pc/:
    (0,0)*{\bbcfe{}};
    (8,4)*{\scs  n};
    (-10,0)*{};(10,0)*{};
    \endxy \maps\cal{E}\cal{F}\onen\la t \ra  \to\onen\la t+1-n \ra \nn
\end{align}
for every $n,t\in\modZ$. The degree of a $2$-morphism is the difference between  degrees of the target and  the source.
\end{itemize}

Diagrams are read from right to left and bottom to top.  The rightmost region
in our diagrams is usually colored by $n$.
The identity $2$-morphism of the $1$-morphism $\cal{E} \onen$ is represented by an upward oriented line (likewise, the identity $2$-morphism of $\cal{F} \onen$ is represented by a downward oriented line).

The $2$-morphisms satisfy the following relations
 (see \cite{Lau1} for more details).
\begin{enumerate}
\item \label{item_cycbiadjoint} The $1$-morphisms $\cal{E} \onen$ and $\cal{F} \onen$ are biadjoint (up to a specified degree shift). Moreover, the $2$-morphisms are cyclic with respect to this biadjoint structure.

\item The $\cal{E}$'s carry an action of the nilHecke algebra. Using the adjoint structure this induces an action of the nilHecke algebra on the $\cal{F}$'s.

\item \label{item_positivity} Dotted bubbles of negative degree are zero, so that
for all $m \geq 0$ one has
\begin{equation} \label{eq_positivity_bubbles}
\xy 0;/r.18pc/:
 (-12,0)*{\cbub{m}{}};
 (-8,8)*{n};
 \endxy
  = 0
 \qquad  \text{if $m< n-1$,} \qquad \xy 0;/r.18pc/: (-12,0)*{\ccbub{m}{}};
 (-8,8)*{n};
 \endxy = 0\quad
  \text{if $m< -n -1$.}
\end{equation}
Dotted bubble of degree $0$ are equal to the identity $2$-morphism:
\[
\xy 0;/r.18pc/:
 (0,0)*{\cbub{n-1}};
  (4,8)*{n};
 \endxy
  =  \Id_{\onen} \quad \text{for $n \geq 1$,}
  \qquad \quad
  \xy 0;/r.18pc/:
 (0,0)*{\ccbub{-n-1}};
  (4,8)*{n};
 \endxy  =  \Id_{\onen} \quad \text{for $n \leq -1$.}\]
\end{enumerate}

We use the following notation for the
 dotted  bubbles:
\begin{equation*}
 \xy
(0,0)*{\lbub};
 (5,4)*{n};
 (6,0)*{};
 (5,-2)*{\scriptstyle m};
\endxy := \xy
 (5,4)*{n};
(0,0)*{\lbbub};
 (9,-2)*{\scriptstyle m+n-1};
\endxy,
\quad\quad
 \xy
(0,0)*{\rbub};
 (5,4)*{n};
 (6,0)*{};
 (5,-2)*{\scriptstyle m};
\endxy := \xy
 (5,4)*{n};
(0,0)*{\rbbub};
 (9,-2)*{\scriptstyle m-n-1};
\endxy,
\end{equation*}
so that
$$\deg\left( \xy
(0,0)*{\lbub};
 (5,4)*{n};
 (6,0)*{};
 (5,-2)*{\scriptstyle m};
\endxy\right)= \deg\left(\xy
(0,0)*{\rbub};
 (5,4)*{n};
 (6,0)*{};
 (5,-2)*{\scriptstyle m};
\endxy\right)=2m.$$
We call a clockwise (resp. counterclockwise) bubble  fake
if $m+n-1<0$ and (resp. if $m-n-1<0$). The fake bubbles are defined recursively by
the homogeneous terms of the equation
\begin{equation}
\label{gras}
 \left(
\xy
(0,0)*{\rbub};
 (5,4)*{n};
 (7,0)*{};
 (5,-2)*{\scriptstyle 0};
\endxy+\xy
(0,0)*{\rbub};
 (5,4)*{n};
 (7,0)*{};
 (5,-2)*{\scriptstyle 1};
\endxy t+\dots+\xy
(0,0)*{\rbub};
 (5,4)*{n};
 (7,0)*{};
 (5,-2)*{\scriptstyle j};
\endxy t^j+\cdots
\right)\left(
\xy
(0,0)*{\lbub};
 (5,4)*{n};
 (7,0)*{};
 (5,-2)*{\scriptstyle 0};
\endxy+\xy
(0,0)*{\lbub};
 (5,4)*{n};
 (7,0)*{};
 (5,-2)*{\scriptstyle 1};
\endxy t+\dots+\xy
(0,0)*{\lbub};
 (5,4)*{n};
 (7,0)*{};
 (5,-2)*{\scriptstyle j};
\endxy t^j+\cdots
\right)=\Id_{\onen}
\end{equation}
and the additional condition
\begin{equation*}
 \xy
(0,0)*{\lbub};
 (5,4)*{n};
 (6,0)*{};
 (5,-2)*{\scriptstyle 0};
\endxy=\xy
(0,0)*{\rbub};
 (5,4)*{n};
 (6,0)*{};
 (5,-2)*{\scriptstyle 0};
\endxy=\Id_{\onen}.
\end{equation*}
One can check that relation \eqref{gras} holds  also for the real bubbles.
So we will not distinguish between them in what follows.

\begin{enumerate}
\setcounter{enumi}{3}
\item There are additional relations:
\begin{equation}
\label{rei1}
\xy
(0,-8)*{};(0,8)*{} **\crv{(0,4) & (6,4) & (6,-4) & (0,-4)}?(1)*\dir{>};
(6,6)*{n};
(-3,0)*{};(8,0)*{};
\endxy = - \sum_{f_1+f_2=-n}\xy
 (0,-8)*{};(0,8)*{} **\dir{-} ?(1)*\dir{>} ?(.5)*{\bullet};
(-2,2)*{\scriptstyle f_1};
 (8,0)*{\lbub};
 (13,-2)*{\scriptstyle f_2};
(6,6)*{n};
 (-3,0)*{};(15,0)*{};
\endxy\, ,\quad\quad\xy
(0,-8)*{};(0,8)*{} **\crv{(0,4) & (-6,4) & (-6,-4) & (0,-4)}?(1)*\dir{>};
(-6,6)*{n};
(3,0)*{};(-8,0)*{};
\endxy = \sum_{f_1+f_2=n}\xy
 (0,-8)*{};(0,8)*{} **\dir{-} ?(1)*\dir{>} ?(.5)*{\bullet};
(2,2)*{\scriptstyle f_1};
 (-8,0)*{\rbub};
 (-3,-2)*{\scriptstyle f_2};
(-6,6)*{n};
 (3,0)*{};(-13,0)*{};
\endxy\, ,
\end{equation}
\begin{equation}
\label{rei2}
\xy
{\ar (0,-8)*{}; (0,8)*{}};
(-3,0)*{};(3,0)*{};
\endxy\xy
{\ar (0,8)*{}; (0,-8)*{}};
(4,4)*{n};
(-3,0)*{};(6,0)*{};
\endxy = -\xy
(-3,-8)*{};(-3,8)*{} **\crv{(-3,-4) & (3,-4) & (3,4) & (-3,4)}?(1)*\dir{>};
(3,-8)*{};(3,8)*{} **\crv{(3,-4) & (-3,-4) & (-3,4) & (3,4)}?(0)*\dir{<};
(6,4)*{n};
(-7,0)*{};(8,0)*{};
\endxy + \sum_{f_1+f_2+f_3=n-1}\xy
 (3,9)*{};(-3,9)*{} **\crv{(3,4) & (-3,4)} ?(1)*\dir{>} ?(.2)*{\bullet};
(5,6)*{\scriptstyle f_1};
 (0,0)*{\rbub};
(5,-2)*{\scriptstyle f_2};
 (-3,-9)*{};(3,-9)*{} **\crv{(-3,-4) & (3,-4)} ?(1)*\dir{>}?(.8)*{\bullet};
(5,-6)*{\scriptstyle f_3};
(-6,4)*{n};
 (-8,0)*{};(8,0)*{};
\endxy\, ,
\end{equation}
\begin{equation*}
\xy
{\ar (0,8)*{}; (0,-8)*{}};
(-3,0)*{};(3,0)*{};
\endxy\xy
{\ar (0,-8)*{}; (0,8)*{}};
(4,4)*{n};
(-3,0)*{};(6,0)*{};
\endxy = -\xy
(-3,-8)*{};(-3,8)*{} **\crv{(-3,-4) & (3,-4) & (3,4) & (-3,4)}?(0)*\dir{<};
(3,-8)*{};(3,8)*{} **\crv{(3,-4) & (-3,-4) & (-3,4) & (3,4)}?(1)*\dir{>};
(6,4)*{n};
(-7,0)*{};(8,0)*{};
\endxy + \sum_{f_1+f_2+f_3=-n-1}\xy
 (3,9)*{};(-3,9)*{} **\crv{(3,4) & (-3,4)} ?(0)*\dir{<} ?(.2)*{\bullet};
(5,6)*{\scriptstyle f_1};
 (0,0)*{\lbub};
(5,-2)*{\scriptstyle f_2};
 (-3,-9)*{};(3,-9)*{} **\crv{(-3,-4) & (3,-4)} ?(0)*\dir{<}?(.8)*{\bullet};
(5,-6)*{\scriptstyle f_3};
(-6,4)*{n};
 (-8,0)*{};(8,0)*{};
\endxy\, .
\end{equation*}
\end{enumerate}

\subsection{The $2$-category $\UcatD$}

The $2$-category $\UcatD$ is defined as $\Kar(\Ucat)$, meaning that
the categories $\Ucat(m,n)$ are replaced by their Karoubi envelopes.

In $\Ucat(m,n)$, for any $a,b\ge0$, there are additional  objects
 $\cal{E}^{(a)}\onen$ and $\cal{F}^{(b)}\onen$ defined by
\begin{equation}
\label{defEa}
 \E^{(a)}\onen\la t\ra := \left(\E^a\onen\left\la t-\frac{a(a-1)}{2}\right\ra,\idemp_a\right) =:\xy
 (0,8);(0,-8); **[grey][|(4)]\dir{-} ?(1)*[grey][|(3)]\dir{>};
 (5,0)*{n};
 (-8,0)*{n+2a};
 (2,-7)*{\scriptstyle a};
 (8,0)*{};
 (-15,0)*{};
\endxy
\end{equation}
\begin{equation*}
 \F^{(a)}\onen\la t\ra := \left(\F^a\onen\left\la t+\frac{a(a-1)}{2}\right\ra,{\idemp}'_a\right) =: \xy
 (0,-8);(0,8); **[grey][|(4)]\dir{-} ?(1)*[grey][|(3)]\dir{>};
 (5,0)*{n};
 (-8,0)*{n-2a};
 (2,7)*{\scriptstyle a};
 (8,0)*{};
 (-15,0)*{};
\endxy
\end{equation*}
where the idempotent $\idemp_a$ is defined as follows
\begin{equation*}
 \idemp_a := \delta_aD_a=\xy
(-18,-4)*{}; (9.5,-4)*{} **\dir{-};
(9.5,-4)*{}; (9.5,2)*{} **\dir{-};
(9.5,2)*{}; (-18,2)*{} **\dir{-};
(-18,2)*{}; (-18,-4)*{} **\dir{-};
(-7.5,-4)*{}; (-7.5,-8)*{} **\dir{-};
(-16,-4)*{}; (-16,-8)*{} **\dir{-};
(2.5,-4)*{}; (2.5,-8)*{} **\dir{-};
(7.5,-4)*{}; (7.5,-8)*{} **\dir{-};
{\ar (-16,2)*{}; (-16,8)*{}};
{\ar (-7.5,2)*{}; (-7.5,8)*{}};
{\ar (2.5,2)*{}; (2.5,8)*{}};
{\ar (7.5,2)*{}; (7.5,8)*{}};
(-7.5,5)*{\bullet};
(-16,5)*{\bullet};
(-11.5,6)*{\scriptstyle a-2};
(-20,6)*{\scriptstyle a-1};
(2.5,5)*{\bullet};
(-2.5,-6)*{\dots}; (-2.5,5)*{\dots};
(-4.25,-1)*{D_a};
(-19.5,0)*{};(11,0)*{};
\endxy.
\end{equation*}
Here $D_a$ is the longest braid on $a$-strands. The idempotents
${\idemp}'_a$ are obtained from $\idemp_a$ by a $180^\circ$ rotation.
We have $\cal{E}^a\onen \cong \oplus_{[a]!} \cal{E}^{(a)}\onen$ and
 $\cal{F}^{(b)}\onen \cong \oplus_{[b]!} \cal{F}^{(b)}\onen$.
Here we use the standard notation
 $$ [a] := \frac{q^a-q^{-a}}{q-q^{-1}}=q^{a-1}+q^{a-3}+\dots+q^{-a+1},\;\;
[a]!=\prod^a_{i=1} [i],
\quad \text{and}\quad
\left[ \begin{array}{c}a\\b\end{array}\right]=\frac{[a]!}{[b]![a-b]!}\, .$$

We define here some additional $2$-morphisms in $\UcatD$, whose
 degrees can be read from the shift on the right-hand side.
\begin{equation*}
\xy
 (0,-8);(0,0) **[grey][|(4)]\dir{-} ?(.6)*[grey][|(3)]\dir{>};
 (0,0);(-4,8)*{} **[grey][|(4)]\crv{(-4,1)} ?(.8)*[grey][|(3)]\dir{>};
 (0,0);(4,8)*{} **[grey][|(4)]\crv{(4,1)} ?(.8)*[grey][|(3)]\dir{>};
 (-2,8)*{\scriptstyle a};
 (6,8)*{\scriptstyle b};
 (7,-2)*{n};
 (4,-8)*{\scriptstyle a+b};
 (-8,0)*{}; (8,0)*{};
\endxy := \xy
(-4,-8);(4,-1)*{} **\crv{(-4,-4.5) & (4,-4.5)};
(4,-8);(-4,-1)*{} **\crv{(4,-4.5) & (-4,-4.5)};
(-1,-1)*{}; (-7,-1)*{} **\dir{-};
(-7,-1)*{}; (-7,5)*{} **\dir{-};
(-7,5)*{}; (-1,5)*{} **\dir{-};
(-1,5)*{}; (-1,-1)*{} **\dir{-};
(1,-1)*{}; (7,-1)*{} **\dir{-};
(7,-1)*{}; (7,5)*{} **\dir{-};
(7,5)*{}; (1,5)*{} **\dir{-};
(1,5)*{}; (1,-1)*{} **\dir{-};
{\ar (-4,5)*{}; (-4,8)*{}};
{\ar (4,5)*{}; (4,8)*{}};
(-4,2)*{\idemp_a};
(4,2)*{\idemp_b};
(-2,-8)*{\scriptstyle b};
(6,-8)*{\scriptstyle a};
 (7,-4)*{n};
(-9,0)*{};(9,0)*{};
\endxy:\E^{(a+b)}\onen\rightarrow\E^{(a)}\E^{(b)}\onen\la -ab \ra,
\end{equation*}
\begin{equation*}
\xy
 (0,0);(0,8) **[grey][|(4)]\dir{-} ?(.6)*[grey][|(3)]\dir{>};
 (-4,-8);(0,0)*{} **[grey][|(4)]\crv{(-4,-1)} ?(.3)*[grey][|(3)]\dir{>};
 (4,-8);(0,0)*{} **[grey][|(4)]\crv{(4,-1)} ?(.3)*[grey][|(3)]\dir{>};
 (-2,-8)*{\scriptstyle a};
 (6,-8)*{\scriptstyle b};
 (7,2)*{n};
 (4,8)*{\scriptstyle a+b};
 (-8,0)*{}; (8,0)*{};
\endxy := \xy
(-4,-8);(-4,-3)*{} **\dir{-};
(4,-8);(4,-3)*{} **\dir{-};
(-6,-3)*{}; (-6,3)*{} **\dir{-};
(-6,3)*{}; (6,3)*{} **\dir{-};
(6,3)*{}; (6,-3)*{} **\dir{-};
(6,-3)*{}; (-6,-3)*{} **\dir{-};
{\ar (0,3)*{}; (0,8)*{}};
(0,0)*{\idemp_{a+b}};
(-2,-8)*{\scriptstyle a};
(6,-8)*{\scriptstyle b};
 (7,6)*{n};
(-8,0)*{};(8,0)*{};
\endxy:\E^{(a)}\E^{(b)}\onen\rightarrow\E^{(a+b)}\onen\la -ab\ra,
\end{equation*}
\begin{equation*}
\xy
 (0,-8);(0,0) **[grey][|(4)]\dir{-} ?(.4)*[grey][|(3)]\dir{<};
 (0,0);(-4,8)*{} **[grey][|(4)]\crv{(-4,1)} ?(.7)*[grey][|(3)]\dir{<};
 (0,0);(4,8)*{} **[grey][|(4)]\crv{(4,1)} ?(.7)*[grey][|(3)]\dir{<};
 (-2,8)*{\scriptstyle a};
 (6,8)*{\scriptstyle b};
 (7,-2)*{n};
 (5,-8)*{\scriptstyle a+b};
 (-8,0)*{}; (8,0)*{};
\endxy := \xy
(-4,8);(-4,3)*{} **\dir{-};
(4,8);(4,3)*{} **\dir{-};
(-6,-3)*{}; (-6,3)*{} **\dir{-};
(-6,3)*{}; (6,3)*{} **\dir{-};
(6,3)*{}; (6,-3)*{} **\dir{-};
(6,-3)*{}; (-6,-3)*{} **\dir{-};
{\ar (0,-3)*{}; (0,-8)*{}};
(0,0)*{{\idemp}'_{a+b}};
(-2,8)*{\scriptstyle a};
(6,8)*{\scriptstyle b};
 (7,-6)*{n};
(-8,0)*{};(8,0)*{};
\endxy:\F^{(a+b)}\onen\rightarrow\F^{(a)}\F^{(b)}\onen\la -ab \ra,
\end{equation*}
\begin{equation*}
\xy
 (0,0);(0,8) **[grey][|(4)]\dir{-} ?(.4)*[grey][|(3)]\dir{<};
 (-4,-8);(0,0)*{} **[grey][|(4)]\crv{(-4,-1)} ?(.2)*[grey][|(3)]\dir{<};
 (4,-8);(0,0)*{} **[grey][|(4)]\crv{(4,-1)} ?(.2)*[grey][|(3)]\dir{<};
 (-2,-8)*{\scriptstyle a};
 (6,-8)*{\scriptstyle b};
 (7,2)*{n};
 (5,8)*{\scriptstyle a+b};
 (-8,0)*{}; (8,0)*{};
\endxy := \xy
(-4,8);(4,1)*{} **\crv{(-4,4.5) & (4,4.5)};
(4,8);(-4,1)*{} **\crv{(4,4.5) & (-4,4.5)};
(-1,1)*{}; (-7,1)*{} **\dir{-};
(-7,1)*{}; (-7,-5)*{} **\dir{-};
(-7,-5)*{}; (-1,-5)*{} **\dir{-};
(-1,-5)*{}; (-1,1)*{} **\dir{-};
(1,1)*{}; (7,1)*{} **\dir{-};
(7,1)*{}; (7,-5)*{} **\dir{-};
(7,-5)*{}; (1,-5)*{} **\dir{-};
(1,-5)*{}; (1,1)*{} **\dir{-};
{\ar (-4,-5)*{}; (-4,-8)*{}};
{\ar (4,-5)*{}; (4,-8)*{}};
(-4,-2)*{{\idemp}'_a};
(4,-2)*{{\idemp}'_b};
(-2,8)*{\scriptstyle b};
(6,8)*{\scriptstyle a};
 (7,4)*{n};
(-9,0)*{};(9,0)*{};
\endxy:\F^{(a)}\F^{(b)}\onen\rightarrow\E^{(a+b)}\onen\la -ab \ra,
\end{equation*}

\begin{equation*}
\xy
 (-4,-4)*{};(4,-4)*{} **[grey][|(4)]\crv{(-4,4) & (4,4)} ?(.15)*[grey][|(3)]\dir{>} ?(.9)*[grey][|(3)]\dir{>};
 (6,2)*{n};
 (-2,-5)*{\scriptstyle a};
 (-8,0)*{};(8,0)*{};
\endxy := \xy
(-1,3)*{}; (-7,3)*{} **\dir{-};
(-7,3)*{}; (-7,-3)*{} **\dir{-};
(-7,-3)*{}; (-1,-3)*{} **\dir{-};
(-1,-3)*{}; (-1,3)*{} **\dir{-};
(1,3)*{}; (7,3)*{} **\dir{-};
(7,3)*{}; (7,-3)*{} **\dir{-};
(7,-3)*{}; (1,-3)*{} **\dir{-};
(1,-3)*{}; (1,3)*{} **\dir{-};
(-4,-8)*{}; (-4,-3)*{} **\dir{-};
(4,-3)*{}; (4,-8)*{} **\dir{-}?(1)*\dir{>};
 (-4,3)*{};(4,3)*{} **\crv{(-4,9) & (4,9)} ?(.55)*\dir{>};
 (-2,-8)*{\scriptstyle a};
(-4,0)*{\idemp_a};
(4,0)*{{\idemp}'_a};
 (7,6)*{n};
 (-8,0)*{};(9,0)*{};
\endxy:\E^{(a)}\F^{(a)}\onen\rightarrow\onen\la a^2-an \ra,
\end{equation*}
\begin{equation*}
\xy
 (4,-4)*{};(-4,-4)*{} **[grey][|(4)]\crv{(4,4) & (-4,4)} ?(.15)*[grey][|(3)]\dir{>} ?(.9)*[grey][|(3)]\dir{>};
 (6,2)*{n};
 (-2,-5)*{\scriptstyle a};
 (-8,0)*{};(8,0)*{};
\endxy := \xy
(-1,3)*{}; (-7,3)*{} **\dir{-};
(-7,3)*{}; (-7,-3)*{} **\dir{-};
(-7,-3)*{}; (-1,-3)*{} **\dir{-};
(-1,-3)*{}; (-1,3)*{} **\dir{-};
(1,3)*{}; (7,3)*{} **\dir{-};
(7,3)*{}; (7,-3)*{} **\dir{-};
(7,-3)*{}; (1,-3)*{} **\dir{-};
(1,-3)*{}; (1,3)*{} **\dir{-};
(-4,-3)*{}; (-4,-8)*{} **\dir{-}?(1)*\dir{>};
(4,-3)*{}; (4,-8)*{} **\dir{-};
 (-4,3)*{};(4,3)*{} **\crv{(-4,9) & (4,9)} ?(.45)*\dir{<};
 (6,-8)*{\scriptstyle a};
(-4,0)*{{\idemp}'_a};
(4,0)*{\idemp_a};
 (7,6)*{n};
 (-8,0)*{};(9,0)*{};
\endxy:\F^{(a)}\E^{(a)}\onen\rightarrow\onen\la a^2+an\ra,
\end{equation*}
\begin{equation*}
\xy
 (4,4)*{};(-4,4)*{} **[grey][|(4)]\crv{(4,-4) & (-4,-4)} ?(.15)*[grey][|(3)]\dir{>} ?(.9)*[grey][|(3)]\dir{>};
 (6,-2)*{n};
 (-2,5)*{\scriptstyle a};
 (-8,0)*{};(8,0)*{};
\endxy := \xy
(-1,3)*{}; (-7,3)*{} **\dir{-};
(-7,3)*{}; (-7,-3)*{} **\dir{-};
(-7,-3)*{}; (-1,-3)*{} **\dir{-};
(-1,-3)*{}; (-1,3)*{} **\dir{-};
(1,3)*{}; (7,3)*{} **\dir{-};
(7,3)*{}; (7,-3)*{} **\dir{-};
(7,-3)*{}; (1,-3)*{} **\dir{-};
(1,-3)*{}; (1,3)*{} **\dir{-};
(-4,3)*{}; (-4,8)*{} **\dir{-}?(1)*\dir{>};
(4,3)*{}; (4,8)*{} **\dir{-};
 (-4,-3)*{};(4,-3)*{} **\crv{(-4,-9) & (4,-9)} ?(.45)*\dir{<};
 (6,8)*{\scriptstyle a};
(-4,0)*{\idemp_a};
(4,0)*{{\idemp}'_a};
 (7,-6)*{n};
 (-8,0)*{};(9,0)*{};
\endxy : \onen\rightarrow\E^{(a)}\F^{(a)}\onen\la a^2-an \ra,
\end{equation*}
\begin{equation*}
\xy
 (-4,4)*{};(4,4)*{} **[grey][|(4)]\crv{(-4,-4) & (4,-4)} ?(.15)*[grey][|(3)]\dir{>} ?(.9)*[grey][|(3)]\dir{>};
 (6,-2)*{n};
 (-2,5)*{\scriptstyle a};
 (-8,0)*{};(8,0)*{};
\endxy := \xy
(-1,3)*{}; (-7,3)*{} **\dir{-};
(-7,3)*{}; (-7,-3)*{} **\dir{-};
(-7,-3)*{}; (-1,-3)*{} **\dir{-};
(-1,-3)*{}; (-1,3)*{} **\dir{-};
(1,3)*{}; (7,3)*{} **\dir{-};
(7,3)*{}; (7,-3)*{} **\dir{-};
(7,-3)*{}; (1,-3)*{} **\dir{-};
(1,-3)*{}; (1,3)*{} **\dir{-};
(-4,8)*{}; (-4,3)*{} **\dir{-};
(4,3)*{}; (4,8)*{} **\dir{-}?(1)*\dir{>};
 (-4,-3)*{};(4,-3)*{} **\crv{(-4,-9) & (4,-9)} ?(.55)*\dir{>};
 (-2,8)*{\scriptstyle a};
(-4,0)*{{\idemp}'_a};
(4,0)*{\idemp_a};
 (7,-6)*{n};
 (-8,0)*{};(9,0)*{};
\endxy:\onen\rightarrow\F^{(a)}\E^{(a)}\onen\la a^2+an \ra.
\end{equation*}

\subsection{The split Grothendieck group $K_0(\Udot)$}

Lusztig's canonical basis $\Bbb B$ of $\UA$ is
\begin{enumerate}[(i)]
     \item $E^{(a)}F^{(b)}1_{n} \quad $ for $a,b\geq 0$,
     $n\in\Z$, $n\leq b-a$,
     \item $F^{(b)}E^{(a)}1_{n} \quad$ for $a$,$b\geq 0$, $n\in\Z$,
     $n\geq
     b-a$,
\end{enumerate}
where $E^{(a)}F^{(b)}1_{b-a}=F^{(b)}E^{(a)}1_{b-a}$. Let
 $_m\Bbb B_n$ be set of elements in $\Bbb B$ belonging to $1_m(\UA)1_n$.
The defining relations for the algebra $\UA$
all lift to explicit isomorphisms in $\UcatD$ (see \cite[Theorem 5.1, Theorem 5.9]{KLMS}) after associating to
 each $x \in \Bbb B$  a $1$-morphism in $\UcatD$ as follows:
\begin{equation} \label{eq_basis}
  x \mapsto \cal{E}(x) := \left\{
\begin{array}{cl}
  \cal{E}^{(a)}\cal{F}^{(b)}\onen & \text{if $x=E^{(a)}F^{(b)}1_n$,} \nn \\
  \cal{F}^{(b)}\cal{E}^{(a)}\onen & \text{if $x=F^{(b)}E^{(a)}1_n$.} \nn
\end{array}
  \right.
\end{equation}

Let $\B=\{\E(x)\mid x\in\Bbb B\}$
and $_m\B_n =\{\E(x)\mid x\in\, _m{\Bbb B_n}\}$.
For later use, we note that there is a bijection
\begin{gather}
  \label{e4}
  \mathcal{E}\col \mathbb{B}\overset{\cong}{\longrightarrow} \mathcal{B}.
\end{gather}

\begin{thm}\label{K0}
(\cite{KLMS})
Every $1$-morphism $f$ in $\UcatD$ is isomorphic to a unique sum of elements in $\B$. The map
\begin{eqnarray} \label{def_gamma}
\gamma\maps  \UA & \longrightarrow& K_0(\UcatD) \\
   x & \mapsto & [\cal{E}(x)]_{\cong}, \nn
\end{eqnarray}
is an isomorphism of linear categories.
\end{thm}

\subsection{The $2$-category $\U^*$}

The $2$-category $\U^*$ is defined as follows.  The objects and
$1$-morphisms of $\U^*$ are the same as those of $\U$.  Given a pair
of $1$-morphisms $f,g\col n\to m$, the abelian group $\U^*(n,m)(f,g)$
is defined by
\begin{gather*}
  \U^*(n,m)(f,g):= \bigoplus_{t\in\Z}\U(n,m)(f,g\la t\ra).
\end{gather*}
The category $\U^*(n,m)$ is additive and enriched over $\Z$-graded
abelian groups.  Alternatively, the linear category $\U^*(n,m)$ is
obtained from $\U(n,m)$ by adding a family of natural isomorphisms
$f\rightarrow f\la1\ra$ for each object $f$ of the category $\U(n,m)$.

In $\U^*(n,m)$ an object $f$ and its translation $f\la t\ra$
are isomorphic via the $2$-isomorphism
\begin{gather*}
  1_{f}\in \U(n,m)(f,f\la 0\ra)=\U(n,m)(f,(f\la t\ra) \la -t \ra)\subset\U^*(n,m)(f,f\la t\ra).
\end{gather*}
The inverse of the isomorphism $1_f \maps f \to f \la t \ra$ is given by
\begin{gather*}
  1_{f}\la t \ra \in \U(n,m)(f\la t \ra,f\la t\ra)=\U(n,m)(f\la t\ra,(f\la 0\ra) \la t \ra)\subset\U^*(n,m)(f\la t\ra ,f).
\end{gather*}
These isomorphisms $f \cong f\la t\ra$ make the Grothendieck group $K_0(\U^{\ast})$ into a $\Z$-module, rather than $\Z[q,q^{-1}]$-module since $[f]_{\cong}=[f\la t \ra ]_{\cong}$ in $\U^{\ast}$.

The horizontal composition in $\U$ induces horizontal composition in
$\U^*$.  It follows that the $\U^*(n,m)$, $n,m\in\Z$, form an additive
$2$-category.

The Karoubi envelope $\Kar(\U^*)$ will be denoted by $\dotU^*$, which
is equivalent as an additive $2$-category to the additive $2$-category
obtained from $\Udot$ by applying the procedure where we obtained
$\U^*$ from $\U$, i.e.,
\begin{gather*}
  \dotU^*(n,m)(f,g) = \bigoplus_{t\in\Z}\dotU(n,m)(f,g\la t\ra).
\end{gather*}

\section{Strongly upper-triangular basis for $\Udot$}
Before we define a basis of the hom-modules in $\Udot$, we need
some basic facts about symmetric functions.

\subsection{Symmetric functions} \label{subsec_sympoly}
For a partition $\lambda=(\lambda_1,\lambda_2,\dots,\lambda_a)$
with $\lambda_1\geq \lambda_2\geq\dots\geq\lambda_a\geq 0$ let
$|\lambda| := \sum_{i=1}^a\lambda_i$.
We denote by $P(a,b)$  the set of all partitions $\lambda$ with at most $a$ parts (i.e.\ with $\lambda_{a+1}=0$) such that $\lambda_1\leq b$. Moreover,
the set of all partitions with at most $a$
parts (i.e.\ the set $P(a,\infty)$) we denote simply by $P(a)$.  We will denote the collection of all partitions of arbitrary size by $P$.
The dual (conjugate) partition of  $\lambda$ is the partition
$\lambda^t=(\lambda^t_1,\lambda^t_2,\ldots)$ with
$\lambda^t_j=\sharp\{i\;|\;\lambda_i\ge j\}$ which is given by reflecting
the Young diagram of $\lambda$ along the diagonal.
For a partition $\lambda\in P(a,b)$ we define the \emph{complementary partition} $\lambda^c=(b-\lambda_a,\dots,b-\lambda_2,b-\lambda_1)$.
Finally we define $\hat\lambda:=(\lambda^c)^t$.

%\subsection{Schur polynomials}
Let us denote by $S_k$ the symmetric group and $\Sym_k=\Z[x_1, \dots,x_k]^{S_k}$
the ring of symmetric polynomials.
For any $k$-tuple $\mu$ of natural numbers we define
$$a_\mu=\sum_{\sigma \in S_k} {\text{\rm sign}}(\sigma) x^{\sigma(\mu)} \in \Sym_k,$$
where $x^\mu=x_1^{\mu_1} x_2^{\mu_2}\dots x_k^{\mu_k}$.
Clearly, $a_\mu=0$ if $\mu_i=\mu_j$ for $i\neq j$, and it changes the sign
after permuting $\mu_i$ and $\mu_{i+1}$.
The Schur polynomials $s_\mu\in \Sym_k$ are then defined as
\begin{equation} \label{def_schur}
s_\mu:=\frac{a_{\mu+\delta}}{a_{\delta}}=
\frac{\det(x_i^{\mu_j+k-j})_{1\leq i,j\leq k}}{\det(x_i^{k-j})_{1\leq i,j\leq k}}.
\end{equation}
Note that this definition make sense for any $k$-tuple $(\mu_1, \dots,
\mu_k)\in \Z^k$ with $\mu_i\geq i-k$. We will make use of this fact in
what follows.

Let $\Sym$ be the ring of symmetric functions, defined as a subring of
the inverse limit of the system $(\Sym_k)_{k\ge0}$ (see
e.g. \cite{McD}).  The elementary symmetric functions
$$
\elem_j:=\sum_{1\leq i_1<i_2<\dots<i_j} x_{i_1}\dots x_{i_j}\,,
$$
or the complete symmetric functions
$$
h_j:=\sum_{1\leq i_1\leq i_2\leq\dots\leq i_j} x_{i_1}\dots x_{i_j}
$$
generate $\Sym$ as a free commutative ring, i.e.,
$\Sym=\Z[\elem_1,\elem_2,\ldots]=\Z[h_1,h_2,\ldots]$.
The power sum symmetric functions are defined by
$$
p_t:=\sum_i x_i^t\,.
$$
We have the Newton identities:
\begin{gather*}
  j\elem_j=\sum_{i=1}^j(-1)^{i-1}\elem_{j-i}p_i,\\
  jh_j=\sum_{i=1}^jh_{j-i}p_i\, .
\end{gather*}

For $t,j\ge0$, set
$$
\elem_{t,j}:=\sum_{1\leq i_1<i_2<\dots<i_j} x_{i_1}^t\dots x_{i_j}^t\,.
$$
Then we have
\begin{gather*}
j\elem_{t,j}=\sum_{i=1}^j(-1)^{i-1}\elem_{t,j-i}p_{it}\,.
\end{gather*}
We also have for $t,j\ge0$
\begin{gather*}
  \elem_{t,j}=\frac{1}{j!}\sum_{\lambda \in P(j)}M_{j,\lambda }p_t^{\lambda _1}p_{2t}^{\lambda _2}p_{3t}^{\lambda _3}\cdots
\end{gather*}
where
\begin{gather} \label{eq_M}
  M_{j,\lambda }=j!(-1)^{|\lambda |-l(\lambda )}/\prod_{i\ge 1}i^{m_i(\lambda )}m_i(\lambda ).
\end{gather}
Here $m_i(\lambda )$ denotes the number of parts in $\lambda $ of size $i$.

By Pieri's formulas for a partition $\lambda$ we have
\begin{equation}\label{pieri}
s_\lambda \elem_j=\sum_{\mu\in \lambda\otimes 1^j} s_\mu,
\end{equation}
where $\lambda\otimes 1^j$ is the set of partitions obtained
by adding $j$ boxes to $\lambda$ at most one per row.

The ring
$\Sym$ has a Hopf algebra structure with comultiplication
\begin{gather*}
  \Delta \col \Sym \longrightarrow \Sym \otimes \Sym
\end{gather*}
 given by
\begin{gather*}
  \Delta (s_\lambda )=\sum_{\mu,\nu \in P} N^\lambda _{\mu \nu }s_\mu \otimes s_\nu ,
\end{gather*}
where $N^\lambda_{\mu\nu}$ are the Littlewood-Richardson coefficients,
counit
\begin{gather*}
  \epsilon \col \Sym \longrightarrow \modZ ,\quad s_\lambda \mapsto \delta _{\lambda ,0},
\end{gather*}
and antipode
\begin{gather*}
  S \col \Sym \longrightarrow \Sym ,\quad s_\lambda \mapsto (-1)^{|\lambda |}s_{\lambda ^t}.
\end{gather*}
We will also use the standard notation for the skew Schur functions
$$s_{\nu/\mu}=\sum_\lambda N^{\nu}_{\lambda \mu} s_\lambda\, .
$$

We will need the following two bases for the ring $\Sym$ of symmetric
functions:
\begin{gather}
\{s_\lambda\;|\;\lambda\in P\},
\end{gather}
\begin{gather}
\label{ba2}
\{\elem_{t_1,j_1}\dots \elem_{t_s,j_s}\;|\;s\geq 0; t_1>\dots>t_s\geq 1; j_1,\dots,j_s\geq 1\}.
\end{gather}

\subsection{Basis for $2$-morphisms in $\Udot$}

For any partition $\lambda\in P(a)$ let us define
\begin{equation}
\label{Elam}
\E^{(a)}_{\lambda}\onen
=\E^{(a)}_{s_{\lambda}}\onen=
\xy
 (0,8);(0,-8); **[grey][|(4)]\dir{-} ?(1)*[grey][|(3)]\dir{>};
 (5,3)*{n};
 (2,-7)*{\scriptstyle a};
 (7,0)*{};
 (-3,0)*{};
(0,0)*{\bigb{\lambda}};
  \endxy
  =\xy
 (0,8);(0,-8); **[grey][|(4)]\dir{-} ?(1)*[grey][|(3)]\dir{>};
 (5,3)*{n};
 (2,-7)*{\scriptstyle a};
 (7,0)*{};
 (-3,0)*{};
(0,0)*{\bigb{s_{\lambda}}};
  \endxy:=
\xy
 (0,-6);(0,6)*{} **\crv{(-18,-5) & (-18,5)}?(.5)*{\bullet};
 (-17,1)*{\scriptstyle \lambda_1+};
 (-17,-1)*{\scriptstyle a-1};
 (0,-6);(0,6)*{} **\crv{(-6,-4) & (-6,4)}?(.5)*{\bullet};
 (-8,1)*{\scriptstyle \lambda_2+};
 (-8,-1)*{\scriptstyle a-2};
 (0,-6);(0,6)*{} **\crv{(6,-4) & (6,4)}?(.5)*{\bullet};
 (8,1)*{\scriptstyle \lambda_{a-1}};
 (8,-1)*{\scriptstyle +1};
 (0,-6);(0,6)*{} **\crv{(18,-5) & (18,5)}?(.5)*{\bullet};
 (16,1)*{\scriptstyle \lambda_a};
 (0,0)*{\dots};
 (0,6);(0,11) **[grey][|(4)]\dir{-}?(1)*[grey][|(3)]\dir{>};
 (0,-6);(0,-11) **[grey][|(4)]\dir{-};
 (9,8)*{n};
 (2,-11)*{\scriptstyle a};
 (18,0)*{};
 (-21,0)*{};
\endxy
 :\E^{(a)}\onen\rightarrow\E^{(a)}\onen\la 2|\lambda| \ra
\end{equation}
and analogously,
\begin{equation*}
\F^{(a)}_{\lambda}\onen=
\F^{(a)}_{s_{\lambda}}\onen=
\xy
 (0,-8);(0,8); **[grey][|(4)]\dir{-} ?(1)*[grey][|(3)]\dir{>};
 (5,3)*{n};
 (2,7)*{\scriptstyle a};
 (7,0)*{};
 (-3,0)*{};
(0,0)*{\bigb{\lambda}};
  \endxy=\xy
 (0,-8);(0,8); **[grey][|(4)]\dir{-} ?(1)*[grey][|(3)]\dir{>};
 (5,3)*{n};
 (2,7)*{\scriptstyle a};
 (7,0)*{};
 (-3,0)*{};
(0,0)*{\bigb{s_{\lambda}}};
  \endxy:=
\xy
 (0,-6);(0,6)*{} **\crv{(-18,-5) & (-18,5)}?(.5)*{\bullet};
 (-16,1)*{\scriptstyle \lambda_a};
 (0,-6);(0,6)*{} **\crv{(-6,-4) & (-6,4)}?(.5)*{\bullet};
 (-8,1)*{\scriptstyle \lambda_{a-1}};
 (-8,-1)*{\scriptstyle +1};
 (0,-6);(0,6)*{} **\crv{(6,-4) & (6,4)}?(.5)*{\bullet};
 (8,1)*{\scriptstyle \lambda_2+};
 (8,-1)*{\scriptstyle a-2};
 (0,-6);(0,6)*{} **\crv{(18,-5) & (18,5)}?(.5)*{\bullet};
 (17,1)*{\scriptstyle \lambda_1+};
 (17,-1)*{\scriptstyle a-1};
 (0,0)*{\dots};
 (0,6);(0,11) **[grey][|(4)]\dir{-};
 (0,-6);(0,-11) **[grey][|(4)]\dir{-}?(1)*[grey][|(3)]\dir{>};
 (9,8)*{n};
 (2,11)*{\scriptstyle a};
 (-18,0)*{};
 (21,0)*{};
\endxy
:\F^{(a)}\onen\rightarrow\F^{(a)}\onen\la 2|\lambda| \ra.
\end{equation*}
These notations for $s_\lambda$ extends linearly to any elements of
$\Sym_a$.
The correspondence
\begin{gather*}
  \Sym_a\ni y \mapsto \E^{(a)}_{y}\onen
\end{gather*}
is multiplicative, i.e.,  we have for $y,z\in \Sym_a$
\begin{equation*}
\xy
 (0,8);(0,-8); **[grey][|(4)]\dir{-} ?(1)*[grey][|(3)]\dir{>};
 (5,4)*{n};
 (2,-7)*{\scriptstyle a};
 (7,0)*{};
 (-3,0)*{};
(0,3)*{\bigb{y}};
(0,-3)*{\bigb{z}};
\endxy=
\xy
 (0,8);(0,-8); **[grey][|(4)]\dir{-} ?(1)*[grey][|(3)]\dir{>};
 (5,4)*{n};
 (2,-7)*{\scriptstyle a};
 (7,0)*{};
 (-5,0)*{};
(0,0)*{\bigb{yz}};
\endxy.
\end{equation*}

We will use the following algebra isomorphisms
 $b^+,b^-\col \Sym\longrightarrow \End(\onen)$  defined by
\begin{equation}\label{defbh}
   b^+(h_i)= \xy
(0,0)*{\rbub};
 (5,4)*{n};
 (6,0)*{};
 (5,-2)*{\scriptstyle i};
\endxy, \quad
 b^-(h_i)= \xy
(0,0)*{\lbub};
 (5,4)*{n};
 (6,0)*{};
 (5,-2)*{\scriptstyle i};
\endxy\, .
\end{equation}
It follows
\begin{equation}\label{defbe}
b^+(\elem_i)=(-1)^i\xy
(0,0)*{\lbub};
 (5,4)*{n};
 (6,0)*{};
 (5,-2)*{\scriptstyle i};
\endxy ,
 \quad   b^-(\elem_i)=(-1)^i\xy
(0,0)*{\rbub};
 (5,4)*{n};
 (6,0)*{};
 (5,-2)*{\scriptstyle i};
\endxy\, .
\end{equation}

For $n\in\Z$, $a,b \geq 0$, $\delta\in\Z$, let us define the following
sets:
\begin{gather*}
 B^+(n,a,b,\delta) := \\
 \{f^{b,a,i,j}_{\lambda,\mu,\nu,\sigma,\tau}\onen |
0\leq i,j\leq \min(a,b), \delta=i-j, \lambda \in P(a-j), \mu \in P(b-j), \nu \in P(i), \sigma \in P(j), \tau \in P \},
\end{gather*}
\begin{gather*}
B^-(n,a,b,\delta) := \\
\{g^{a,b,i,j}_{\lambda,\mu,\nu,\sigma,\tau}\onen |
0\leq i,j\leq \min(a,b), \delta=i-j, \lambda \in P(a-j), \mu \in P(b-j), \nu \in P(i), \sigma \in P(j), \tau \in P\},
\end{gather*}
where
\begin{equation*}
 f^{b,a,i,j}_{\lambda,\mu,\nu,\sigma,\tau}\onen := \xy
 (-8,14);(-8,-14); **[grey][|(4)]\dir{-} ?(.85)*[grey][|(3)]\dir{<}?(.5)*{\bigb{\mu}};
 (8,-14);(8,14); **[grey][|(4)]\dir{-} ?(.85)*[grey][|(3)]\dir{<}?(.5)*{\bigb{\lambda}};
 (8,10)*{};(-8,10)*{} **[grey][|(4)]\crv{(8,2) & (-8,2)} ?(.25)*[grey][|(3)]\dir{<}?(.5)*{\bigb{\nu}};
 (-8,-10)*{};(8,-10)*{} **[grey][|(4)]\crv{(-8,-2) & (8,-2)} ?(.25)*[grey][|(3)]\dir{<}?(.5)*{\bigb{\sigma}};
 (-13,14)*{\scriptstyle b+i-j};
 (13,14)*{\scriptstyle a+i-j};
 (-10,-14)*{\scriptstyle b};
 (10,-14)*{\scriptstyle a};
 (5,-7)*{\scriptstyle j};
 (-5,7)*{\scriptstyle i};
 (18,2)*{b^+(s_{\tau})};
 (16,-8)*{n};
 (24,0)*{};
\endxy,\quad\quad g^{a,b,i,j}_{\lambda,\mu,\nu,\sigma,\tau}\onen  := \xy
 (-8,14);(-8,-14); **[grey][|(4)]\dir{-} ?(.9)*[grey][|(3)]\dir{>}?(.5)*{\bigb{\lambda}};
 (8,-14);(8,14); **[grey][|(4)]\dir{-} ?(.9)*[grey][|(3)]\dir{>}?(.5)*{\bigb{\mu}};
 (8,10)*{};(-8,10)*{} **[grey][|(4)]\crv{(8,2) & (-8,2)} ?(.3)*[grey][|(3)]\dir{>}?(.5)*{\bigb{\nu}};
 (-8,-10)*{};(8,-10)*{} **[grey][|(4)]\crv{(-8,-2) & (8,-2)} ?(.3)*[grey][|(3)]\dir{>}?(.5)*{\bigb{\sigma}};
 (-13,14)*{\scriptstyle a+i-j};
 (13,14)*{\scriptstyle b+i-j};
 (-10,-14)*{\scriptstyle a};
 (10,-14)*{\scriptstyle b};
 (5,-7)*{\scriptstyle j};
 (-5,7)*{\scriptstyle i};
 (18,2)*{b^+(s_{\tau})};
 (16,-8)*{n};
 (24,0)*{};
\endxy
\end{equation*}

\begin{prop}[Proposition 5.15 of \cite{KLMS}]
\label{BAZA}
Let  $a,b\ge0$, $\delta\in\Z$. The $\Z$-module
\begin{equation*}
 \dot{\U}(\F^{(b)}\E^{(a)}\onen,\F^{(b+\delta)}\E^{(a+\delta)}\onen\la t\ra)
\end{equation*}
is free with basis given by the elements of $B^+(n,a,b,\delta)$ of
degree $t$.  Similarly,  the $\Z$-module
\begin{equation*}
 \dot{\U}(\E^{(a)}\F^{(b)}\onen,\E^{(a+\delta)}\F^{(b+\delta)}\onen\la t\ra)
\end{equation*}
is free with basis given by the elements of $B^-(n,a,b,\delta)$
of degree $t$.
\end{prop}

\begin{cor}\label{3.4}
The minimal degree of an element of $B^+(n,a,b,\delta)$ for $n\geq b-a$, and of $B^-(n,a,b,\delta)$ for $n\leq b-a$ is at least $\delta^2$.
The only degree $0$ element in $B^+(n,a,b,0)$ for $n\geq b-a$ (in $B^-(n,a,b,0)$ for $n\leq b-a$) is the identity.
\end{cor}

\begin{proof}
Let $n\leq b-a$. Polynomials and bubbles can only increase the degree, so it is enough to compute the degree of $g^{a,b,i,j}_{0,0,0,0,0}\onen$, which is
$$-j(b-j)-j(a-j)-i(b-j)-i(a-j)+i^2+j^2-i(n-2(b-j))-j(n-2(b-j))=$$
$$= -\delta(a+b-2j)+i^2+j^2-\delta(n-2(b-j))= -\delta(n+a-b)+i^2+j^2\geq i^2+j^2\geq\delta^2\, .$$
For $\delta=0$ it is $0$ if and only if $i=0$.
The case when $n\geq b-a$ is similar.
\end{proof}

\begin{cor}
\label{2con}
\label{deg}
Let $t,t'\in\Z$ and $x,y\in \Bbb B$. If $t'-t<0$  or $t-t'=0$ and $x\neq y$, then
\begin{equation*}
 \dot\U(\E(x)\la t\ra,\E(y)\la t'\ra)=0.
\end{equation*}
The only elements in $\dot\U(\E(x)\la t\ra,\E(x)\la t\ra)$ are multiples of the identity.
\end{cor}

\subsection{Strongly upper-triangular basis for
  $\Udot(n,m)$}

\begin{prop}
  \label{r16}
  For $m,n\in \modZ $, the set $_n\mathcal{B}'_m:=\{x\la t\ra\;|\;x\in {_n\mathcal{B}_m}, t\in\modZ\}$  is a strongly upper-triangular basis for
  $\Udot(n,m)$.
\end{prop}

\begin{proof}
 Let $B:={_n\mathcal{B}'_m}$. Corollary \ref{2con} ensures that $\Udot(n,m)|_B$ is strongly upper-triangular.
 Proposition \ref{BAZA} ensures that the inclusion $\Udot(n,m)|_B\rightarrow\Udot(n,m)$ induces equivalence of additive
 categories $(\Udot(n,m)|_B)^\oplus\simeq \Udot(n,m)$.
\end{proof}

\subsection{Proof of Theorem \ref{thm1}}

Now we prove Theorem \ref{thm1}.

The homomorphism
\begin{gather*}
  h_{\Udot(n,m)}\col K_0(\Udot(n,m))\to \Tr(\Udot(n,m))
\end{gather*}
is an isomorphism since it factors as the composition of isomorphisms
\begin{gather*}
  K_0(\Udot(n,m)) \overset{\gamma^{-1}}{\ct}
  1_m\UA1_n
  = \Z\, {_n\mathbb{B}'_m}
  \overset{\Z\mathcal{E}'}{\ct}\Z\,{_n\mathcal{B}'_m}\cong \Tr(\Udot(n,m)).
\end{gather*}
Here $\gamma$ is given in Theorem \ref{K0},
$\mathcal{E}'\col{_n\mathbb{B}'_m}\to{_n\mathcal{B}'_m}$ maps $x\la
t\ra$ to $\mathcal{E}(x)\la t\ra$, and the last isomorphism follows from Propositions \ref{r24} and \ref{r16}.
Hence the linear functor $h_{\Udot}\col K_0(\Udot)\to\Tr(\Udot)$ is an
isomorphism.

Propositions \ref{r24} and \ref{r16} imply also that
$\HH_i(\Udot)=0$ for $i>0$.

$\hfill\Box$

\section{An alternative presentation of $\dotU^*(n,m)$ }
\label{sec:presentation-un-m}

In this section we give an algebraic presentation of the Karoubi envelope $\dotU^*(n,m)$ of
 $\U^*(n,m)$ obtained by reformulating results in \cite{KLMS}.

If $n+m\ge0$, then every object of $\dotU^*(n,m)$ is isomorphic to a
direct sum of finitely many copies of the objects
$\F^{(b)}\E^{(a)}\onen$ with $2(a-b)=m-n$.  In the following, we give
a presentation of the full subcategory of $\dotU^*(n,m)$ with objects
$\{\F^{(b)}\E^{(a)}\onen\;|\;2(a-b)=m-n\}$, which essentially
gives a presentation of $\dotU^*(n,m)$.  Using symmetry, one can
similarly define $\dotU^*(n,m)$ when $n+m\le0$, by giving a
presentation of the full subcategory of $\dotU^*(n,m)$ with objects
$\{\E^{(a)}\F^{(b)}\onen\;|\;2(a-b)=m-n\}$.

In the following we consider the full subcategory of $\dotU^*(n,m)$
with objects $\{\F^{(b)}\E^{(a)}\onen\;|\;2(a-b)=m-n\}$, where
we do not assume $n+m\ge0$ or $n+m\le0$.

Let us define:
$$
t_j^{(b,a)}=\xy
 (4,3)*{};(-4,3)*{} **\crv{(4,-2) & (-4,-2)} ?(.5)*{\bullet};
 (-4,-6);(-4,6); **[grey][|(4)]\dir{-} ?(1)*[grey][|(3)]\dir{>};
 (4,6);(4,-6); **[grey][|(4)]\dir{-} ?(1)*[grey][|(3)]\dir{>};
 (8,1)*{n};
 (1,-2)*{\scriptstyle j};
 (-3,-7)*{\scriptstyle b};
 (5,-7)*{\scriptstyle a};
 (-7,7)*{\scriptstyle b+1};
 (7,7)*{\scriptstyle a+1};
 (11,8)*{};
 (-7,-8)*{};
\endxy,
\quad u_j^{(b,a)}=\xy
 (4,-3)*{};(-4,-3)*{} **\crv{(4,2) & (-4,2)} ?(.5)*{\bullet};
 (-4,-6);(-4,6); **[grey][|(4)]\dir{-} ?(1)*[grey][|(3)]\dir{>};
 (4,6);(4,-6); **[grey][|(4)]\dir{-} ?(1)*[grey][|(3)]\dir{>};
 (8,1)*{n};
 (-1,2)*{\scriptstyle j};
 (-3,-7)*{\scriptstyle b};
 (5,-7)*{\scriptstyle a};
 (-7,7)*{\scriptstyle b-1};
 (7,7)*{\scriptstyle a-1};
 (11,8)*{};
 (-7,-8)*{};
\endxy,
$$
$$
d_\lambda^{(b,a)}=\xy
 (-3,-6);(-3,6); **[grey][|(4)]\dir{-} ?(1)*[grey][|(3)]\dir{>} ?(.5)*{\bigb{\lambda}};
 (3,6);(3,-6); **[grey][|(4)]\dir{-} ?(1)*[grey][|(3)]\dir{>};
 (7,1)*{n};
 (-2,-7)*{\scriptstyle b};
 (4,-7)*{\scriptstyle a};
 (10,8)*{};
 (-6,-8)*{};
\endxy,
\quad {d'}_{\lambda}^{(b,a)}=\xy
 (-3,-6);(-3,6); **[grey][|(4)]\dir{-} ?(1)*[grey][|(3)]\dir{>};
 (3,6);(3,-6); **[grey][|(4)]\dir{-} ?(1)*[grey][|(3)]\dir{>} ?(.5)*{\bigb{\lambda}};
 (7,1)*{n};
 (-2,-7)*{\scriptstyle b};
 (4,-7)*{\scriptstyle a};
 (10,8)*{};
 (-6,-8)*{};
\endxy,
$$
$$
b_\lambda^{(b,a)}={b^+}_\lambda^{(b,a)}=\xy
 (-3,-6);(-3,6); **[grey][|(4)]\dir{-} ?(1)*[grey][|(3)]\dir{>};
 (3,6);(3,-6); **[grey][|(4)]\dir{-} ?(1)*[grey][|(3)]\dir{>};
 (-9,1)*{b^+(s_{\lambda})};
 (7,1)*{n};
 (-2,-7)*{\scriptstyle b};
 (4,-7)*{\scriptstyle a};
 (10,8)*{};
 (-16,-8)*{};
\endxy,
\quad {b^-}_\lambda^{(b,a)}=\xy
 (-3,-6);(-3,6); **[grey][|(4)]\dir{-} ?(1)*[grey][|(3)]\dir{>};
 (3,6);(3,-6); **[grey][|(4)]\dir{-} ?(1)*[grey][|(3)]\dir{>};
 (-9,1)*{b^-(s_{\lambda})};
 (7,1)*{n};
 (-2,-7)*{\scriptstyle b};
 (4,-7)*{\scriptstyle a};
 (10,8)*{};
 (-16,-8)*{};
\endxy.
$$
We extend these definitions by setting $u_j^{(b,a)}=0$ for $a=0$ or
$b=0$.

\begin{thm}
  \label{r3}
Let $n,m$ be integers with $n-m\in2\Z$.
The full subcategory of $\dotU^*(n,m)$ with objects
$\{\F^{(b)}\E^{(a)}\onen\;|\;2(a-b)=m-n\}$ is generated as a linear
category by the morphisms
\begin{gather*}
   t_j^{(b,a)}\col\F^{(b)}\E^{(a)}\onen\to\F^{(b+1)}\E^{(a+1)}\onen,\quad j\ge0,\\
 u_j^{(b,a)}\col\F^{(b)}\E^{(a)}\onen\to\F^{(b-1)}\E^{(a-1)}\onen,\quad j\ge0,\\
 d_\lambda^{(b,a)}\col\F^{(b)}\E^{(a)}\onen\to\F^{(b)}\E^{(a)}\onen,\quad \lambda\in P(b),\\
 {d'}_{\lambda}^{(b,a)}\col\F^{(b)}\E^{(a)}\onen\to\F^{(b)}\E^{(a)}\onen,\quad \lambda\in P(a),\\
 b_\lambda^{(b,a)}\col\F^{(b)}\E^{(a)}\onen\to\F^{(b)}\E^{(a)}\onen,\quad \lambda\in P
\end{gather*}
for $a,b\ge0$, $2(a-b)=m-n$,
subject to the relations
\begin{gather}
\label{relnm} t_it_j+t_jt_i=0,\quad u_iu_j+u_ju_i=0,\\
\label{relnm2} u_it_j+t_ju_i= \tilde{c}_{1+\frac{m+n}{2}+i+j},\\
\label{relnm3} d'_\lambda t_i=\sum_{m\ge0,\nu\in P}N_{(m)\nu}^\lambda t_{i+m}d'_\nu ,\quad
d_\lambda t_i=\sum_{m\ge0,\nu\in P}N_{(m)\nu}^\lambda t_{i+m} d_\nu,\\ %\nn
 u_i d'_\lambda=\sum_{m\ge0,\nu\in P}N_{(m)\nu}^\lambda d'_\nu u_{i+m},\quad
 u_i d_\lambda=\sum_{m\ge0,\nu\in P}N_{(m)\nu}^\lambda d_\nu u_{i+m},\\
\label{relnm4} d_\mu d_\nu=\sum_\lambda N^\lambda_{\mu\nu}d_\lambda,\quad
 d'_\mu d'_\nu=\sum_\lambda N^\lambda_{\mu\nu}d'_\lambda,\quad
 b_\mu b_\nu=\sum_\lambda N^\lambda_{\mu\nu}b_\lambda,\\
\label{relnm5} [d'_\lambda,d_\mu]=[b_\lambda,t_i]=[b_\lambda,u_i]=[b_\lambda,d_\mu]=[b_\lambda,d'_\mu]=0,
\end{gather}
where we omit the superscripts always assuming that the last
superscript in each relation is $(b,a)$.  (For example, the first
relation in \eqref{relnm} is
$t_i^{(b+1,a+1)}t_j^{(b,a)}+t_j^{(b+1,a+1)}t_i^{(b,a)}=0$.)
In \eqref{relnm2}, $\tilde{c}_k$ for $k\in\Z$ is defined by
\begin{gather*}
  \tilde{c}_k=\sum_{i,i',i''\ge0,i+i'+i''=k}b_{(i)}(-1)^{i'}d_{(1^{i'})}d'_{(i'')}.
\end{gather*}
Note that $\tilde{c}_k=0$ for $k<0$.
\end{thm}

\begin{proof}
Here we give a sketch proof since the result is not used in the rest
of the paper.

Let $V(n,m)$ be the linear category with $\Ob(V(n,m))=\Ob(\dotU^*(n,m))$
and with generators and relations as stated
in the theorem. We define a linear functor
\begin{gather*}
  \mathcal F\col V(n,m)\rightarrow \dotU^*(n,m)|_{\{\F^{(b)}\E^{(a)}\onen\;|\;2(a-b)=m-n\}}
\end{gather*}
which maps the objects and the generating morphisms in an obvious way.
By checking the relations \eqref{relnm}--\eqref{relnm5}, one can
verify that $\mathcal{F}$ is well defined.  (For this verification, we
need identities proved in \cite{KLMS}.)
Since $\mathcal F$ is identity on objects, it suffices to prove that
$\mathcal F$ is full and faithful.

We first prove that $\mathcal F$ is full.
For $i=0,\dots,\min(a,b)$ and $\lambda\in P(i)$, we set
\begin{gather*}
t_\lambda^{(b,a)}=t_{\lambda_1+i-1}t_{\lambda_2+i-2}\dots t_{\lambda_{i-1}+1}t_{\lambda_i}^{(b,a)},\\
u_\lambda^{(b,a)}=u_{\lambda_i}u_{\lambda_{i-1}+1}\dots u_{\lambda_2+i-2}u_{\lambda_1+i-1}^{(b,a)}.
\end{gather*}
where we omit the obvious superscripts.
For $a,b\ge0$, $\delta\in\Z$, let us define the following subsets of
$V(n,m)$
\begin{gather*}
 B(n,a,b,\delta) := \\
\{t_\nu d_\mu d'_\lambda u_\sigma b_\tau^{(b,a)} \;|\;
0\leq i,j\leq \min(a,b), \delta=i-j, \lambda \in P(a-j), \mu \in P(b-j), \nu \in P(i), \sigma \in P(j), \tau \in P \}.
\end{gather*}
 Using Lemma \ref{BAZAstar} below
we see that the functor $\mathcal F$ sends $B(n,a,b,\delta)$
bijectively to the base $B^+(n,a,b,\delta)$ of
$\dotU^*(\F^{(b)}\E^{(a)}\onen,\F^{(b+\delta)}\E^{(a+\delta)}\onen)$.
Hence $\F$ is full.

To see that $\F$ is faithful, one has only to see that every morphism
can be expressed as a linear combination of elements in
$B(n,a,b,\delta)$.  This is easily checked.
\end{proof}

The following lemma is a $\dotU^*$-version of Proposition \ref{BAZA} that follows immediately.
\begin{lem}
\label{BAZAstar}
Let $a,b\ge0$, $\delta\in\Z$. The $\Z$-graded module
\begin{equation*}
 \dotU^*(\F^{(b)}\E^{(a)}\onen,\F^{(b+\delta)}\E^{(a+\delta)}\onen)
\end{equation*}
is free with basis given by the elements of $B^+(n,a,b,\delta)$.  Similarly,
the $\Z$-graded module
\begin{equation*}
 \dotU^*(\E^{(a)}\F^{(b)}\onen,\E^{(a+\delta)}\F^{(b+\delta)}\onen)
\end{equation*}
is free with basis given by the elements of $B^-(n,a,b,\delta)$.
\end{lem}

\section{The linear category $_\modZ \dU\LL$}
\label{sec:linear-category-modz}

\subsection{The $\modQ $-algebra $\UL$}

Recall that as a $\modQ $-algebra, the current algebra $\UL$ has the following
presentation.  The generators are $E_i,F_i$ and $H_i$ for $i\geq 0$,
where $X_i=X\otimes t^i$. The relations are
\begin{gather*}
  [H_i,H_j]=[E_i,E_j]=[F_i,F_j]=0,\label{rel1}\\
  [H_i,E_j]=2E_{i+j},\quad [H_i,F_j]=-2F_{i+j},\quad [E_i,F_j]=H_{i+j}
\end{gather*}
for $i,j\ge 0$.

Let $\bfU^+\LL$, $\bfU^-\LL$ and $\bfU^0\LL$ be subalgebras of
$\UL$ generated by $\{E_i\;|\;i\geq 0\}$, $\{F_i\;|\;i\geq 0\}$ and
$\{H_i\;|\;i\geq 0\}$ respectively.

For every $i\geq 0$ we define
\begin{gather*}
  |E_i|=2,\quad |F_i|=-2,\quad |H_i|=0
\end{gather*}
and extend this definition of degree to $\UL$ by setting  $|xy|=|x|+|y|$.
%% Since the relations preserve the degree, this is well-defined.

The $\modQ$-algebra $\UL$ has a basis given by the elements
  \begin{gather*}
    F_{i_1}^{a_1}\dots F_{i_r}^{a_r}H_{j_1}^{b_1}\dots H_{j_s}^{b_s}E_{k_1}^{c_1}\dots E_{k_t}^{c_t},
  \end{gather*}
where
  \begin{gather*}
    r\ge 0,\quad  i_1>\dots >i_r\ge 0,\quad  a_1,\ldots ,a_r\ge 1,\\
    s\ge 0,\quad  j_1>\dots >j_s\ge 0,\quad  b_1,\ldots ,b_s\ge 1,\\
    t\ge 0,\quad  k_1>\dots >k_t\ge 0,\quad  c_1,\ldots ,c_t\ge 1.
  \end{gather*}

\subsection{The integral form $_\modZ \bfU \LL$ of $\bfU\LL$}

For $a,i\ge 0$, define the divided powers of $E_i$ and $F_i$ by
\begin{gather*}
  E_i^{(a)}=\frac{1}{a!}E_i^a,\quad
  F_i^{(a)}=\frac{1}{a!}F_i^a.
\end{gather*}
Note that $E_i^{(0)}=F_i^{(0)}=1$.

Let $\UZL$ denote the $\modZ $-subalgebra of $\UL$
generated by $E_i^{(a)}$, $F_i^{(a)}$, $i\ge 0$, $a\ge 1$.
Set
\begin{gather*}
  \UZ^0\LL= \UZL \cap \bfU^0\LL,\\
  \UZ^+\LL= \UZL \cap \bfU^+\LL,\\
  \UZ^-\LL= \UZL \cap \bfU^-\LL,
\end{gather*}
which are $\modZ $-subalgebras of $\UZL$.

For $j\geq 0$ set $H_{j,0}=1$. For $j\ge 0$ and $b>0$, set recursively
\begin{gather*}
  bH_{j,b}=\sum_{l=1}^b(-1)^{l-1}H_{j,b-l}H_{lj},
\end{gather*}
or explicitly
\begin{gather*}
  H_{j,b}=\frac{1}{b!}\sum_{\lambda \in P(b)}M_{b,\lambda } H_j^{\lambda _1}H_{2j}^{\lambda _2}H_{3j}^{\lambda _3}\cdots
\end{gather*}
with $M_{b, \lambda}$ as in \eqref{eq_M}.

Define a ring homomorphism $\phi\col \Sym\rightarrow \bfU^0\LL$ by
$$
\phi(\elem_j)=H_{1,j}.
$$
Then we have
\begin{gather}
\label{hs}
 H_j=\phi(p_j),\quad H_{i,b}=\phi(\elem_{i,b}).
\end{gather}
Let $\UZ^P\LL\subset \UZ^0\LL$ be the image of $\phi$.

In \cite{Gar}, Garland defined an integral basis of the loop algebra.

\begin{prop}[Garland \cite{Gar}, Thm. 5.8]
  \label{r15}
  The $\modZ$-algebra $\UZL$ is a free abelian group with basis given by the elements
  \begin{gather*}
    F_{i_1}^{(a_1)}\dots F_{i_r}^{(a_r)}H_{j_1,b_1}\dots H_{j_s,b_s}E_{k_1}^{(c_1)}\dots E_{k_t}^{(c_t)},
  \end{gather*}
  where
  \begin{gather*}
    r\ge 0,\quad  i_1>\dots >i_r\ge 0,\quad  a_1,\ldots ,a_r\ge 1,\\
    s\ge 0,\quad  j_1>\dots >j_s\ge 0,\quad  b_1,\ldots ,b_s\ge 1,\\
    t\ge 0,\quad  k_1>\dots >k_t\ge 0,\quad  c_1,\ldots ,c_t\ge 1.
  \end{gather*}
\end{prop}

From this proposition and the basis for symmetric functions given in equation \eqref{ba2} the next lemma follows easily.

\begin{lem}
The $\Z$-subalgebra $\UZ^P\LL$ is free with basis given by the elements
  \begin{gather*}
    H_{j_1,b_1}\dots H_{j_s,b_s},
  \end{gather*}
where
  \begin{gather*}
    s\ge 0,\quad  j_1>\dots >j_s\ge 0,\quad  b_1,\ldots ,b_s\ge 1.
  \end{gather*}
Thus, the map $\phi\col \Sym\rightarrow \UZ^P\LL$ is an isomorphism.
\end{lem}

By using the Schur functions as a base, we obtain the following corollary.

\begin{cor}
  \label{r15_a}
  The $\Z$-algebra $\UZL$ is free with basis given by the elements
  \begin{gather*}
    F_{i_1}^{(a_1)}\dots F_{i_r}^{(a_r)}\phi(s_\tau)H_{0,b_0}E_{k_1}^{(c_1)}\dots E_{k_t}^{(c_t)},
  \end{gather*}
  where
  \begin{gather*}
    r\ge 0,\quad  i_1>\dots >i_r\ge 0,\quad  a_1,\ldots ,a_r\ge 1,\\
    \tau\in P,\\
    b_0\geq 0,\\
    t\ge 0,\quad  k_1>\dots >k_t\ge 0,\quad  c_1,\ldots ,c_t\ge 1.
  \end{gather*}
\end{cor}

\subsection{The idempotented version $\dUL$}
The idempotented version $\dUL$ of $\UL$ is the $\modQ $-linear category
defined as follows.  The objects of $\dUL$ are integers.  For
$m,n\in \modZ $, the $\modQ $-module $\dUL(m,n)$ is defined by
\begin{gather*}
  \dUL(m,n)= \UL/( \UL(H_0-m)+ (H_0-n)\UL).
\end{gather*}
The element in $\dUL(m,n)$ represented by $x\in \UL$ is denoted by
\begin{gather*}
  1_nx1_m=x1_m=1_nx,
\end{gather*}
which is zero unless $n-m=|x|$. Composition in $\dUL$ is induced by multiplication in $\UL$, i.e.,
\begin{gather*}
  (1_px1_n)(1_ny1_m)=1_pxy1_m,
\end{gather*}
for $x,y\in \UL$, $m,n,p\in \modZ $, $p-n=|x|$, $n-m=|y|$.  The identity morphism for $n\in \modZ $ is
denoted by $1_n$.

\subsection{The idempotented integral form $\dUZL$}

The idempotented version $\dUZL$ of the integral form $\UZL$ is the
linear subcategory of $\dUL$, such that $\Ob(\dUZL)=\modZ $ and
\begin{gather*}
  \dUZL(m,n)= 1_m( \UZL ) 1_n\subset  \dUL(m,n).
\end{gather*}
It follows that $\dUZL(m,n)$ has a basis as a free $\modZ$-module given by the elements
\begin{gather}
  \label{e3}
  1_mF_{i_1}^{(a_1)}\dots F_{i_r}^{(a_r)}\phi(s_\tau)E_{k_1}^{(c_1)}\dots E_{k_t}^{(c_t)}1_n,
\end{gather}
  \begin{gather*}
    r\ge 0,\quad  i_1>\dots >i_r\ge 0,\quad  a_1,\ldots ,a_r\ge 1,\\
    \tau\in P,\\
    t\ge 0,\quad  k_1>\dots >k_t\ge 0,\quad  c_1,\ldots ,c_t\ge 1,\\
    c_1+\dots +c_t-(a_1+\dots +a_r)=2(n-m).
  \end{gather*}

The above base will be used for the case $m+n\geq 0$. For $m+n\leq 0$
we will use another base, obtained from previous one by acting with
the automorphism $\Phi\col\UL\rightarrow\UL$ such that
$$
\Phi(E_i)=F_i,\quad \Phi(F_i)=E_i,\quad \Phi(H_i)=-H_i.
$$

\section{Trace of the $2$-category  $\modU^*$}

In this section we prove Theorem \ref{qqq} by computing $\Tr(\dotU^*)\cong\Tr(\U^*)$.

\subsection{Split Grothendieck group of $\dotU^*$}
Recall that the addition of translations in the definition of $\U^*$
identifies every $1$-morphism with all its shifts.  Hence, on the
level of the split Grothendieck group multiplication with $q$ becomes
a trivial operation.  Thus the split Grothendieck group of $K_0(\dotU^*)$
can be obtained from those of $\UcatD$ by setting $q=1$ and it
coincides with the integral idempotented version of $\bfU$.

\subsection{Generators of the trace of $\dotU^*$}

Let us introduce the following notation
$$
\rE^{(a)}_\mu\onen:=[\E^{(a)}_\mu\onen],\quad
\rF^{(a)}_\lambda\onen:=[\F^{(a)}_\lambda\onen],\quad
b^+(\tau):=b^+(s_\tau).
$$
Recall that $\rE^{(a)}_\mu$ is the trace of the $2$-endomorphism
of $\E^{(a)}$ given by the multiplication with the Schur function
indexed by $\mu$.

\begin{prop}(Triangular Decomposition)
\label{triangular}  For $n+m\geq 0$,
$\Tr\dotU^*(n,m)$ is a free $\modZ$-module with basis
\begin{gather*}
  \rF^{(b)}_{\mu}b^+(\tau)\rE^{(a)}_{\lambda}\onen
  \quad \text{for $n\in\Z$, $a,b\ge0$, $2(a-b)=m-n$, $\lambda\in
  P(a)$, $\mu\in P(b)$, $\tau\in P$,}
\end{gather*}
and for $n+m\leq 0$ by
\begin{gather*}
  \rE^{(a)}_{\lambda}b^+(\tau)\rF^{(b)}_{\mu}\onen
  \quad \text{for $n\in\Z$, $a,b\ge0$, $2(a-b)=m-n$, $\lambda\in P(a)$, $\mu\in P(b)$, $\tau\in P$.}
\end{gather*}
\end{prop}

Note that if $n+m=0$ then both these bases can be used.

\begin{proof}
We will prove the proposition for  $m+n\geq 0$. The other case is similar.

After bubble slides (Corollary 4.7 and Proposition 4.8 of \cite{KLMS}),
we have to show that $\Tr\dotU^*(n,m)$ has a basis given by
 $$\rF^{(b)}_{\mu}\rE^{(a)}_{\lambda}b^+(\tau)\onen
\quad {\text {for}} \quad  n\in\Z, \quad a,b\ge0, \quad 2(a-b)=m-n, \quad \lambda\in P(a), \quad\mu\in P(b).$$
Since every object of $\dotU^*(n,m)$ is isomorphic to a direct sum of elements of $_m\B_n$ (Theorem \ref{K0}), by Proposition \ref{trdec} we have
$\Tr\dotU^*(n,m)\simeq\Tr(\dotU^*(n,m)|_{_m\B_n})$.
We will use Proposition \ref{general} for the linear category
$\modC:=\dotU^*(n,m)|_{_m\B_n}$.

Let $H=\bigoplus_{x\in{_m\B_n}}\C(x,x)$. By Proposition \ref{BAZA},
$H$ is a free $\Z$-module generated by $\{f^{b,a,i,i}_{\lambda,\mu,\nu,\sigma,\tau}\onen\;|\;2(a-b)=n-m\}$.
Let a subspace $K\subset H$ be generated by $\{f^{b,a,0,0}_{\lambda,\mu,0,0,\tau}\onen\;|\;2(a-b)=n-m\}$.

Define $p\col H \rightarrow H$ as follows:
$$p\left(f^{b,a,i,i}_{\lambda,\mu,\nu,\sigma,\tau}\onen\right)
=
p\left(\xy
 (-8,14);(-8,-14); **[grey][|(4)]\dir{-} ?(.85)*[grey][|(3)]\dir{<}?(.5)*{\bigb{\mu}};
 (8,-14);(8,14); **[grey][|(4)]\dir{-} ?(.85)*[grey][|(3)]\dir{<}?(.5)*{\bigb{\lambda}};
 (8,10)*{};(-8,10)*{} **[grey][|(4)]\crv{(8,2) & (-8,2)} ?(.25)*[grey][|(3)]\dir{<}?(.5)*{\bigb{\nu}};
 (-8,-10)*{};(8,-10)*{} **[grey][|(4)]\crv{(-8,-2) & (8,-2)} ?(.25)*[grey][|(3)]\dir{<}?(.5)*{\bigb{\sigma}};
 (-10,14)*{\scriptstyle b};
 (10,14)*{\scriptstyle a};
 (-10,-14)*{\scriptstyle b};
 (10,-14)*{\scriptstyle a};
 (5,-7)*{\scriptstyle i};
 (-5,7)*{\scriptstyle i};
 (18,2)*{b^+(s_{\tau})};
 (16,-8)*{n};
 (24,0)*{};
 (-11,0)*{};
\endxy\right)
:=
\xy
 (-8,14);(-8,-14); **[grey][|(4)]\dir{-} ?(.1)*[grey][|(3)]\dir{<}?(.8)*{\bigb{\mu}};
 (8,-14);(8,14); **[grey][|(4)]\dir{-} ?(.85)*[grey][|(3)]\dir{<}?(.2)*{\bigb{\lambda}};
 (8,-4)*{};(-8,-4)*{} **[grey][|(4)]\crv{(8,-12) & (-8,-12)} ?(.25)*[grey][|(3)]\dir{<}?(.5)*{\bigb{\nu}};
 (-8,-1)*{};(8,-1)*{} **[grey][|(4)]\crv{(-8,7) & (8,7)} ?(.25)*[grey][|(3)]\dir{<}?(.5)*{\bigb{\sigma}};
 (-11,14)*{\scriptstyle b-i};
 (11,14)*{\scriptstyle a-i};
 (-11,-14)*{\scriptstyle b-i};
 (11,-14)*{\scriptstyle a-i};
 (5,2)*{\scriptstyle i};
 (-5,-7)*{\scriptstyle i};
 (18,2)*{b^+(s_{\tau})};
 (16,-8)*{n};
 (24,0)*{};
\endxy
$$
It is obvious that $p(f)=f$ for $f\in K$, and that $[p(f)]=[f]$.

Using thick calculus relations to simplify the results of iteratively applying $p$, one can show that for every $f\in H$ there is $k\ge0$ such that $p^k(f)\in K$.
Let $\pi: H \rightarrow K$ be defined by
$\pi(f)=p^k(f)$ where $k$ is chosen as above.
Since it is true for $p$, we have $\pi(f)=f$ for $f\in K$, and  $[\pi(f)]=[f]$. So conditions (1) and (2) of Proposition \ref{general} for the map $\pi$ are satisfied.
Condition (3) is proved in the following Lemma~\ref{ap1}.

Thus, $\Tr\dotU^*=\bigoplus_{n,m}\Tr\dotU^*(n,m)$ is freely generated by $\{[f^{b,a,0,0}_{\lambda,\mu,0,0,\tau}\onen]=\rF^{(b)}_{\lambda}\rE^{(a)}_{\mu}b^+(\tau)\onen\}$
as desired.
\end{proof}

\begin{lem}
\label{ap1}
For every $g\in(\dotU^*(n,m))(s,t)$ and $h\in(\dotU^*(n,m))(t,s)$, for
$s,t\in{_m\B_n}$, the equation $\pi(gh)=\pi(hg)$ holds.
\end{lem}

\begin{proof}
Let $s=\E^{(a)}\F^{(b)}\onen$ and $t=\E^{(a+i-j)}\F^{(b+i-j)}\onen$. By Proposition~\ref{BAZA} it is enough to prove the statement for generators $g=f^{b,a,i,j}_{\lambda,\mu,\nu,\sigma,\tau}\onen$ and $h=f^{b+i-j,a+i-j,i',j'}_{\lambda',\mu',\nu',\sigma',\tau'}\onen$.  Moreover, it is enough to prove it for:
\begin{enumerate}
\item $g=f^{b,a,i,i}_{\lambda,\mu,\nu,\sigma,\tau}\onen$ and $h=f^{b,a,0,0}_{\lambda',\mu',0,0,\tau'}\onen$,
\item $g=f^{b,a,i,j}_{\lambda,\mu,\nu,\sigma,\tau}\onen$ and $h=f^{b+i-j,a+i-j,i-j,0}_{0,0,\nu',0,0}\onen$ and
\item $g=f^{b,a,i,j}_{\lambda,\mu,\nu,\sigma,\tau}\onen$ and $h=f^{b+i-j,a+i-j,0,j-i}_{0,0,0,\sigma',0}\onen$,
\end{enumerate}
because all generators can be constructed from the above $h$'s by vertical composition.

Schur functions can be slid through splitters using \cite[equation (2.74)]{KLMS}.  We will not need the precise form of these relations.  In the calculations below it suffices to consider the generic formula
of the form
$$\xy
 (0,-8);(0,1) **[grey][|(4)]\dir{-} ?(.5)*{\bigb{\mu}};
 (0,1);(-4,8)*{} **[grey][|(4)]\crv{(-4,2)} ?(.8)*[grey][|(3)]\dir{>};
 (0,1);(4,8)*{} **[grey][|(4)]\crv{(4,2)} ?(.8)*[grey][|(3)]\dir{>};
 (-2,8)*{\scriptstyle a};
 (6,8)*{\scriptstyle b};
 (7,-2)*{n};
 (4,-8)*{\scriptstyle a+b};
 (-8,0)*{}; (8,0)*{};
\endxy
=
\sum_l\xy
 (0,-8);(0,-1) **[grey][|(4)]\dir{-} ?(.6)*[grey][|(3)]\dir{>};
 (0,-1);(-4,8)*{} **[grey][|(4)]\crv{(-4,0)} ?(.7)*{\bigb{\mu_l}};
 (0,-1);(4,8)*{} **[grey][|(4)]\crv{(4,0)} ?(.7)*{\bigb{\bar\mu_l}};
 (-2,8)*{\scriptstyle a};
 (6,8)*{\scriptstyle b};
 (7,-2)*{n};
 (4,-8)*{\scriptstyle a+b};
 (-8,0)*{}; (8,0)*{};
\endxy \; ,
\quad\quad\quad
\xy
 (0,-8);(0,1) **[grey][|(4)]\dir{-} ?(.5)*{\bigb{\mu}};
 (0,1);(-4,8)*{} **[grey][|(4)]\crv{(-4,2)} ?(.6)*[grey][|(3)]\dir{<};
 (0,1);(4,8)*{} **[grey][|(4)]\crv{(4,2)} ?(.6)*[grey][|(3)]\dir{<};
 (-2,8)*{\scriptstyle a};
 (6,8)*{\scriptstyle b};
 (7,-2)*{n};
 (4,-8)*{\scriptstyle a+b};
 (-8,0)*{}; (8,0)*{};
\endxy
=
\sum_l\xy
 (0,-8);(0,-1) **[grey][|(4)]\dir{-} ?(.4)*[grey][|(3)]\dir{<};
 (0,-1);(-4,8)*{} **[grey][|(4)]\crv{(-4,0)} ?(.7)*{\bigb{\mu_l}};
 (0,-1);(4,8)*{} **[grey][|(4)]\crv{(4,0)} ?(.7)*{\bigb{\bar\mu_l}};
 (-2,8)*{\scriptstyle a};
 (6,8)*{\scriptstyle b};
 (7,-2)*{n};
 (4,-8)*{\scriptstyle a+b};
 (-8,0)*{}; (8,0)*{};
\endxy \; .
$$

Without loss of generality let us assume $m\geq n$. Let us prove (1) by induction on $b$. If $b=0$ it is obvious since $\pi(gh)=gh=hg=\pi(hg)$. If $b>0$
$$p(hg)
=
p\left(\xy
 (-8,18);(-8,-18); **[grey][|(4)]\dir{-} ?(.05)*[grey][|(3)]\dir{<}?(.4)*{\bigb{\mu}}?(.8)*{\bigb{\mu'}};
 (8,-18);(8,18); **[grey][|(4)]\dir{-} ?(.05)*[grey][|(3)]\dir{<}?(.6)*{\bigb{\lambda}}?(.2)*{\bigb{\lambda'}};
 (8,6)*{};(-8,6)*{} **[grey][|(4)]\crv{(8,-2) & (-8,-2)} ?(.25)*[grey][|(3)]\dir{<}?(.5)*{\bigb{\nu}};
 (-8,-14)*{};(8,-14)*{} **[grey][|(4)]\crv{(-8,-6) & (8,-6)} ?(.25)*[grey][|(3)]\dir{<}?(.5)*{\bigb{\sigma}};
 (-10,18)*{\scriptstyle b};
 (10,18)*{\scriptstyle a};
 (-10,-18)*{\scriptstyle b};
 (10,-18)*{\scriptstyle a};
 (5,-11)*{\scriptstyle i};
 (-5,3)*{\scriptstyle i};
 (18,-2)*{b^+(s_{\tau})};
 (16,-12)*{n};
 (24,0)*{};
 (-11,0)*{};
 (18,12)*{b^+(s_{\tau'})};
\endxy\right)=
$$
$$
=\sum_{l,l'}p\left(\xy
 (-12,18);(-12,-18); **[grey][|(4)]\dir{-} ?(.05)*[grey][|(3)]\dir{<}?(.35)*{\bigb{\mu}}?(.55)*{\bigb{\mu'_l}};
 (12,-18);(12,18); **[grey][|(4)]\dir{-} ?(.05)*[grey][|(3)]\dir{<}?(.65)*{\bigb{\lambda}}?(.45)*{\bigb{\lambda'_{l'}}};
 (12,12)*{};(-12,12)*{} **[grey][|(4)]\crv{(12,6) & (-12,6)}?(.5)*{\bigb{\nu}}?(.7)*{\bigb{\bar\mu'_l}}?(.3)*{\bigb{\bar\lambda'_{l'}}};
 (-12,-14)*{};(12,-14)*{} **[grey][|(4)]\crv{(-12,-6) & (12,-6)}?(.5)*{\bigb{\sigma}}?(.3)*[grey][|(3)]\dir{<};
 (-14,18)*{\scriptstyle b};
 (14,18)*{\scriptstyle a};
 (-14,-18)*{\scriptstyle b};
 (14,-18)*{\scriptstyle a};
 (7,-11)*{\scriptstyle i};
 (22,-2)*{b^+(s_{\tau})};
 (20,-12)*{n};
 (24,0)*{};
 (-15,0)*{};
 (22,12)*{b^+(s_{\tau'})};
\endxy\right)
=
\sum_{l,l'}\xy
 (-12,18);(-12,-18); **[grey][|(4)]\dir{-} ?(.05)*[grey][|(3)]\dir{<}?(.6)*{\bigb{\mu}}?(.8)*{\bigb{\mu'_l}};
 (12,-18);(12,18); **[grey][|(4)]\dir{-} ?(.05)*[grey][|(3)]\dir{<}?(.4)*{\bigb{\lambda}}?(.2)*{\bigb{\lambda'_{l'}}};
 (12,-7)*{};(-12,-7)*{} **[grey][|(4)]\crv{(12,-13) & (-12,-13)}?(.5)*{\bigb{\nu}}?(.7)*{\bigb{\bar\mu'_l}}?(.3)*{\bigb{\bar\lambda'_{l'}}};
 (-12,-4)*{};(12,-4)*{} **[grey][|(4)]\crv{(-12,4) & (12,4)}?(.5)*{\bigb{\sigma}}?(.3)*[grey][|(3)]\dir{<};
 (-15,18)*{\scriptstyle b-i};
 (15,18)*{\scriptstyle a-i};
 (-15,-18)*{\scriptstyle b-i};
 (15,-18)*{\scriptstyle a-i};
 (7,-1)*{\scriptstyle i};
 (22,-2)*{b^+(s_{\tau})};
 (20,-12)*{n};
 (24,0)*{};
 (-15,0)*{};
 (22,12)*{b^+(s_{\tau'})};
\endxy
$$
By acting with $\pi$ on above relations and using assumption of the induction we get
$$
\pi(hg)=\sum_{l,l'}\pi\left(\xy
 (-12,18);(-12,-18); **[grey][|(4)]\dir{-} ?(.05)*[grey][|(3)]\dir{<}?(.6)*{\bigb{\mu}}?(.8)*{\bigb{\mu'_l}};
 (12,-18);(12,18); **[grey][|(4)]\dir{-} ?(.05)*[grey][|(3)]\dir{<}?(.4)*{\bigb{\lambda}}?(.2)*{\bigb{\lambda'_{l'}}};
 (12,-7)*{};(-12,-7)*{} **[grey][|(4)]\crv{(12,-13) & (-12,-13)}?(.5)*{\bigb{\nu}}?(.7)*{\bigb{\bar\mu'_l}}?(.3)*{\bigb{\bar\lambda'_{l'}}};
 (-12,-4)*{};(12,-4)*{} **[grey][|(4)]\crv{(-12,4) & (12,4)}?(.5)*{\bigb{\sigma}}?(.3)*[grey][|(3)]\dir{<};
 (-15,18)*{\scriptstyle b-i};
 (15,18)*{\scriptstyle a-i};
 (-15,-18)*{\scriptstyle b-i};
 (15,-18)*{\scriptstyle a-i};
 (7,-1)*{\scriptstyle i};
 (22,-2)*{b^+(s_{\tau})};
 (20,-12)*{n};
 (24,0)*{};
 (-15,0)*{};
 (22,12)*{b^+(s_{\tau'})};
\endxy\right)=
$$
$$
=\sum_{l,l'}\pi\left(\xy
 (-12,18);(-12,-18); **[grey][|(4)]\dir{-} ?(.05)*[grey][|(3)]\dir{<}?(.8)*{\bigb{\mu}}?(.2)*{\bigb{\mu'_l}};
 (12,-18);(12,18); **[grey][|(4)]\dir{-} ?(.05)*[grey][|(3)]\dir{<}?(.2)*{\bigb{\lambda}}?(.8)*{\bigb{\lambda'_{l'}}};
 (12,-1)*{};(-12,-1)*{} **[grey][|(4)]\crv{(12,-7) & (-12,-7)}?(.5)*{\bigb{\nu}}?(.7)*{\bigb{\bar\mu'_l}}?(.3)*{\bigb{\bar\lambda'_{l'}}};
 (-12,2)*{};(12,2)*{} **[grey][|(4)]\crv{(-12,8) & (12,8)}?(.5)*{\bigb{\sigma}}?(.3)*[grey][|(3)]\dir{<};
 (-15,18)*{\scriptstyle b-i};
 (15,18)*{\scriptstyle a-i};
 (-15,-18)*{\scriptstyle b-i};
 (15,-18)*{\scriptstyle a-i};
 (7,4)*{\scriptstyle i};
 (22,-2)*{b^+(s_{\tau})};
 (20,-12)*{n};
 (24,0)*{};
 (-15,0)*{};
 (22,12)*{b^+(s_{\tau'})};
\endxy\right)
=
\pi\left(\xy
 (-8,18);(-8,-18); **[grey][|(4)]\dir{-} ?(.05)*[grey][|(3)]\dir{<}?(.8)*{\bigb{\mu}}?(.4)*{\bigb{\mu'}};
 (8,-18);(8,18); **[grey][|(4)]\dir{-} ?(.05)*[grey][|(3)]\dir{<}?(.2)*{\bigb{\lambda}}?(.6)*{\bigb{\lambda'}};
 (8,-8)*{};(-8,-8)*{} **[grey][|(4)]\crv{(8,-14) & (-8,-14)} ?(.25)*[grey][|(3)]\dir{<}?(.5)*{\bigb{\nu}};
 (-8,1)*{};(8,1)*{} **[grey][|(4)]\crv{(-8,7) & (8,7)} ?(.25)*[grey][|(3)]\dir{<}?(.5)*{\bigb{\sigma}};
 (-11,18)*{\scriptstyle b-i};
 (11,18)*{\scriptstyle a-i};
 (-11,-18)*{\scriptstyle b-i};
 (11,-18)*{\scriptstyle a-i};
 (5,3)*{\scriptstyle i};
 (-5,-10)*{\scriptstyle i};
 (18,-2)*{b^+(s_{\tau})};
 (16,-12)*{n};
 (24,0)*{};
 (-11,0)*{};
 (18,12)*{b^+(s_{\tau'})};
\endxy\right)
$$
On the other side
$$
p(gh)
=
p\left(\xy
 (-8,18);(-8,-18); **[grey][|(4)]\dir{-} ?(.05)*[grey][|(3)]\dir{<}?(.6)*{\bigb{\mu}}?(.2)*{\bigb{\mu'}};
 (8,-18);(8,18); **[grey][|(4)]\dir{-} ?(.05)*[grey][|(3)]\dir{<}?(.4)*{\bigb{\lambda}}?(.8)*{\bigb{\lambda'}};
 (8,14)*{};(-8,14)*{} **[grey][|(4)]\crv{(8,8) & (-8,8)} ?(.25)*[grey][|(3)]\dir{<}?(.5)*{\bigb{\nu}};
 (-8,-6)*{};(8,-6)*{} **[grey][|(4)]\crv{(-8,0) & (8,0)} ?(.25)*[grey][|(3)]\dir{<}?(.5)*{\bigb{\sigma}};
 (-11,18)*{\scriptstyle b-i};
 (11,18)*{\scriptstyle a-i};
 (-11,-18)*{\scriptstyle b-i};
 (11,-18)*{\scriptstyle a-i};
 (5,-6)*{\scriptstyle i};
 (-5,12)*{\scriptstyle i};
 (18,-2)*{b^+(s_{\tau})};
 (16,-12)*{n};
 (24,0)*{};
 (-11,0)*{};
 (18,12)*{b^+(s_{\tau'})};
\endxy\right)
=
\xy
 (-8,18);(-8,-18); **[grey][|(4)]\dir{-} ?(.05)*[grey][|(3)]\dir{<}?(.8)*{\bigb{\mu}}?(.4)*{\bigb{\mu'}};
 (8,-18);(8,18); **[grey][|(4)]\dir{-} ?(.05)*[grey][|(3)]\dir{<}?(.2)*{\bigb{\lambda}}?(.6)*{\bigb{\lambda'}};
 (8,-8)*{};(-8,-8)*{} **[grey][|(4)]\crv{(8,-14) & (-8,-14)} ?(.25)*[grey][|(3)]\dir{<}?(.5)*{\bigb{\nu}};
 (-8,1)*{};(8,1)*{} **[grey][|(4)]\crv{(-8,7) & (8,7)} ?(.25)*[grey][|(3)]\dir{<}?(.5)*{\bigb{\sigma}};
 (-11,18)*{\scriptstyle b-i};
 (11,18)*{\scriptstyle a-i};
 (-11,-18)*{\scriptstyle b-i};
 (11,-18)*{\scriptstyle a-i};
 (5,3)*{\scriptstyle i};
 (-5,-10)*{\scriptstyle i};
 (18,-2)*{b^+(s_{\tau})};
 (16,-12)*{n};
 (24,0)*{};
 (-11,0)*{};
 (18,12)*{b^+(s_{\tau'})};
\endxy
$$
and by acting with $\pi$ we get $\pi(hg)=\pi(gh)$.

Let us prove (2) by induction on $b$. If $b=0$ then $h=\onen$, so it is obvious. Likewise, if $i-j=0$ then $h=\onen$. So let us assume $b>0$ and $i-j>0$.
$$
p(hg)
=
p\left(\xy
 (-8,18);(-8,-18); **[grey][|(4)]\dir{-} ?(.05)*[grey][|(3)]\dir{<}?(.4)*{\bigb{\mu}};
 (8,-18);(8,18); **[grey][|(4)]\dir{-} ?(.05)*[grey][|(3)]\dir{<}?(.6)*{\bigb{\lambda}};
  (8,15)*{};(-8,15)*{} **[grey][|(4)]\crv{(8,7) & (-8,7)} ?(.25)*[grey][|(3)]\dir{<}?(.5)*{\bigb{\nu'}};
 (8,6)*{};(-8,6)*{} **[grey][|(4)]\crv{(8,-2) & (-8,-2)} ?(.25)*[grey][|(3)]\dir{<}?(.5)*{\bigb{\nu}};
 (-8,-14)*{};(8,-14)*{} **[grey][|(4)]\crv{(-8,-6) & (8,-6)} ?(.25)*[grey][|(3)]\dir{<}?(.5)*{\bigb{\sigma}};
 (-10,18)*{\scriptstyle b};
 (10,18)*{\scriptstyle a};
 (-10,-18)*{\scriptstyle b};
 (10,-18)*{\scriptstyle a};
 (5,-11)*{\scriptstyle i};
 (-5,3)*{\scriptstyle j};
 (-4,12)*{\scriptstyle i-j};
 (18,-2)*{b^+(s_{\tau})};
 (16,-12)*{n};
 (24,0)*{};
 (-11,0)*{};
\endxy\right)
=
\xy
 (-8,18);(-8,-18); **[grey][|(4)]\dir{-} ?(.05)*[grey][|(3)]\dir{<}?(.8)*{\bigb{\mu}};
 (8,-18);(8,18); **[grey][|(4)]\dir{-} ?(.05)*[grey][|(3)]\dir{<}?(.2)*{\bigb{\lambda}};
 (8,0)*{};(-8,0)*{} **[grey][|(4)]\crv{(8,-8) & (-8,-8)} ?(.25)*[grey][|(3)]\dir{<}?(.5)*{\bigb{\nu'}};
 (8,-8)*{};(-8,-8)*{} **[grey][|(4)]\crv{(8,-16) & (-8,-16)} ?(.25)*[grey][|(3)]\dir{<}?(.5)*{\bigb{\nu}};
 (-8,1)*{};(8,1)*{} **[grey][|(4)]\crv{(-8,9) & (8,9)} ?(.25)*[grey][|(3)]\dir{<}?(.5)*{\bigb{\sigma}};
  (-11,18)*{\scriptstyle b-i};
 (11,18)*{\scriptstyle a-i};
 (-11,-18)*{\scriptstyle b-i};
 (5,4)*{\scriptstyle i};
 (-5,-11)*{\scriptstyle j};
 (-4,-3)*{\scriptstyle i-j};
 (18,-2)*{b^+(s_{\tau})};
 (16,-12)*{n};
 (24,0)*{};
 (-11,0)*{};
\endxy
$$
By acting with $\pi$ on above relations and using assumption of the induction we get
$$
\pi(hg)
=
\pi\left(\xy
 (-8,18);(-8,-18); **[grey][|(4)]\dir{-} ?(.05)*[grey][|(3)]\dir{<}?(.8)*{\bigb{\mu}};
 (8,-18);(8,18); **[grey][|(4)]\dir{-} ?(.05)*[grey][|(3)]\dir{<}?(.2)*{\bigb{\lambda}};
 (8,0)*{};(-8,0)*{} **[grey][|(4)]\crv{(8,-8) & (-8,-8)} ?(.25)*[grey][|(3)]\dir{<}?(.5)*{\bigb{\nu'}};
 (8,-8)*{};(-8,-8)*{} **[grey][|(4)]\crv{(8,-16) & (-8,-16)} ?(.25)*[grey][|(3)]\dir{<}?(.5)*{\bigb{\nu}};
 (-8,1)*{};(8,1)*{} **[grey][|(4)]\crv{(-8,9) & (8,9)} ?(.25)*[grey][|(3)]\dir{<}?(.5)*{\bigb{\sigma}};
  (-11,18)*{\scriptstyle b-i};
 (11,18)*{\scriptstyle a-i};
 (-11,-18)*{\scriptstyle b-i};
 (5,4)*{\scriptstyle i};
 (-5,-11)*{\scriptstyle j};
 (-4,-3)*{\scriptstyle i-j};
 (18,-2)*{b^+(s_{\tau})};
 (16,-12)*{n};
 (24,0)*{};
 (-11,0)*{};
\endxy\right)
=
\pi\left(\xy
 (-8,18);(-8,-18); **[grey][|(4)]\dir{-} ?(.05)*[grey][|(3)]\dir{<}?(.6)*{\bigb{\mu}};
 (8,-18);(8,18); **[grey][|(4)]\dir{-} ?(.05)*[grey][|(3)]\dir{<}?(.4)*{\bigb{\lambda}};
 (8,-7)*{};(-8,-7)*{} **[grey][|(4)]\crv{(8,-15) & (-8,-15)} ?(.25)*[grey][|(3)]\dir{<}?(.5)*{\bigb{\nu'}};
 (8,14)*{};(-8,14)*{} **[grey][|(4)]\crv{(8,6) & (-8,6)} ?(.25)*[grey][|(3)]\dir{<}?(.5)*{\bigb{\nu}};
 (-8,-6)*{};(8,-6)*{} **[grey][|(4)]\crv{(-8,2) & (8,2)} ?(.25)*[grey][|(3)]\dir{<}?(.5)*{\bigb{\sigma}};
 (-13,18)*{\scriptstyle b-i+j};
 (13,18)*{\scriptstyle a-i+j};
 (-13,-18)*{\scriptstyle b-i+j};
 (13,-18)*{\scriptstyle a-i+j};
 (5,-3)*{\scriptstyle i};
 (-5,11)*{\scriptstyle j};
 (-4,-10)*{\scriptstyle i-j};
 (18,-2)*{b^+(s_{\tau})};
 (16,-12)*{n};
 (24,0)*{};
 (-11,0)*{};
\endxy\right)
=
\pi(gh).
$$

Case (3) is similar.
\end{proof}

The degree of $\rF^{(b)}_\lambda b^+(\tau) \rE^{(a)}_\mu  \onen$ is $2(|\lambda |+|\mu |+|\tau |)\geq 0$. So we have the following.

\begin{cor}
  \label{r4}
The negative degree part of $\Tr\U ^*$ is zero.
For any positive integer $i$, the $2i$-degree part of $\Tr\dotU^*(n,m)$ is freely generated by $\rF^{(a)}_\lambda b^+(\tau) \rE^{(b)}_\mu \onen$ for $n+m\geq 0$ and by
$\rE^{(b)}_\mu b^+(\tau) \rF^{(a)}_\lambda \onen$ for $n+m\leq 0$. In both cases $2(a-b)=m-n$ and $|\lambda |+|\mu |+|\tau |=i$.
In particular, the degree zero part of $\Tr\dotU^*$ coincides with $K_0(\dotU^*)$.
\end{cor}

\subsection{Wedge product of symmetric polynomials}
\label{sec:schur-polyn-non}

As a preparation for the proof of Proposition \ref{zzz} below, we need the
following extension of the Schur polynomials to non-partition
sequences.

For $a\ge0$, set
\begin{gather*}
  \tP(a)=\{ (\tm_1,\dots,\tm_a)\in\Z^a\;|\;\tm_j\ge j-a\}.
\end{gather*}
Note that $P(a)\subset \tP(a)$.  As mentioned in Section
\ref{subsec_sympoly}, the definition \eqref{def_schur} of the Schur
polynomial extends to sequences in $\tP(a)$ as follows.  For
$\tm=(\tm_1,\dots,\tm_a)\in \tP(a)$, set
\begin{equation}\label{e6}
  s_{\tm}:=\frac{a_{\tm+\delta}}{a_{\delta}}=
\frac{\det(x_i^{\tm_j+a-j})_{1\leq i,j\leq a}}{\det(x_i^{a-j})_{1\leq i,j\leq a}}.
\end{equation}

For $k=1,\dots,a-1$, we have
\begin{gather}
  \label{e7}
  s_{(\tm_1,\dots,\tm_a)}=
    -s_{(\tm_1,\dots,\tm_{k-1},\tm_{k+1}-1,\tm_{k}+1,\tm_{k+2},\dots,\tm_a)}.
\end{gather}
Note that the sequence
$(\tm_1,\dots,\tm_{k-1},\tm_{k+1}-1,\tm_{k}+1,\tm_{k+2},\dots,\tm_a)$
of the right hand side is obtained from $\tm$ by permuting the $k$th
and $k+1$st entries and adding $\pm1$ to them.  Using \eqref{e7}, it
is easily checked that for every $\tm\in P(a)$ we have either
$s_{\tm}=0$ or $s_{\tm}=\pm s_{\mu}$ for some uniquely determined
$\mu\in P(a)$.  Consequently, $s_{\tm}\in \Sym_a$.

For $\tl\in \Z^a$, $\tm\in \Z^b$, set
\begin{gather*}
  \tl\cup\tm=(\tl_1,\dots,\tl_a,\tm_1,\dots,\tm_b) \in \Z^{a+b}.
\end{gather*}
Note that we have $\tl\cup\tm\in P(a+b)$ if and  only if $\tl\in
P(a)$, $\tm\in P(b)$ and $\tl_a\ge\tm_1$.

Define the {\em wedge product} of symmetric polynomials
\begin{gather*}
  \wedge_{a,b}\col \Sym_a\times \Sym_b \rightarrow \Sym_{a+b}
\end{gather*}
by
\begin{gather*}
  \wedge_{a,b}(s_\lambda,s_\mu)= s_{(\lambda-b)\cup\mu}
\end{gather*}
for $\lambda\in P(a)$, $\mu\in P(b)$, where $\lambda-b=(\lambda_1-b,\dots,\lambda_a-b)\in\tP(a)$.  Since
$(\lambda-b)\cup\mu\in\tP(a+b)$, it follows that
$s_{\lambda\cup\mu}\in \Sym_{a+b}$ is well defined.
The maps $\wedge_{a,b}$, $a,b\ge0$, form an algebra structure on
the graded $\Z$-module $\bigoplus_{a\ge0}\Sym_a$, which is isomorphic to the exterior
algebra $\bigwedge \Sym_1\cong\bigwedge\Z[h_1]$.

We have (\cite[Proposition 2.9]{KLMS})
\begin{gather*}
  \xy
 (0,11);(0,-11); **[grey][|(4)]\dir{-} ?(1)*[grey][|(3)]\dir{>};
 (5,6)*{n};
 (3,-10)*{\scriptstyle {a+b}};
 (11,0)*{};
 (-3,0)*{};
(0,0)*{\bigb{\wedge_{a,b}(x,y)}};
\endxy
=
\text{$\xy
 (0,6);(0,11) **[grey][|(4)]\dir{-}?(1)*[grey][|(3)]\dir{>};
  (0,-6);(0,-11) **[grey][|(4)]\dir{-};
   (0,-6);(0,6)*{} **[grey][|(4)]\crv{(-5,-5) & (-5,5)}?(.5)*{{\bigb{x}}};
   (0,-6);(0,6)*{} **[grey][|(4)]\crv{(5,-5) & (5,5)}?(.5)*{\bigb{y}};
  (6,7)*{n};
  (3,-11)*{\scriptstyle a+b};
  (-3,-6)*{\scriptstyle a};
  (3,-6)*{\scriptstyle b};
  (10,0)*{};
  (-8,0)*{};
\endxy$}
\end{gather*}
for $x\in \Sym_a,y\in \Sym_b$.

\subsection{Structure of the category $(\mathrm{Tr}\,\dotU^*)^+$}
\label{sec:struct-categ-mathrmt}

Let $\Tr(\dotU^*)^+$ denote the linear subcategory of $\Tr(\dotU^*)$ such
that $\Ob(\Tr(\dotU^*)^+)=\Z$ and the morphisms are generated by composites of
$\rE^{(a)}_\lambda\onen$
for $n\in\Z$, $a\ge0$, $\lambda\in P(a)$.
Similarly, let $\Tr(\dotU^*)^-$ denote the linear subcategory of $\Tr(\dotU^*)$ such
that $\Ob(\Tr(\dotU^*)^-)=\Z$ and the morphisms are generated by
$\rF^{(a)}_\lambda\onen$
for $n\in\Z$, $a\ge0$, $\lambda\in P(a)$.

We define $\Tr(\dotU^*)^0$ as the linear subcategory of $\Tr(\dotU^*)$
with $\Ob(\Tr(\dotU^*)^0)=\Z$ and the morphisms are generated by
$b^+(\tau)\onen$, $n\in\modZ$, $\tau\in P$.  It is easy to check that
$\Tr(\dotU^*)^0(n,m)=0$ for $n\neq m$, and that $\Tr(\dotU^*)^0(n,n)$ has basis given by $b^+(\tau)\onen$ for $\tau\in P$.

For $x\in \Sym_a$, set
\begin{gather*}
  \rE^{(a)}_x\onen=\rE^{(a)}(x)\onen:=[\E^{(a)}_{x}\onen].
\end{gather*}
The following lemma gives the composition rule in $\Tr((\dotU^*)^+)$.
\begin{lem} \label{lem-compUp}
  \label{a1}
  For $x\in \Sym_a$, $y\in \Sym_b$, we have
  \begin{gather*}
    \rE^{(a)}_x \rE^{(b)}_y \onen
    =\sum_{\tau\in P(a,b)}
    (-1)^{|\hat\tau|}
    \rE^{(a+b)}\left(\wedge_{a,b} (xs_{\tau}\otimes s_{\hat\tau} y)\right)\onen.
  \end{gather*}
\end{lem}

\begin{proof}
We have
\begin{gather*}
\rE^{(a)}_{x}\rE^{(b)}_{y}\onen
=
\left[\xy
 (8,8);(8,-8); **[grey][|(4)]\dir{-} ?(1)*[grey][|(3)]\dir{>}?(.5)*{\bigb{y}};
 (10,-7)*{\scriptstyle b};
 (0,8);(0,-8); **[grey][|(4)]\dir{-} ?(1)*[grey][|(3)]\dir{>}?(.5)*{\bigb{x}};
 (2,-7)*{\scriptstyle a};
 (12,4)*{n};
 (14,0)*{};
 (-3,0)*{};
\endxy\right]
=\sum_{\tau\in P(a,b)} (-1)^{|\hat\tau|}
\left[\xy
 (0,-5);(0,-1) **[grey][|(4)]\dir{-};
 (0,-1);(-6,15)*{} **[grey][|(4)]\crv{(-6,0)} ?(1)*[grey][|(3)]\dir{>}?(.83)*{\bigb{x}}?(.55)*{\bigb{\tau}};
 (0,-1);(6,15)*{} **[grey][|(4)]\crv{(6,0)} ?(1)*[grey][|(3)]\dir{>}?(.7)*{\bigb{y}};
 (0,-5);(6,-15)*{} **[grey][|(4)]\crv{(6,-6)} ?(.65)*{\bigb{\hat\tau}};
 (0,-5);(-6,-15)*{} **[grey][|(4)]\crv{(-6,-6)};
 (8,0)*{n};
 (-4,-15)*{\scriptstyle a};
 (8,-15)*{\scriptstyle b};
 (11,0)*{};
 (-11,0)*{};
\endxy\right]
\end{gather*}
\begin{gather*}
=\sum_{\tau\in P(a,b)} (-1)^{|\hat\tau|}
\left[\text{$\xy
 (0,-10);(0,10)*{} **[grey][|(4)]\crv{(-7,-9) & (-7,9)} ?(.3)*{\bigb{\tau}}?(.5)*{\bigb{x}};
 (0,10);(0,15) **[grey][|(4)]\dir{-}?(1)*[grey][|(3)]\dir{>};
 (0,-10);(0,-15) **[grey][|(4)]\dir{-};
 (0,-10);(0,10)*{} **[grey][|(4)]\crv{(7,-9)  & (7,9) }?(.5)*{\bigb{y}}?(.7)*{\bigb{\hat\tau}};
 (6,12)*{n};
 (3,-15)*{\scriptstyle a+b};
 (-3,-10)*{\scriptstyle a};
 (3,-10)*{\scriptstyle b};
 (11,0)*{};
 (-11,0)*{};
\endxy$}\right]
=\sum_{\tau\in P(a,b)}
    (-1)^{|\hat\tau|}
    \rE^{(a+b)}\left(\wedge_{a,b} (xs_{\tau}\otimes s_{\hat\tau} y)\right)\onen.
\end{gather*}
\end{proof}

Proposition \ref{triangular} and Lemma~\ref{lem-compUp} imply the following.

\begin{prop}
  \label{aaaa}
Let $n,m\in\Z$, $a\ge0$, and $m-n=2a$. Then $\Tr(\dotU^*)^+(n,m)$ is a free abelian group with basis
\begin{gather*}
\{\rE^{(a)}_\lambda\onen\;|\;\lambda\in
  P(a)\}.
  \end{gather*}
\end{prop}

Our goal in this section is to obtain a presentation for $\Tr(\dotU^*)^+$ and bases for $\Tr(\dotU^*)^+(n,m)$ that facilitate a direct comparison with $\UZ^+\LL$ and its basis given in Proposition~\ref{r15}.  To do this we will need to relate the basis for $\Tr(\dotU^*)^+(n,m)$ given in Proposition~\ref{aaaa} to one given by composites of the form
\begin{gather*}
  \rE^{(a_1)}_{l_1^{a_1}} \rE^{(a_2)}_{l_2^{a_2}} \dots \rE^{(a_p)}_{l_p^{a_p}} \onen,
\end{gather*}
where $a_1+a_2+\dots +a_p=a$, $l_1>l_2>\dots> l_p\geq 0$, $a_i\geq 1$, $p\geq 0$.

In what follows we utilize the lexicographic order on partitions, where $\lambda > \mu$ implies $\lambda_1 > \mu_1$ or else $\lambda_i = \mu_i$ for $1 \leq i \leq k$ and $\lambda_{k+1}> \mu_{k+1}$ for some $k \geq 1$.   The lemma below gives the leading term in this order for the change of basis from  $\rE_{j^a}^{(a)}\rE^{(b)}_{\lambda}\on$ to $\rE^{(a+b)}_{\tau}\on$.

\begin{lem}\label{ja}
  For  $\lambda\in P(b)$, $j>\lambda_1$, we have
  \begin{gather*}
    \rE_{j^a}^{(a)}\rE^{(b)}_{\lambda}\on-\rE^{(a+b)}_{j^a\cup\lambda}\on \in
    \Span_\modZ \{\rE^{(a+b)}_\tau\on \;|\; \tau \in P(a+b),\tau <j^a\cup\lambda\}.
  \end{gather*}
\end{lem}

\begin{proof}
By Lemma \ref{a1}, we have
\begin{align*}
  \rE^{(a)}_{j^a}\rE^{(b)}_{\lambda}\on
  &=\sum_{\tau\in P(a,b)}(-1)^{\hat{\tau}}\rE^{(a+b)}(\wedge_{a,b}(s_{j^a}s_\tau\otimes s_{\hat{\tau}}s_\lambda ))\on\\
  &=\sum_{\tau\in P(a,b)}(-1)^{\hat{\tau}}\rE^{(a+b)}(\wedge_{a,b}(s_{\tau+j}\otimes \sum_{\nu\in P(b)}N^\nu_{\hat{\tau},\lambda}s_{\nu} ))\on\\
  &=\rE^{(a+b)}\left(\sum_{\tau\in P(a,b)}\sum_{\nu\in P(b)}
  (-1)^{\hat{\tau}}N^\nu_{\hat{\tau},\lambda} \wedge_{a,b}(s_{\tau+j}\otimes s_{\nu} )\right)\on\\
  &=\rE^{(a+b)}\left(\sum_{\tau\in P(a,b)}\sum_{\nu\in P(b)}
  (-1)^{\hat{\tau}}N^\nu_{\hat{\tau},\lambda}
  s_{(\tau_1+j-b,\dots,\tau_a+j-b,\nu_1,\dots,\nu_b)}\right)\on,
\end{align*}
where $\tau+j=(\tau_1+j,\dots,\tau_a+j)\in P(a)$.
Note that the term for $\tau=b^a$ in the above sum is exactly $s_{j^a\cup\lambda}$.
One can check that the other terms are contained in $\Span_\modZ \{s_\tau \;|\; \tau \in P(a+b),\tau <j^a\cup\lambda\}$.
\end{proof}

\begin{prop} \label{zzz}
The linear category $(\Tr\dotU^*)^+$ has the following presentation.
  \begin{itemize}
  \item
    Objects are integers $n\in \Z$.
  \item
    Morphisms are generated by $\rE^{(a)}_{l^a} \onen \in (\Tr\dotU^*)(n,n+2a)$
    for $a,l\geq 0$, $n\in\Z$.
  \item
    The morphisms satisfy the following relations:
    \begin{align}
      \label{e10}
      \rE^{(a)}_{l^a}\rE^{(b)}_{s^b} \onen &=
      \rE^{(b)}_{s^b} \rE^{(a)}_{l^a} \onen\quad \text{for $n\in\Z$, $a,b,l,s\ge0$},\\
      \label{e11}
      \rE^{(0)}_{l^0} \onen & = \onen \quad \text{for $n\in\Z$, $l\ge0$,} \\
      \label{e12}
      \rE^{(a)}_{l^a}\rE^{(b)}_{l^b} \onen &= \bins{a+b}{a} \rE^{(a+b)}_{l^{a+b}} \onen \quad \text{for $n\in\Z$, $a,b,l\ge0$}.
    \end{align}
  \end{itemize}
\end{prop}

\begin{proof}

First we prove that $\Tr(\dotU^*)^+$ is spanned by composites of the morphisms $\rE^{(a)}_{l^a}\on$, $n\in\Z$, $a,l\ge0$.  Suppose we are given
$\rE^{(a)}_\lambda\onen$, $n\in Z$, $a\ge0$, $\lambda\in P(a)$.
We will show by induction on $a$ and $\lambda$ (in the lexicographic
order) that it is a linear combination of composites of morphisms of
the form $\rE^{(i)}_{l^i}\onen$ with $i,l\ge0$ and $n\in\Z$.  If $a=0$, then there is
nothing to prove.  Suppose $a>0$.  If $\lambda=(\lambda_1)^a$, then
we are done. Otherwise, let $k\ge1$ be such that
$\lambda_1=\dots=\lambda_{k}>\lambda_{k+1}$.  Then, by Lemma~\ref{ja}, we have
    \begin{gather*}
      \rE^{(a)}_\lambda\on-\rE^{(k)}_{\lambda_1^k}\rE^{(a-k)}_{(\lambda_{k+1},\dots,\lambda_a)}\on
       \in
      \Span_\modZ \{\rE^{(a)}_\tau\on \;|\; \tau \in P(a),\tau <\lambda\}.
    \end{gather*}
By the induction hypothesis, it follows that $\rE^{(a)}_\lambda\on$ is as
desired.

Now we prove the relations given in the statement.
The relation \eqref{e11} is obvious.  Relation \eqref{e12} is proven using Lemma \ref{a1},
  \begin{align*}
    \rE^{(a)}_{l^a}\rE^{(b)}_{l^b}\on
    &=\sum_{\tau\in P(a,b)}(-1)^{\hat{\tau}}\rE^{(a+b)}(\wedge_{a,b}(s_{l^a}s_\tau\otimes s_{\hat{\tau}}s_{l^b} ))\on\\
    &=\sum_{\tau\in P(a,b)}(-1)^{\hat{\tau}}\rE^{(a+b)}(\wedge_{a,b}(s_{\tau+l}\otimes s_{\hat{\tau}+l} ))\on\\
    &=\rE^{(a+b)}\left(\sum_{\tau\in P(a,b)}(-1)^{\hat{\tau}}(s_{(\tau+l)\cup (\hat{\tau}+l)})\right)\on.
  \end{align*}
  Since we have
  \begin{gather*}
    s_{(\tau+l)\cup (\hat{\tau}+l)}=(-1)^{\hat{\tau}}s_{l^{a+b}}
  \end{gather*}
  for $\tau\in P(a,b)$ and we have $|P(a,b)|=\binom{a+b}{a}$ relation
  \eqref{e12} follows.
\end{proof}

\begin{cor}
\label{cor-traceUp}
As a $\Z$-module, $(\Tr\dotU^*)^+(n,n+2a)$ has a basis given by
\begin{gather}
  \label{e13}
  \rE^{(a_1)}_{l_1^{a_1}} \rE^{(a_2)}_{l_2^{a_2}} \dots \rE^{(a_p)}_{l_p^{a_p}} \onen,
\end{gather}
where $a_1+a_2+\dots +a_p=a$, $l_1>l_2>\dots> l_p\geq 0$, $a_i\geq 1$, $p\geq 0$.
\end{cor}

\begin{proof}
Using the relation \ref{e12} it is clear that $(\Tr\dotU^*)^+(n,n+2a)$ is spanned by the
elements given in \eqref{e13}.  To see that these elements are linearly
independent, we use Lemma \ref{ja} repeatedly to deduce that
  \begin{gather*}
    \rE^{(a_1)}_{l_1^{a_1}} \rE^{(a_2)}_{l_2^{a_2}} \dots
    \rE^{(a_p)}_{l_p^{a_p}} \onen
    -\rE^{(a)}_{(l_1^{a_1},l_2^{a_2},\dots,l_p^{a_p})}\onen
    \in\Span\{\rE^{(a)}_{\tau} \onen \;|\;\tau\in P(a),\tau< (l_1^{a_1},l_2^{a_2},\dots,l_p^{a_p})\},
  \end{gather*}
so that the change of basis between elements in \eqref{e13} and those in Proposition~\ref{aaaa} is upper triangular.
\end{proof}

\subsection{Commutation relations in $\mathrm{Tr}\,\U^*$}
Before the main result we need the following two lemmas.

\begin{lem} The equation
\label{ap2}
$$b^-(p_m)\E \onen-\E b^-(p_m)\onen = 2\E_m \onen$$
holds in the 2-category $\U.$
\end{lem}
\begin{proof} The proof is by direct computation using the relations in $\U$.
\begin{align*}
b^-(p_m)\E\onen & =\sum_{l=0}^ml\xy
 (-7,0)*{\rbub};
 (-1,-2)*{\scriptstyle m-l};
 (7,0)*{\lbub};
 (11,-2)*{\scriptstyle l};
 {\ar (14,-5)*{}; (14,5)*{}};
 (18,2)*{n};
 (-12,0)*{}; (21,0)*{};
\endxy  \\
& = \sum_{l=0}^m\sum_{i=0}^ll(l+1-i)\xy
 (-7,0)*{\rbub};
 (-1,-2)*{\scriptstyle m-l};
 (12,0)*{\lbub};
 (16,-2)*{\scriptstyle i};
 {\ar (3,-5)*{}; (3,5)*{}};
 (3,0)*{\bullet};
 (6,1)*{\scriptstyle l-i};
 (18,2)*{n};
 (-12,0)*{}; (21,0)*{};
\endxy
 \\
& = \sum_{i=0}^m\sum_{l=i}^ml(l+1-i)\xy
 (-7,0)*{\rbub};
 (-1,-2)*{\scriptstyle m-l};
 (12,0)*{\lbub};
 (16,-2)*{\scriptstyle i};
 {\ar (3,-5)*{}; (3,5)*{}};
 (3,0)*{\bullet};
 (6,1)*{\scriptstyle l-i};
 (18,2)*{n};
 (-12,0)*{}; (21,0)*{};
\endxy
\\
& = \sum_{i=0}^m\left\{\sum_{l=i}^ml(l+1-i)\xy
 (-1,0)*{\rbub};
 (5,-2)*{\scriptstyle m-l};
 (12,0)*{\lbub};
 (16,-2)*{\scriptstyle i};
 {\ar (-11,-5)*{}; (-11,5)*{}};
 (-11,0)*{\bullet};
 (-8,1)*{\scriptstyle l-i};
 (18,2)*{n};
 (-13,0)*{}; (21,0)*{};
\endxy
  \; -2 \sum_{l=i}^ml(l+1-i)\xy
 (3,0)*{\rbub};
 (11,-2)*{\scriptstyle m-l-1};
 (20,0)*{\lbub};
 (24,-2)*{\scriptstyle i};
 {\ar (-11,-5)*{}; (-11,5)*{}};
 (-11,0)*{\bullet};
 (-6,1)*{\scriptstyle l-i+1};
 (26,2)*{n};
 (-13,0)*{}; (21,0)*{};
\endxy \right.\\
& \left. \quad\quad\quad + \sum_{l=i}^ml(l+1-i)\xy
 (3,0)*{\rbub};
 (11,-2)*{\scriptstyle m-l-2};
 (20,0)*{\lbub};
 (24,-2)*{\scriptstyle i};
 {\ar (-11,-5)*{}; (-11,5)*{}};
 (-11,0)*{\bullet};
 (-6,1)*{\scriptstyle l-i+2};
 (26,2)*{n};
 (-13,0)*{}; (21,0)*{};
\endxy \right\}
 \\
& = \sum_{i=0}^m\left\{\sum_{l=i}^ml(l+1-i)\xy
 (-1,0)*{\rbub};
 (5,-2)*{\scriptstyle m-l};
 (12,0)*{\lbub};
 (16,-2)*{\scriptstyle i};
 {\ar (-11,-5)*{}; (-11,5)*{}};
 (-11,0)*{\bullet};
 (-8,1)*{\scriptstyle l-i};
 (18,2)*{n};
 (-13,0)*{}; (21,0)*{};
\endxy \; -2 \sum_{l=i+1}^{m+1}(l-1)(l-i)\xy
 (-1,0)*{\rbub};
 (5,-2)*{\scriptstyle m-l};
 (12,0)*{\lbub};
 (16,-2)*{\scriptstyle i};
 {\ar (-11,-5)*{}; (-11,5)*{}};
 (-11,0)*{\bullet};
 (-8,1)*{\scriptstyle l-i};
 (18,2)*{n};
 (-13,0)*{}; (21,0)*{};
\endxy \right.\\
& \left. \quad\quad\quad + \sum_{l=i+2}^{m+2}(l-2)(l-1-i)\xy
 (-1,0)*{\rbub};
 (5,-2)*{\scriptstyle m-l};
 (12,0)*{\lbub};
 (16,-2)*{\scriptstyle i};
 {\ar (-11,-5)*{}; (-11,5)*{}};
 (-11,0)*{\bullet};
 (-8,1)*{\scriptstyle l-i};
 (18,2)*{n};
 (-13,0)*{}; (21,0)*{};
\endxy \right\}
\\
& = \sum_{i=0}^m\left\{i\xy
 (-1,0)*{\rbub};
 (5,-2)*{\scriptstyle m-i};
 (12,0)*{\lbub};
 (16,-2)*{\scriptstyle i};
 {\ar (-7,-5)*{}; (-7,5)*{}};
 (18,2)*{n};
 (-10,0)*{}; (21,0)*{};
\endxy + \sum_{l=i+1}^m 2\; \xy
 (-1,0)*{\rbub};
 (5,-2)*{\scriptstyle m-l};
 (12,0)*{\lbub};
 (16,-2)*{\scriptstyle i};
 {\ar (-11,-5)*{}; (-11,5)*{}};
 (-11,0)*{\bullet};
 (-8,1)*{\scriptstyle l-i};
 (18,2)*{n};
 (-13,0)*{}; (21,0)*{};
\endxy \right\}
 \\
& = \E b^-(p_m)\onen+2\sum_{i=0}^m\sum_{l=1}^{m-i} \xy
 (-4,0)*{\rbub};
 (4,-2)*{\scriptstyle m-i-l};
 (13,0)*{\lbub};
 (17,-2)*{\scriptstyle i};
 {\ar (-11,-5)*{}; (-11,5)*{}};
 (-11,0)*{\bullet};
 (-9,1)*{\scriptstyle l};
 (19,2)*{n};
 (-13,0)*{}; (22,0)*{};
\endxy
 \\
& = \E b^-(p_m)\onen+2\sum_{l=1}^m\sum_{i=0}^{m-l} \xy
 (-4,0)*{\rbub};
 (4,-2)*{\scriptstyle m-i-l};
 (13,0)*{\lbub};
 (17,-2)*{\scriptstyle i};
 {\ar (-11,-5)*{}; (-11,5)*{}};
 (-11,0)*{\bullet};
 (-9,1)*{\scriptstyle l};
 (19,2)*{n};
 (-13,0)*{}; (22,0)*{};
\endxy  \\
& = \E b^-(p_m)\onen + 2 \;\xy
 {\ar (0,-5)*{}; (0,5)*{}};
 (0,0)*{\bullet};
 (2,1)*{\scriptstyle m};
 (5,-2)*{n};
 (-2,0)*{}; (8,0)*{};
\endxy
\end{align*}
\end{proof}

\begin{lem}
  \label{r7}
The commutation relation
  \begin{gather*}
    \rE_i\rF_j\onen-\rF_j\rE_i\onen =
    \begin{cases}
      n\onen&\text{if $i+j=0$}\\
      b^-(p_{i+j})\onen&\text{otherwise},
    \end{cases}
  \end{gather*}
  holds in $\mathrm{Tr}\,\U^*$.
\end{lem}

\begin{proof}
Let $m:=i+j$. For $n\geq 0$, the relations
$$
\xy
{\ar (0,-8)*{}; (0,8)*{}};
(0,0)*{\bullet};
(-2,1)*{\scriptstyle i};
(-4,0)*{};(3,0)*{};
\endxy\xy
{\ar (0,8)*{}; (0,-8)*{}};
(0,0)*{\bullet};
(-2,1)*{\scriptstyle j};
(-4,0)*{};(3,0)*{};
\endxy = -\xy
(-3,-8)*{};(-3,8)*{} **\crv{(-3,-4) & (3,-4) & (3,4) & (-3,4)}?(1)*\dir{>}?(.5)*{\bullet};
(3,-8)*{};(3,8)*{} **\crv{(3,-4) & (-3,-4) & (-3,4) & (3,4)}?(0)*\dir{<}?(.5)*{\bullet};
(-4.5,1)*{\scriptstyle j};
(4.5,1)*{\scriptstyle i};
(-7,0)*{};(7,0)*{};
\endxy + \sum_{f_1+f_2+f_3=n-1+m}\xy
 (3,9)*{};(-3,9)*{} **\crv{(3,4) & (-3,4)} ?(1)*\dir{>} ?(.2)*{\bullet};
(5,6)*{\scriptstyle f_1};
 (0,0)*{\rbub};
(5,-2)*{\scriptstyle f_2};
 (-3,-9)*{};(3,-9)*{} **\crv{(-3,-4) & (3,-4)} ?(1)*\dir{>}?(.8)*{\bullet};
(5,-6)*{\scriptstyle f_3};
 (-6,0)*{};(8,0)*{};
\endxy
$$
$$
\xy
{\ar (0,8)*{}; (0,-8)*{}};
(0,0)*{\bullet};
(-2,1)*{\scriptstyle j};
(-4,0)*{};(3,0)*{};
\endxy\xy
{\ar (0,-8)*{}; (0,8)*{}};
(0,0)*{\bullet};
(-2,1)*{\scriptstyle i};
(-4,0)*{};(3,0)*{};
\endxy = -\xy
(-3,-8)*{};(-3,8)*{} **\crv{(-3,-2) & (3,-2) & (3,5) & (-3,5)}?(0)*\dir{<}?(.08)*{\bullet};
(3,-8)*{};(3,8)*{} **\crv{(3,-2) & (-3,-2) & (-3,5) & (3,5)}?(1)*\dir{>}?(.08)*{\bullet};
(-4.5,-4)*{\scriptstyle j};
(4.5,-4)*{\scriptstyle i};
(-7,0)*{};(7,0)*{};
\endxy
$$
imply
$$
 \rE_i\rF_j\onen-\rF_j\rE_i\onen = \sum_{l=0}^m(l+1)\xy
 (-7,0)*{\rbub};
 (-1,-2)*{\scriptstyle m-l};
 (7,0)*{\lbub};
 (11,-2)*{\scriptstyle l};
 (-13,0)*{}; (12,0)*{};
\endxy = \sum_{l=0}^ml\xy
 (-7,0)*{\rbub};
 (-1,-2)*{\scriptstyle m-l};
 (7,0)*{\lbub};
 (11,-2)*{\scriptstyle l};
 (-13,0)*{}; (12,0)*{};
\endxy =
$$
$$
= \sum_{l=0}^m (-1)^{m-l}lb^-(h_l)b^-(\elem_{m-l})\onen = b^-\left[\sum_{l=0}^m(-1)^{m-l}lh_l\elem_{m-l}\right]\onen = $$
$$
=b^-\left[s_m-s_{m-1,1}+s_{m-2,1^2}-\dots +(-1)^{m-1}s_{1^m}\right]\onen = b^-(p_m)\onen,
$$
where we use that $\sum^m_{l=0} (-1)^lh_{m-l}\elem_l=0$ and the identity
 \begin{gather*}
    p_m=s_m-s_{m-1,1}+s_{m-2,1^2}-\dots +(-1)^{m-1}s_{1^m}\in \Sym.
  \end{gather*}
If $n\leq 0$ then the relations
$$
\xy
{\ar (0,-8)*{}; (0,8)*{}};
(0,0)*{\bullet};
(-2,1)*{\scriptstyle i};
(-4,0)*{};(3,0)*{};
\endxy\xy
{\ar (0,8)*{}; (0,-8)*{}};
(0,0)*{\bullet};
(-2,1)*{\scriptstyle j};
(-4,0)*{};(3,0)*{};
\endxy = -\xy
(-3,-8)*{};(-3,8)*{} **\crv{(-3,-2) & (3,-2) & (3,5) & (-3,5)}?(1)*\dir{>}?(.08)*{\bullet};
(3,-8)*{};(3,8)*{} **\crv{(3,-2) & (-3,-2) & (-3,5) & (3,5)}?(0)*\dir{<}?(.08)*{\bullet};
(-4.5,-4)*{\scriptstyle i};
(4.5,-4)*{\scriptstyle j};
(-7,0)*{};(7,0)*{};
\endxy
$$
$$
\xy
{\ar (0,8)*{}; (0,-8)*{}};
(0,0)*{\bullet};
(-2,1)*{\scriptstyle j};
(-4,0)*{};(3,0)*{};
\endxy\xy
{\ar (0,-8)*{}; (0,8)*{}};
(0,0)*{\bullet};
(-2,1)*{\scriptstyle i};
(-4,0)*{};(3,0)*{};
\endxy = -\xy
(-3,-8)*{};(-3,8)*{} **\crv{(-3,-4) & (3,-4) & (3,4) & (-3,4)}?(0)*\dir{<}?(.5)*{\bullet};
(3,-8)*{};(3,8)*{} **\crv{(3,-4) & (-3,-4) & (-3,4) & (3,4)}?(1)*\dir{>}?(.5)*{\bullet};
(-4.5,1)*{\scriptstyle j};
(4.5,1)*{\scriptstyle i};
(-7,0)*{};(7,0)*{};
\endxy + \sum_{f_1+f_2+f_3=-n-1+m}\xy
 (3,9)*{};(-3,9)*{} **\crv{(3,4) & (-3,4)} ?(0)*\dir{<} ?(.2)*{\bullet};
(5,6)*{\scriptstyle f_1};
 (0,0)*{\lbub};
(5,-2)*{\scriptstyle f_2};
 (-3,-9)*{};(3,-9)*{} **\crv{(-3,-4) & (3,-4)} ?(0)*\dir{<}?(.8)*{\bullet};
(5,-6)*{\scriptstyle f_3};
 (-6,0)*{};(8,0)*{};
\endxy
$$
imply
$$
 \rE_i\rF_j\onen-\rF_j\rE_i\onen = -\sum_{l=0}^m(l+1)\xy
 (-7,0)*{\lbub};
 (-1,-2)*{\scriptstyle m-l};
 (7,0)*{\rbub};
 (11,-2)*{\scriptstyle l};
 (-13,0)*{}; (12,0)*{};
\endxy =
$$
$$
= \sum_{l=0}^m(m-l)\xy
 (-7,0)*{\lbub};
 (-1,-2)*{\scriptstyle m-l};
 (7,0)*{\rbub};
 (11,-2)*{\scriptstyle l};
 (-13,0)*{}; (12,0)*{};
\endxy = \sum_{l=0}^ml\xy
 (4,0)*{\rbub};
 (-3,-2)*{\scriptstyle l};
 (-7,0)*{\lbub};
 (10,-2)*{\scriptstyle m-l};
 (-13,0)*{}; (12,0)*{};
\endxy =
$$
$$
= \sum_{l=0}^m (-1)^{m-l}lb^-(h_l)b^-(\elem_{m-l})\onen = b^-\left[\sum_{l=0}^m(-1)^{m-l}lh_l\elem_{m-l}\right] \onen= $$
$$
=b^-\left[s_m-s_{m-1,1}+s_{m-2,1^2}-\dots +(-1)^{m-1}s_{1^m}\right]\onen = b^-(p_m)\onen.
$$
\end{proof}

\subsection{Main result}

\begin{proof}[Proof of Theorem \ref{qqq}]
We need to show that $\Tr\U^*$, as a linear category, is isomorphic
to the idempotented integral form $\dUZL$. Because of Proposition \ref{r1} it is enough to prove that $\Tr\dotU^*$ is isomorphic
to $\dUZL$.
%%\MZ{I have added this sentence.}

Let us define the functor $\mathcal F:\dUL \rightarrow \Tr\dotU^*\otimes\modQ$
to be identity on the objects and to send $E_i 1_n$ and $F_i 1_n$ to $\rE_i \onen$ and $\rF_i \onen$, respectively.
Moreover, let
$$H_i1_n\mapsto \left\{\begin{array}{ll}  n\onen & i=0,\\
b^-(p_i)\onen& i>0.\end{array}\right.$$

Let us check the relations on target, i.e. we need to show that
the following relations hold true:
\begin{align*}
[b^-(p_i),b^-(p_j)]&=0\\
[\rE_i \onen,\rE_j \onen]&=[\rF_i \onen,\rF_j \onen]=0\\
[b^-(p_i),\rE_j \onen]&=2\rE_{i+j}\\
 [b^-(p_i),\rF_j \onen]&=-2\rF_{i+j}\\
[\rE_i, \rF_j]&=\left\{ \begin{array}{ll}
n\onen & \text{if}\quad i+j=0,\\ b^-(p_{i+j}\onen) &\text{otherwise}.
\end{array}
\right.
\end{align*}

The first relation is obvious, the second follows from Proposition
\ref{zzz}. The next one follows from Lemma \ref{ap2}.  % in \ref{sec:proofs}.
The last relation is shown in Lemma \ref{r7}.  % in \ref{sec:proofs}.
So, the functor $\mathcal F$ is well defined.

It is clear that $\mathcal F$ sends
$$
E^{(a)}_j \mapsto \rE^{(a)}_{j^a}, \quad F^{(a)}_j \mapsto \rF^{(a)}_{j^a}, \quad \phi(s_\tau)\mapsto b^+(\tau).
$$

Proposition \ref{triangular} and Corollary~\ref{cor-traceUp} give
the basis of $\Tr\dotU^*(n,m)$ for $n+m\geq 0$:
\begin{gather*}
  \rF_{i_1^{a_1}}^{(a_1)}\dots \rF_{i_r^{a_r}}^{(a_r)}b^+(\tau)\rE_{k_1^{c_1}}^{(c_1)}\dots \rE_{k_t^{c_t}}^{(c_t)}\onen,
\end{gather*}
  \begin{gather*}
    i_1>\dots >i_r\ge 0,\quad  r\ge 0,\quad  a_1,\ldots ,a_r\ge 1,\\
    \tau\in P,\\
    k_1>\dots >k_t\ge 0,\quad  t\ge 0,\quad  c_1,\ldots ,c_t\ge 1,\\
    c_1+\dots +c_t-(a_1+\dots +a_r)=2(n-m),
  \end{gather*}
and similar for $m+n\leq 0$.

We see that $\mathcal F$ sends basis
elements \eqref{e3} of $\dUZL(n,m)$  to the
basis elements of $\Tr\dotU^*(n,m)$, both for $n+m\geq 0$ and $n+m\leq 0$.
So it is an isomorphism.
Therefore, the restriction $\mathcal F|_{\dUZL}:
 \dUZL\rightarrow \Tr\dotU^*$ is also an isomorphism.
\end{proof}

By using Theorem \ref{qqq},
 we can  get rid of two cases in triangular decomposition \ref{triangular} and have the following corollary.

\begin{cor}[Triangular Decomposition]
For every $n,m\in\modZ$, $\Tr\dotU^*(n,m)$ is free with basis
 $$\rF^{(b)}_{\mu}b^+(\tau)\rE^{(a)}_{\lambda}\onen
\quad {\text {for}} \quad  n\in\Z, \quad a,b\ge0, \quad 2(a-b)=m-n,
\quad \lambda\in P(a), \quad\mu\in P(b), \quad\tau\in P.$$
\end{cor}

\addcontentsline{toc}{section}{References}

% ==============================================================================
% REFERENCES

%\bibliographystyle{plain}
%\bibliography{bib-trace}

\begin{thebibliography}{10}

\bibitem{BLM}
A.~Beilinson, G.~Lusztig, and R.~MacPherson.
\newblock A geometric setting for the quantum deformation of {${\rm GL}\sb n$}.
\newblock {\em Duke Math. J.}, 61(2):655--677, 1990.

\bibitem{Benabou}
Jean B{\'e}nabou.
\newblock Introduction to bicategories.
\newblock In {\em Reports of the Midwest Category Seminar}, pages 1--77.
  Springer, Berlin, 1967.

\bibitem{Bor}
F.~Borceux.
\newblock {\em Handbook of categorical algebra. 1}, volume~50 of {\em
  Encyclopedia of Mathematics and its Applications}.
\newblock Cambridge University Press, Cambridge, 1994.

\bibitem{EY}
D.~Evans and Y.~Kawahigashi.
\newblock On {O}cneanu's theory of asymptotic inclusions for subfactors,
  topological quantum field theories and quantum doubles.
\newblock {\em Internat. J. Math.}, 6(2):205--228, 1995.

\bibitem{GK}
N.~Ganter and M.~Kapranov.
\newblock Representation and character theory in 2-categories.
\newblock {\em Adv. Math.}, 217(5):2268--2300, 2008.

\bibitem{Gar}
H.~Garland.
\newblock The arithmetic theory of loop algebras.
\newblock {\em J. Algebra}, 53(2):480--551, 1978.

\bibitem{KLMS}
M.~Khovanov, A.~Lauda, M.~Mackaay, and M.~{S}to{\v{s}i\'c}.
\newblock Extended graphical calculus for categorified quantum sl(2).
\newblock {\em Memoirs of the AMS}, 219.

\bibitem{Lau1}
A.~D. Lauda.
\newblock A categorification of quantum sl(2).
\newblock {\em Adv. Math.}, 225:3327--3424, 2008.
\newblock math.QA/0803.3652.

\bibitem{Lau3}
A.~D. Lauda.
\newblock An introduction to diagrammatic algebra and categorified quantum
  ${\mathfrak{sl}}_2$.
\newblock {\em Bulletin Inst. Math. Academia Sinica}, 7:165--270, 2012.
\newblock arXiv:1106.2128,.

\bibitem{Loday}
J.L. Loday.
\newblock {\em Cyclic homology}, volume 301 of {\em Grundlehren der
  Mathematischen Wissenschaften [Fundamental Principles of Mathematical
  Sciences]}.
\newblock Second edition.

\bibitem{MacLane}
Saunders Mac~Lane.
\newblock Natural associativity and commutativity.
\newblock {\em Rice Univ. Studies}, 49, 1963.

\bibitem{McD}
I.~G. Macdonald.
\newblock {\em Symmetric functions and {H}all polynomials}.
\newblock The Clarendon Press Oxford University Press, New York, 1979.
\newblock Oxford Mathematical Monographs.

\bibitem{Mitchell}
B.~Mitchell.
\newblock Rings with several objects.
\newblock {\em Advances in Math.}, 8:1--161, 1972.

\bibitem{Ocn}
A.~Ocneanu.
\newblock Chirality for operator algebras.
\newblock In {\em Subfactors ({K}yuzeso, 1993)}, pages 39--63. World Sci.
  Publ., River Edge, NJ, 1994.

\end{thebibliography}

%
% ==============================================================================

\end{document}